\documentclass[11pt]{amsart}
\usepackage{amsmath}
\usepackage{amssymb}
\usepackage{amsfonts}
\usepackage{multirow}
\usepackage{epsfig}
\usepackage[all]{xy}
\newcommand{\al}{\alpha}

\newcommand{\B}{\mathcal{B}}

\newcommand{\N}{\mathbb{N}}
\newcommand{\C}{\mathcal{C}}
\newcommand{\CC}{\mathbb{C}}
\newcommand{\ZZ}{\mathbb{Z}}
\newcommand{\RR}{\mathbb{R}}

\newcommand{\A}{\mathcal{A}}
\newcommand{\LL}{\mathcal{L}}

\newcommand{\bound}{\text{bound}}

\newcommand{\pp}{\mathbb{P}}  % projective space

\newtheorem{thm}{Theorem}[section]
\newtheorem*{thm*}{Theorem}
\newtheorem{corollary}[thm]{Corollary}
\newtheorem{lemma}[thm]{Lemma}

\newtheorem*{lemma*}{Lemma}
\newtheorem{definition}[thm]{Definition}
\newtheorem{example}[thm]{Example}
\newtheorem{prs}[thm]{Proposition}
\newtheorem{remark}[thm]{Remark}

\newtheorem{question}[thm]{Question}

\newtheorem*{notation*}{Notation}

\begin{document}
\title{On Left regular bands and real Conic-Line arrangements}

\author{Michael Friedman and David Garber}

\address{Michael Friedman, Institut Fourier, 100 rue des maths, BP 74, 38402 St Martin d'H\'eres cedex, France and   Image Knowledge Gestaltung. An Interdisciplinary Laboratory,	
  Cluster of Excellence, Humboldt-Universität zu Berlin
  Unter den Linden 6, 10099 Berlin}
\email{michael.friedman@hu-berlin.de}

\address{David Garber, Department of Applied Mathematics, Faculty of Sciences, Holon Institute of Technology, Golomb 52,
PO Box 305, Holon 58102, Israel and (Sabbatical:) Einstein Institute of Mathematics, Hebrew University of Jerusalem, Jerusalem, Israel}
\email{garber@hit.ac.il}

\begin{abstract}
An arrangement of curves in the real plane divides it into a collection of faces.   In the case of line arrangements, there exists an associative product which gives this collection a structure of a left regular band. A natural question is whether the same is possible for other arrangements. In this paper, we try to answer this question for the simplest generalization of line arrangements, that is, conic--line arrangements.

Investigating the different algebraic structures induced by the face poset of a conic--line arrangement, we present two different generalizations for the product and its associated structures: an alternative left regular band and an associative aperiodic semigroup. We also study the structure of sub left regular bands induced by these arrangements. We finish with some chamber counting results for conic--line arrangements.
\end{abstract}

\date{\today}

\maketitle

\section{Introduction}

An arrangement of curves $\A$ in $\RR^2$ induces a partition of the plane into a collection of faces, denoted by $\LL(\A)$. For a line arrangement $\A$, the set  $\LL(\A)$ already gives rise to a variety of interesting questions, lying in the intersection of several mathematical areas: algebra, geometry, topology and combinatorics. For example, one can define
a product on $\LL(\A)$, making this set of faces into a left regular band (LRB), that is, a semigroup whose every element is an idempotent and for every two elements $x,y$ in it the following property is satisfied: $x\cdot y \cdot  x = x \cdot y$ (see
\cite{B,B2} and \cite[Section 3]{MSS} for surveys on bands and examples of left regular bands).

Moreover, as $\LL(\A)$ determines the combinatorics of the line arrangement $\A$, one can ask what are the connections between $\LL(\A)$ and other invariants of these arrangements. The relations between the face LRB on $\A$ and the combinatorics of $\LL(\A)$ can be found in the numerous restriction-deletion principles:  Zaslavsky's chamber counting formula \cite{Z}, the restriction-deletion formula for the Poincar\'{e} polynomial $\pi(\A,t)$ and  the addition-deletion theorem for the module $D(\A)$ of $\A$-derivations, see e.g. \cite{OT} and Section \ref{secApp} here. Other applications can be found in random walks on the faces of a hyperplane arrangement \cite{B}, in the ongoing investigation of the connections between the fundamental group $\pi_1(\CC^2-\A)$ of the complement of $\A$ and the combinatorics of $\A$ (see e.g. \cite{ATY,ABGBVS,EGT1,EGT2,FG,GTV,GBVS,WY,Ye} and many more), and in the description of the algebra $kS$ in terms of quivers (see \cite{Sal1,Sal2}).

\medskip

A natural question is what happens to these algebraic structures, associated to $\LL(\A)$, when one deals with  arrangements of smooth curves in $\RR^2$; i.e. topologically speaking, when we deal with a real conic--line arrangement $\A$ in $\RR^2$. Can one associate an algebraic structure to $\LL(\A)$ in a way that it preserves some of the algebraic properties that an LRB satisfies? How would this reflect the proposed algebraic structure that a conic--line arrangement has?

This investigation already took place to some extent. Zaslavsky \cite{Z2} generalized the restriction-deletion formula in several directions and a study of the fundamental group $\pi_1(\CC^2-\A)$ for some families of conic-line arrangements has taken place, see e.g. \cite{AT3,AT1,FG,FG2,Tok}. Also, in \cite{ST},
the existence of other restriction-deletion formulas with respect to the module $D(\A)$ of $\A$-derivations for a quasihomogeneous free conic--line arrangement $\A$ was proven.

However, a deeper investigation of the possible algebraic structures of $\LL(\A)$ associated to a conic--line arrangement $\A$ is needed. In this paper, we show that one of these structures is  a natural candidate to be an alternative product. This kind of phenomenon - an alternative product that replaces the associative one - is not unnatural: it appears also when looking at the poset of the faces of a  building (see Tits \cite[Section 3.19]{T}), and more generally, in a {\it projection poset} (see \cite[p. 26]{AM} for its definition). We will see that one of the products we define for the face poset of a conic--line arrangement will indeed be alternative. Therefore, a first step for understanding the connections between the above mentioned structures is the investigation of the algebraic structure of the face poset $\LL(\A)$ associated to a real conic--line arrangement $\A$ and its applications.

\medskip

The purpose of this paper is to study these structures associated to real conic--line arrangements.  Before stating the main results of this paper, we explicitly introduce the notion of a {\it real conic--line arrangement} which will be used in this paper:

\begin{definition}\label{defRealCLarr}
A {\em real conic--line (CL) arrangement} $\A$ is a collection of ellipses, parabolas and lines defined by the equations $\{f_i=0\}$ in $\CC^2$, where $f_i \in \RR[x,y]$. %and every singular point of the arrangement is in $\RR^2$.
Moreover, for every conic $C \in \A$, $C \cap \RR^2$  is not an empty set, neither a point nor a (double) line.
\end{definition}

Note that we do not include hyperbolas in our definition, since the two sheets of a hyperbola $C_0 = \{f_0(x,y)=0\}$ are not connected in the real plane and hence the hyperbola divides the real plane $\RR^2$ into {\em three} different regions, in contrast to the situation that for each point $(x_1,y_1)\not\in C_0$, we have either $f_0(x_1,y_1)> 0$ or $f_0(x_1,y_1) < 0$, i.e. there are only {\em two} possible options. As the investigation on the poset of faces is based on a one-to-one correspondence between the tiling of the plane induced by the curve and its equation, we leave the case of a hyperbola for a future investigation.

\medskip

We now state the main results of this paper, according to the order of their appearance in the paper.
Section \ref{secFaceSemi} looks for the natural generalization of the structure of $\LL(\A)$ from the case of hyperplane arrangements to the case of real CL arrangements. Based on the problems one encounters during this generalization, we propose two different possibilities for a well-defined product on this set: the first (appears in Definition \ref{defPart2}) turns $\LL(\A)$ into an alternative LRB and the second (appears in Definition \ref{defAssocProd}) turns $\LL(\A)$ into an aperiodic semigroup. In Section \ref{subsec_App}, we present some applications of these generalized products: random walks on this LRB and the possible connection to stereographic projections.

Section \ref{secStrSubLRB} investigates the embedding principles for sub-LRBs for a given LRB, induced by a real CL arrangement. Connections between the band, induced by restricting the real CL arrangement to a conic or to a line, and the band induced by the whole arrangement, are presented. We start by presenting the embedding principle for LRBs induced by line arrangements (Lemma \ref{lemLineArrLRB}). Then, in Section \ref{subsec-3.2}, we define the notion of an LRB of a pointed curve, and based on it, we prove the general embedding principle for LRBs induced by arrangements of smooth real curves (Proposition \ref{prsEmbed}).

The main result of Section \ref{secComb} is a generalization of the restriction-deletion principle for chamber counting in line arrangements (appears in Equation (\ref{eqnResDelHyperplane}) in Section \ref{subsecDelRes}) to the case of CL arrangements (see Proposition \ref{prop-res-del-CL}).

Section \ref{secApp} is an appendix which presents several relevant basic topics in the theory of line arrangements, for the convenience of the reader.

\bigskip

\textbf{Acknowledgements:} We would like to thank Mikhail Zaidenberg, Franco Saliola and especially Stuart Margolis and Benjamin Steinberg for stimulating and inspiring talks. We also thank the anonymous referee of an earlier version of this paper for giving stimulating advices.

The first author would like to thank the Max-Planck-Institute f\"ur Mathematik in Bonn, the Fourier Institut in Grenoble and the Excellence Cluster: {\it Bild Wissen Gestaltung. Ein Interdisziplin\"ares Labor} in Berlin for the warm hospitality and support, where the research of this paper was carried out.

%%%%%%%%%%%%%%%%%%%%%%%%%%%%%%%%%%%%%%%%%%%%%%%%%%%%%%%%%%%%%%%%%%%%%%%%%%%%%%%%%%%%%5

\section[CL arrangements: The face semigroup]{Real CL arrangements: The face semigroup}\label{secFaceSemi}

In this section, we concentrate on the structure of the face semigroup of real CL arrangements. In Section \ref{subsecLRB_Hyp}, we review the corresponding known structure of the face semigroup $\LL(\A)$ associated to a hyperplane arrangement $\A$. In Section \ref{subsec_semiCL}, we study the corresponding face semigroup in the case of a real CL arrangement. The main results of this section appear in Section \ref{subsec_product}, where we introduce two different generalizations for the corresponding product defined for hyperplane arrangements to the case of real CL arrangements: one product turns $\LL(\A)$ into an alternative LRB, while the other turns $\LL(\A)$ into an aperiodic semigroup.
In Section \ref{subsec_App}, we present several applications, emphasizing the differences between CL arrangements and hyperplane arrangements:
Section \ref{subsecRandomWalk} shortly presents random walks on LRBs induced by CL arrangements and Section \ref{subsec_strProj} proposes a possible connection between plane arrangements in $\RR^3$ and CL arrangements in $\RR^2$, showing  that the LRB of the induced CL arrangement might differentiate between central plane arrangements whose LRBs are isomorphic. In Section \ref{subsecLRB_cl}, we present an example of an LRB induced by a CL arrangement, which is not induced by any hyperplane arrangement.

\subsection{Preliminaries: The left regular band and the face semigroup of a hyperplane arrangement} \label{subsecLRB_Hyp}

In this section, we recall the notion of  a {\it left regular band} (LRB) and its connections to the combinatorics of hyperplane arrangements (see also surveys in \cite{B,B2,MSS}).

\medskip

Let $\A = \{ H_1,\dots, H_n \} $ be a hyperplane arrangement (not necessarily central) in $\RR^N$ consisting of $n$ hyperplanes, where $H_i$ is defined by the equation $\{f_i = 0\}$, where $f_i \in \RR[x_1,\ldots,x_N]$. Recall that
for $H \in \A$, the arrangement
$\A^H = \A - \{H\}$ is called the \emph{deleted arrangement} in $\RR^N$ and
 $\A_H = \left\{ K \cap H \ | \ K \in \A^H \right\}$ is called the \emph{restricted arrangement} in $H$. Let
${\mathcal C}(\A)$ be  the set of chambers of $\A$, which are the connected components\footnote{$\ $ Note that when we use the notion ``connected components'' we refer to connected subsets of $\RR^N - \A$, and when we use the notion ``components'' we refer to the elements of the arrangement $\A$.} of  $\RR^N - \A$. Note that these connected components are relatively open sets. Let $L=L(\A)$ be the semi-lattice of non-empty intersections of elements of $\A$.

Define the partially ordered set of faces as:
$$\LL = \LL(\A) = \bigcup_{X \in L} \C(\A_X),$$
where $\LL$ is ordered by inclusion in the closure of the larger face, which will be denoted by $\preceq$,
i.e. $P_1 \preceq P_2$ if $P_1 \subseteq \overline{P_2}$ (some authors order $\LL$ by \emph{reverse} inclusion).

\medskip

Here, we recall the definition of the {\em dimension} of a face (see e.g. \cite{massey}):

\begin{definition}
For any face $P \in \LL$, define the {\em dimension} of a face $P$, denoted by ${\rm dim} (P)$, to be the integer $n$ such that each point in $P$ has an open neighborhood homeomorphic to the open $n$-dimensional ball
$U^n= \left\{ x \in \RR^n \ | \ |x| <1 \right\}$.
\end{definition}

For example, given a line arrangement $\A$ in the plane, the connected components of $\RR^2-\A$ have dimension $2$, the connected segments contained in the components$^1$ of $\A$ have dimension $1$ and the intersection points have dimension $0$.

\begin{remark} \label{rem-comb}\ \\
{\rm (1) Note that for any two faces, a face is not contained in a different face (but may be only contained in its closure). \\
(2) The poset $\LL$ determines the combinatorics of the arrangement $\A$. Explicitly, by the information given by $\LL$, one can  reconstruct the combinatorial data associated to the arrangement $\A$. By \emph{combinatorial data} we mean the number of lines, the number of intersection points and their multiplicity (i.e. how many lines pass through an intersection point), and also which points are contained in each line. Hence, the poset $\LL$ also determines any combinatorial invariant of $\A$, which is determined by this combinatorial data of $\A$.}
\end{remark}

Define a (monomorphic) function $i : \LL \to (\{+,-,0\})^n$, as follows:  for a face $P \in \LL$, define:
$$(i(P))_k = \text{sign}(f_k(P)),$$
where $(i(P))_k$ denotes the value of the $k^{\text{th}}$ coordinate of the vector $i(P)$. The generalization of this function to complex hyperplane arrangements already appeared in \cite{Bj2}, where the vector $i(P)$ is called there {\it a complex sign vector}.

We denote by $\LL_0(\A)$ the image of the function $i$, i.e. ${\rm Image}(i) = \LL_0(\A)$.

\medskip

We now introduce the main algebraic structure used in this paper:
\begin{definition}
A {\em left regular band} \emph{(}{\em LRB}\emph{)} is a semigroup $(S,\cdot)$ that satisfies the identities:
$$ x \cdot x=x \ \ \text{and} \ \ x\cdot y\cdot x=x\cdot y \text{ for every $x,y \in S$.}$$
\end{definition}

It is well-known that one can define an associative product on the set $\{+,-,0\}$,
given by $x\cdot y = x$ if $x \neq 0$, and $y$ otherwise. This product induces an LRB structure on $\{+,-,0\}$, which is denoted  by $L_2^1$. This product can be extended componentwise to a product on $\left( L_2^1 \right)^n$. Thus, ${\rm Image}(i) = \LL_0(\A) \subseteq (\{+,-,0\})^n=\left( L_2^1 \right)^n$ has the structure of an LRB, and therefore also $\LL(\A)$, when we identify it with ${\rm Image}(i)$. Note that $\LL(A) \simeq \LL_0(A)$ since $i$ is monomorphic.

For hyperplane arrangements, this product has been given a geometric meaning: for $F,K\in \LL$, the product  $F \cdot K$ is the face that we are in after moving a small  distance from a generic point of the face
$F$ towards a generic point of the face  $K$ along a straight segment connecting these points (see e.g. \cite{B,BD} and \cite[Section 1.4.6]{AB}).

\begin{remark} \label{remDifSignIsoLRB}
\rm{Note that the embedding $i$ depends on the choice of the functions $\{ f_j \}$: let $\A = \{ H_1,\dots, H_n \}$ be a hyperplane arrangement, where the component $H_i$ is defined by $\{f_i=0\}$. Denote by $i(\LL(\A))$ the embedding of $\LL(\A)$ into $(L^1_2)^n$.
Let $J$ be a nonempty subset of $\{1,\ldots,n\}$ and define $g_j = -f_j$ if $j \in J$ and $g_j =f_j$ otherwise. Let $H'_i \doteq \{g_i=0\}$ and $\A' = \{ H'_1,\dots, H'_n \}$. Obviously, $\A' = \A$. However,
as the LRB structure on $\LL(\A)$ is defined by the sign function, the embedding of $\LL(\A')$ into $(L^1_2)^n$ will be different than the embedding of $\LL(\A)$ (that is, as sets, $i(\LL(\A')) \neq i(\LL(\A))$; explicitly: $(i(\LL(\A')))_j = -(i(\LL(\A)))_j$ for all $j \in J$), but still the two LRBs will be isomorphic.}
\end{remark}

\subsection{The semigroups $\LL $ and $\LL_0$ for CL arrangements}\label{subsec_semiCL}

Let  $\A = \{ H_1,\dots, H_n \}  \subset \RR^2$ be a real CL arrangement with $n$  components, and let $f_i \in \RR[x,y]$ be the corresponding
forms of the component $H_i$ for $1 \leq i \leq n$. As before, let $L = L(\A)$ be the poset of non-empty intersections (i.e. the non-empty intersections of elements of $\A$, ordered by inclusion). In most of the CL arrangements, the poset $L$  does not have a greatest lower bound, when there is at least one conic, since in that case it might be that not all the components
can pass through one point, which should have functioned as the greatest lower bound (this property holds also in the case of non-central line arrangements). For example, given the CL arrangement in Figure \ref{notMeetSemi}(a), its lattice is presented in Figure \ref{notMeetSemi}(b), and in this lattice there is no (greatest) lower bound for the elements $\{p_3 \}$ and $\{ p_1,p_2\}$. However, one may add artificially an element $e$, as a greatest lower bound, in order to transform $L$ into a meet-semilattice. In any case, $L$ is always a join-semilattice, since a least upper bound always exists.

\begin{figure}[h]
\epsfysize=5cm \centerline{\epsfbox{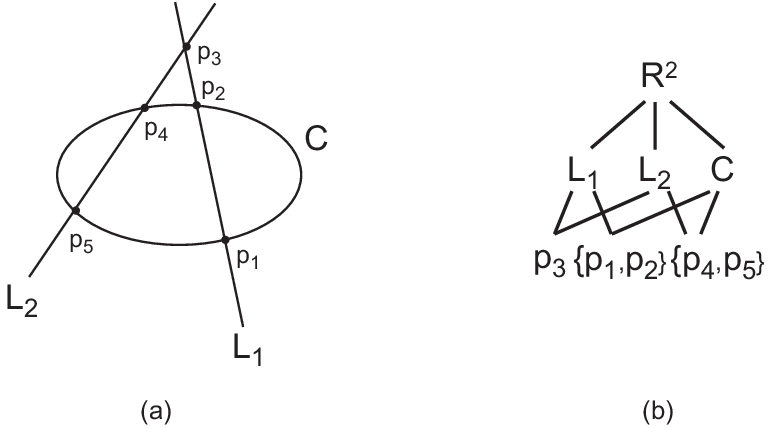}}
\caption{An example showing that in the lattice of a CL arrangement, a greatest lower bound of two elements does not necessarily exist:  the elements $\{p_3 \}$ and $\{ p_1,p_2\}$ do not have a (greatest) lower bound.}\label{notMeetSemi}
\end{figure}

Recall the set  $\A_X = \left\{ K \cap X \ |\ K \in \A - \{X\} \right\}$ for $X \in L$. This set may not be a CL arrangement,
but still we may  regard it as an object for which one can also associate a set of chambers ${\mathcal C}(\A_{X})$. For example, if $X$ is a conic, $\A_{X}$ is an arrangement of $n_X$ points $\{p_1,\dots,p_{n_X}\}$  (which are the intersection points between $X$ and the other components of $\A$) on $X$ and in that case ${\mathcal C}(\A_X)$ would be the set of the connected components of $X - \{p_1,\dots,p_{n_X}\} $.

Hence, we can define the partially-ordered set of faces as:
$$\LL = \LL(\A)= \bigcup_{X \in L} \C(\A_X),$$
where $\LL$ is ordered by inclusion in the closure of the larger face. We denote this partial order by $\preceq$.

As before, define a function:
$$i : \LL \to (\{+,-,0\})^n$$
as follows: $(i(P))_k = \text{sign}(f_k(P)),$ where $(i(P))_k$ is the value of the $k^{\text{th}}$ coordinate of the vector $i(P)$.

\begin{definition}
Define $\LL_0 = \LL_0(\A) = {\rm Image}(i) \subseteq (L_2^1)^n$. %As $L_2^1$ is an LRB, so does $\LL_0$.
\end{definition}

A detailed example of $\LL, \LL_0$ and $L$ can be found in Example \ref{Example2.14} below.

\medskip

\begin{remark}\label{rem2.7}
{\rm Note that the function $i$ is indeed well-defined also for the faces of a CL arrangement, since any face is connected and therefore any point in the face has the same mutual position with any component of the CL arrangement. Explicitly, if $a_1,a_2$ are two points in a face $P$, then since $P$ is connected, there exists a path $p:[0,1] \to P$ such that $p(0)=a_1$, $p(1)=a_2$ and for every $t_1,t_2 \in [0,1]$, $i(p(t_1)) = i(p(t_2))$.}
\end{remark}

\begin{remark}\label{rem2.8}
{\rm We remind that we exclude the hyperbolas from our family of CL arrangements since they divide the real plane into {\em three} chambers, despite the fact that two of these chambers have the same sign with respect to the hyperbola. However, if we work in $\RR\pp^2$, then these two chambers will be unified to the same chamber and hence CL arrangements in $\RR\pp^2$ can include hyperbolas as well.}
\end{remark}

We study some properties of the image of the function $i$, denoted by Image$(i)$.
Note that for real hyperplane arrangements, the function $i$ is monomorphic: every face $P$ is uniquely determined  by its vector of $n$ signs.
However, for real CL arrangements, this function might not be monomorphic.
For example, given a line $H_1$ and a circle $H_2$ tangent to it (see Figure \ref{ChamberConic2}(a)), the two parts of the line $H_1$ have the same pair of signs $(0,+)$. Another example is presented in Figure \ref{ChamberConic2}(b), where we have that:
$$i(P_1) = i(P_2) = (+,-,+)\in \left( L_2^1 \right)^3.$$

\begin{figure}[h]
\epsfysize=4cm
\centerline{\epsfbox{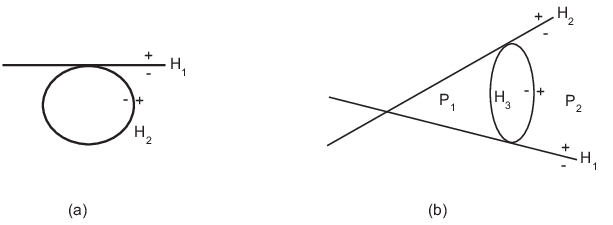}}
\caption{In part (a), the two parts of the line $H_1$ have the same pair of signs $(0,+)$.
Part (b) presents another real CL arrangement consisting of three components $H_1,H_2,H_3$ with two 2-dimensional faces $P_1,P_2$ having the same vector of signs: $$i(P_1)=i(P_2)=(+,-,+).$$}\label{ChamberConic2}
\end{figure}

Recall that one can define an associative product on $(L_2^1)^n = (\{+,-,0\})^n$ (see Section \ref{subsecLRB_Hyp}). We start with the following lemma regarding the relation between the product and the order relation:
\begin{lemma}\label{Lemma2.9}
Let $P_1,P_2$ be two faces of a CL arrangement $\A$. If $P_1 \preceq P_2$, then: $i(P_1)i(P_2)=i(P_2)$.
\end{lemma}

\begin{proof}
Recall that $P_1 \preceq P_2$ means that $P_1 \subseteq \overline{P_2}$. On the other hand, by the definition of the product, the equality $i(P_1)i(P_2)=i(P_2)$ means that any non-zero coordinate in $i(P_1)$ is equal to the corresponding coordinate in $i(P_2)$.

Assume that $P_1 \preceq P_2$. By definition, $P_1 \subseteq \overline{P_2}$, and since a face is not contained in the interior of a different face (see Remark \ref{rem-comb}(1)), then $P_1 \subseteq \partial P_2$, i.e. $P_1$ is contained in the boundary of $P_2$. Assume on the contrary that $i(P_1)i(P_2) \neq i(P_2)$, then there exists an index $j$ such that
$(i(P_1))_j \neq 0$ and $(i(P_1))_j \neq (i(P_2))_j$. Since the corresponding component $H_j$ of $\A$ divides $\RR^2$ into two parts: $\RR^2-H_j=R_1 \cup R_2$, it follows that $P_1$ and $P_2$ are not contained in the same part (due to the different signs in the $j^{\rm th}$ coordinate). Without loss of generality, we can assume that $P_1 \subset R_1$ and $P_2 \subset R_2$. Hence, $\overline{P_2} \subset \overline{R_2}=R_2 \cup H_j$. Since
$R_1 \cap (R_2 \cup H_j) =\emptyset$, this implies that $P_1 \not\subseteq \overline{P_2}$, which contradicts the assumption.
\end{proof}

Note that the converse of Lemma \ref{Lemma2.9} does not hold: In Figure \ref{prodNotGeo}(a) below, we have that
$i(p)i(x) = i(x)$, but $p \not\preceq x$.

\medskip

The following question is now raised: does this product give
${\rm Image}(i)$  the structure of a sub-semigroup of $(L_2^1)^n$? For hyperplane arrangements,  the answer is positive as one identifies $\LL$ with ${\rm Image}(i)$; thus $\LL$ is endowed with  a
semigroup structure. However, for real CL arrangements, as $i$ is not necessarily monomorphic (see Figures \ref{ChamberConic2}(a) and \ref{ChamberConic2}(b)), one  cannot identify
$\LL$ with ${\rm Image}(i)$ (and thus we need to  redefine the product on $\LL$).
A more serious problem is presented in the following example.

\begin{example} \label{exmNonClosedI}
\rm{There are real CL arrangements whose ${\rm Image}(i)$  is not even closed under the product induced by $(L_2^1)^n$, and thus it is not even a semigroup.

(1) For example, take three lines $H_1,H_2,H_3$ in general position (i.e. $H_1,H_2,H_3$  are not passing through a single point) and a circle $H_4=C$ passing through the three intersection points of the lines; see Figure \ref{NotClosedMulti}. Let $\al,\beta \in {\rm Image}(i) \subset  (L_2^1)^4$ be two quadruples of signs  associated to  two different triple intersection points (the points are $a$ and $b$ in Figure \ref{NotClosedMulti}). Explicitly:
$$\al = i(a) = (0,+,0,0),\,\,\, \beta = i(b)= (0,0,-,0).$$
Though $\al,\beta \in {\rm Image}(i)$, $\al  \beta = (0,+,-,0) \not\in {\rm Image}(i)$, since there is no face which corresponds to the quadruple of signs $\al  \beta = (0,+,-,0)$, as there is no  element in ${\rm Image}(i)$ that has exactly two zeros in its presentation as a quadruple in   $(L_2^1)^4$.

\begin{figure}[h]
\epsfysize=5cm \centerline{\epsfbox{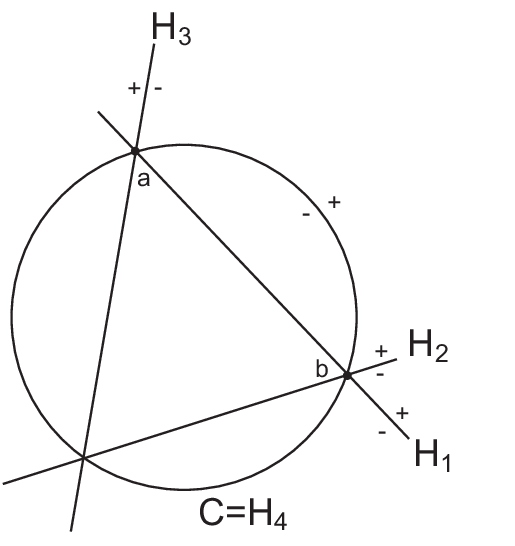}}
\caption{An example showing that Image$(i)$ is not necessarily closed under the product induced by $(L^1_2)^n$: $\al = i(a) = (0,+,0,0), \,\, \beta = i(b)= (0,0,-,0)$, but $$\al  \beta = (0,+,-,0) \not\in {\rm Image}(i).$$}\label{NotClosedMulti}
\end{figure}

Note that this is the minimal degree example for this phenomenon to occur: one can verify that for any real CL arrangement with degree $n \leq 4$, ${\rm Image}(i)$ is always closed under the  product induced by $(L_2^1)^n$.

Moreover, taking three generic lines  and a circle passing through only two of the intersection points, one can check that the product of the corresponding vectors of signs of the pair of triple points of this CL arrangement does not represent any face of this CL arrangement (by the same argument as above).

(2) The above example can be generalized: take an $n$-gon, where $n>3$, draw a circle passing through all the vertices of the $n$-gon, and extend the edges of the polygon into straight lines. One can easily check that the product of the corresponding vectors of signs of any pair of consecutive triple points of this CL arrangement does not represent any face of this CL arrangement (by the same argument as above).
}
\end{example}

By the previous examples, the following question arises:
\begin{question} \ \\
What are the necessary and sufficient conditions on $\A$ under which $\LL_0(\A)$ is closed under the product of $(L^1_2)^n$?
\end{question}

Obviously, if $\A$ is a line arrangement, then $\LL_0(\A)$ is a semigroup.  In order to partially answer this question, we recall the following definition:

\begin{definition}\label{def_inter_mul}
Let $a = (a_x,a_y) \in \CC^2$ be an intersection point of two curves given by the equations $\{f(x,y) = 0\}$ and $\{g(x,y) = 0\}$  where $f,g \in \CC[x,y]$. The {\em intersection multiplicity} of $a$, denoted by ${\rm multi}(a)$, is:
$${\rm multi}(a) = \dim\,\left( \CC[x,y]_{(x-a_x,y-a_y)}/(f,g) \right),$$
where $\CC[x,y]_{(x-a_x,y-a_y)}$ is the localization of the ring $\CC[x,y]$ at the ideal ${(x-a_x,y-a_y)}$  (see e.g. \cite[Chapter 1]{Ful}).
\end{definition}

For example, for $f = \{y=0\}$ and $g = \{y=x^2\}$, the point $a=(0,0)$ has multiplicity $2$,  so we write: ${\rm multi}(a)=2$.

\medskip

We introduce a sufficient condition on a real CL arrangement $\A$ for asserting  $\LL_0(\A)$ being a semigroup:

\begin{prs} \label{prsL0semigrp}
Let $\A$ be a real CL arrangement. Assume that for every singular point $p$ (i.e. a point for which at least two components of $\A$ pass through it) there are {\em only  two components} passing through it. Then
$\LL_0(\A)$ is a semigroup.
\end{prs}

\begin{proof}
We have to check that $\LL_0(\A)$ is closed under the product induced by  $(L_2^1)^n$. For each face $c \in \LL(\A)$, we have to go over all the possible products of the form $i(c) i(a)$, where $a \in \LL(\A)$, and check that $i(c)i(a) \in \LL_0(\A)$. The proof depends on the dimension of the face $c$.

If $\dim(c)=2$, there is nothing to check, as $i(c)i(a) = i(c)$ for every $a \in \LL(\A)$, since all the entries of $i(c)$ are non-zero and obviously $i(c) \in \LL_0(\A)$.

If $\dim(c)=1$, let $H = \{f=0\}$ be the component on which the face $c$ lays. Without loss of generality, we can assume that $H = H_1$, i.e. $(i(c))_1=0$ and $(i(c))_j \neq 0$ for $j>1$. Hence $i(c)i(a)$ is either $i(c)$ or an element $\beta \in (L^1_2)^n$ such that $(\beta)_1\in \{+,-\}$ and $(\beta)_j = (i(c))_j$ for all $j>1$, which means that $\beta$ is the image of one of the faces that has $c$ in its boundary (which lies inside the domain $\{f>0\}$ or $\{f<0\}$), and therefore $i(c)i(a) \in \LL_0(\A)$.

If $\dim(c)=0$, then $c$ is a point and by the assumption given in the formulation of the proposition, there are {\it exactly} two components passing through $c$. Without loss of generality, we can assume that the two components that pass through $c$ are $H_1$ and $H_2$ (where the other components in the CL arrangement are $H_3,\dots,H_n$). Since there exists a neighbourhood of $c$ that does not intersect  $H_3,\dots,H_n$, this means that for every face $b$, satisfying $c \in \overline{b}$, we have that $(i(b))_k = (i(c))_k$ for all $3 \leq k \leq n$. Since $(i(c))_k \neq 0$ for  $3 \leq k \leq n$ (otherwise $c$ would also be contained in one of the components $H_k$, where $3 \leq k \leq n$, which is impossible), we have that $(i(c))_k (i(a))_k = (i(c))_k$ for all $3 \leq k \leq n$. This means that with respect to the components $H_3,\dots,H_n$, the face $i(c)i(a)$ is in the neighbourhood of $c$ (i.e. in one of the faces $b$ mentioned above), since with respect to these components, the mutual position has not been changed. Thus, we just have to prove that when looking on the first two entries we also get an element in $\LL_0(\A)$.

We look now at two cases: either ${\rm multi}(c) \equiv 0 ({\rm mod}\, 2)$ or ${\rm multi}(c) \equiv 1 ({\rm mod}\, 2)$.

If  ${\rm multi}(c) \equiv 1 ({\rm mod}\, 2)$, then $c$  is either a node or a tangent point of multiplicity $3$. Therefore {\em locally}, in the neighbourhood of $c$, the CL arrangement $\A$ is of the form $\{xy=0\}$ or $\{y(y-x^3)=0\}$ (note that this {\em does not mean} that the curve itself has degree $3$). Note that as arrangements in $\RR^2$, $\LL_0(\{xy=0\}) = \LL_0(\{y(y-x^3)=0\}) = (L_2^1)^2$ and therefore $i(c)i(a) \in \LL_0(\A)$ for every $a \in \LL(\A)$.

If  ${\rm multi}(c) \equiv 0 ({\rm mod}\, 2)$, then $c$ is a tangent point of multiplicity $2$ or $4$, then {\em locally}, in the neighbourhood of $c$, the CL arrangement $\A$ is of the form $\{y(y-x^2)=0\}$ or $\{(y+x^2)(y-x^2)=0\}$ (note that this {\em does not mean} that the curve itself has degree $2$ or $4$), and thus, as arrangements in $\RR^2$, we have:
$$L_0 \doteq \LL_0(\{y(y-x^2)=0\}) =  \LL_0(\{(y+x^2)(y-x^2)=0\})  = (L_2^1)^2 - \{(-,+),(-,0),(0,+)\},$$
\noindent
where the first coordinate corresponds to the line $\{y=0\}$ or to the curve $\{y+x^2=0\}$ and the second coordinate corresponds to the curve $\{y-x^2=0\}$.
As can be easily checked,  $\LL_0(\A)$ is closed under this product, which means that $i(c)i(a) \in \LL_0(\A)$ for every $a \in \LL(\A)$.
\end{proof}

\begin{example}\label{Example2.14}
\rm{In Figure \ref{exampleDef} we present a CL arrangement $\A$, which consists of a conic $H_2$ and two lines $H_1$ and $H_3$, as an example that illustrates the definitions introduced so far: $L,\LL$ and $\LL_0$.

\begin{figure}[h]
\epsfysize=16cm \centerline{\epsfbox{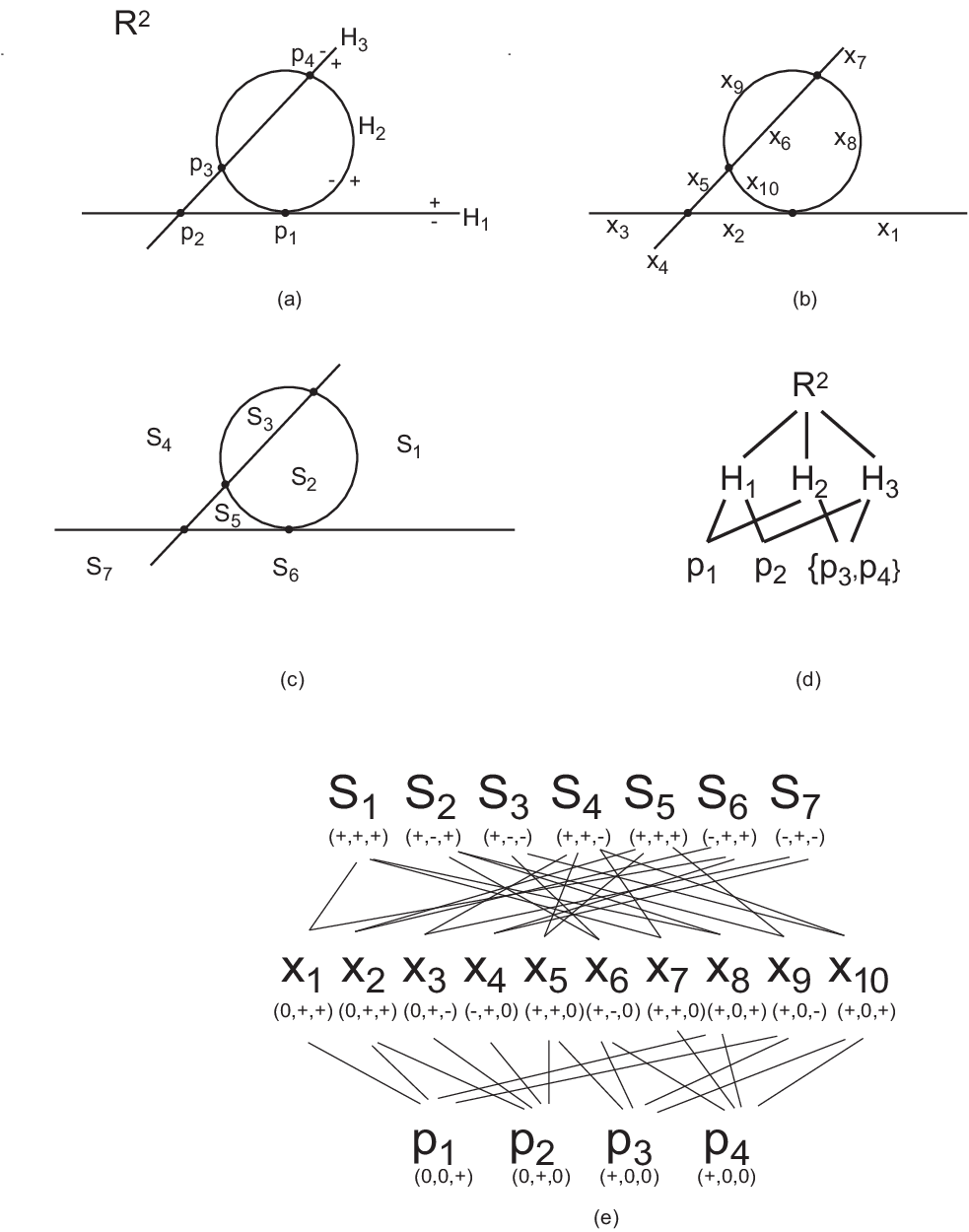}}
\caption{An example of a CL arrangement, whose components with the corresponding signs and the 0-dimensional faces are presented in part (a), whose 1-dimensional faces are presented in part (b) and whose 2-dimensional faces are presented in part (c). The corresponding intersection semi-lattice $L(\A)$ is depicted in part (d) and the poset of faces $\LL(\A)$ together with their corresponding vectors of signs under the function $i$, i.e. the LRB $\LL_0(\A)$, is presented in part (e).}\label{exampleDef}
\end{figure}

Figure \ref{exampleDef}(a) presents the components of the intersection semi-lattice $L$: $\RR^2$, the lines and the conic $\{H_1,H_2,H_3\}$ and the intersections of the components: $H_1 \cap H_2=\{ p_1 \}$, $H_1 \cap H_3=\{ p_2 \}$,
$H_2 \cap H_3=\{ p_3,p_4 \}$. Figure \ref{exampleDef}(b) presents the  $1$-dimensional elements of the poset of faces:  the open 1-dimensional sections $\{x_1,\dots,x_{10}\}$. Figure \ref{exampleDef}(c) presents the 2-dimensional elements of the poset of faces $\{S_1,\dots,S_7\}$. Figure \ref{exampleDef}(d) presents the intersection semi-lattice $L=L(\A)$ itself, while
Figure \ref{exampleDef}(e) presents the poset of faces $\LL(\A)$, together with their corresponding vectors of signs (according to the signs chosen in Figure \ref{exampleDef}(a)), i.e. their corresponding elements in $\LL_0(\A)$. By Proposition \ref{prsL0semigrp}, $\LL_0(\A)$ is a semigroup, hence an LRB too. Note that in this case the function $i : \LL(\A) \to \LL_0(\A)$ is not monomorphic, since $i(p_3) = i(p_4)$, $i(x_1) = i(x_2)$, $i(x_5) = i(x_7)$, $i(x_8) = i(x_{10})$ and $i(S_1) = i(S_5)$.
}
\end{example}

\subsection{Redefining the product}\label{subsec_product}

In line arrangements, we have that $\LL(\A) \simeq {\rm Image} (i)$ and hence $\LL(\A)$ has a natural product induced by
$(L^1_2)^n$. In the case of CL arrangements, we cannot identify $\LL(\A)$ with Image$(i)$ (as $i$ is not necessarily monomorphic as we saw in the previous example), and thus we have to redefine the product on $\LL(\A)$.
In this section, we introduce two different generalizations for the product defined for hyperplane arrangements to the case of real CL arrangements: one product turns $(\LL(\A), \cdot)$ into an alternative LRB (i.e. an alternative magma satisfying $x^2=x$ and $x\cdot y\cdot x=x\cdot y$ for every $x,y \in \LL = \LL(\A)$) and the second product turns  $(\LL(\A), \cdot)$  into an aperiodic semigroup.

\medskip

We want to use the same geometric intuition of the  product for
hyperplane arrangements (see Section \ref{subsecLRB_Hyp}) for defining the corresponding product on the face poset $(\LL,\preceq)$ for real CL arrangements (where $\preceq$ is the partial order defined by inclusion in the closure of the larger face).
Explicitly, we would like to maintain the following properties for every $x,y,z \in \LL$:
\begin{enumerate}
\item For every $x,y \in \LL\, , x^2=x$ and $x\cdot y\cdot x = x\cdot y$ (the LRB properties).
\item If $x \cdot y = z$, then $i(x)i(y) = i(z)$ (if there  exists a face with a  vector of signs $i(x)i(y)$). Explicitly,
if $\LL(\A)$ and $\LL_0(\A)$ are semigroups, then the surjective map $i: \LL(\A) \to \LL_0(\A)$ will be a homomorphism.
\item If $x \cdot y = z$, then $x \preceq z$.
\item If $x \preceq y$, then $x\cdot y = y$.
\item  $(x\cdot y)\cdot z = x\cdot (y\cdot z)$ (associativity).
\end{enumerate}

We present two different definitions for this product in the case of real CL arrangements. The first definition, appearing in Section \ref{subsecGeoProd}, preserves properties (1), (3) and (4) and thus will be more geometric, inducing a structure of an alternative LRB on $\LL$; the second, appearing in Section \ref{subsecAssocProd},  preserves properties (2), (5) and a weaker version of property (1), and thus will be more algebraic, inducing a structure of an aperiodic semigroup on $\LL$.

\begin{remark}
{\rm Based on the first definition (and when the defined product is associative), one can associate a quiver to the semigroup algebra $k \LL$, as was already done in the case of hyperplane arrangements (see \cite{Sal1}). We plan to check its properties in the future.}
\end{remark}

\subsubsection{The geometric product} \label{subsecGeoProd}
We start with examining   the CL arrangement presented in Figure \ref{prodNotGeo}, which shows that property (3) is not entirely based on the definition of $i$. Explicitly, we want that if  $x\cdot y = z$, then $x \preceq z$, i.e. $z$ is a face intersecting any neighbourhood of $x$. The example in Figure \ref{prodNotGeo} shows that this is not always the case when using the product induced by $(L^1_2)^n$. In the CL arrangement presented in Figure \ref{prodNotGeo}, $i(p)i(x)= (0,0,0)\cdot(0,-,0) = (0,-,0) =  i(x)$, but $p \not\preceq x$.

\begin{figure}[h]
\epsfysize=4cm \centerline{\epsfbox{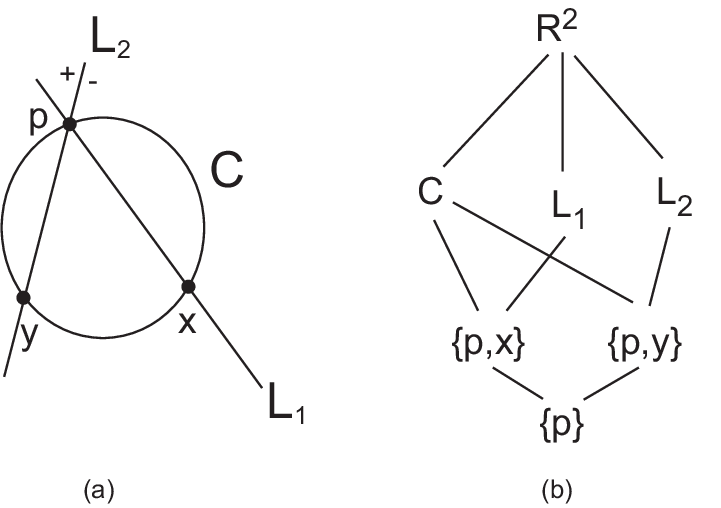}}
\caption{As $i(p) = (0,0,0), i(x) = (0,-,0)$, we have that  $i(p)i(x) = i(x)$, but $p \not\preceq x$. The signs of the line $L_2$ are also depicted (only its signs are relevant in this example). Part (b) depicts the intersection poset of the CL arrangement.}\label{prodNotGeo}
\end{figure}

We start by defining a set of faces $F(P_1,P_2)$ which plays a crucial role in the geometric definition:
\begin{definition}
Let $\A$ be a real CL arrangement, and let $P_1,P_2 \in \LL(\A)$.  Define:
$$F(P_1,P_2) \doteq \{ P \in \LL(\A) : i(P) = i(P_1)  i(P_2) {\rm{\,\, and\,\, }} P_1 \preceq P \}.$$
\end{definition}

The motivation for defining the set $F(P_1,P_2)$ is geometric: for defining the product of two faces (see Definition \ref{defPart2} below), we look for a face $P$ in the {\it neighbourhood} of $P_1$ (i.e. satisfies the condition $P_1 \preceq P$), whose expected vector of signs is induced by the corresponding product of the vectors of signs of $P_1$ and $P_2$ in $(L_2^1)^n$.

\medskip

Let us give two examples for computing the set $F(\cdot,\cdot)$, leaving the explicit computations to the reader.
\begin{example} \emph{(1) In the CL arrangement presented in Figure \ref{exmplDefF},
we have that:
$$F(p_1,p_2) = \{A_1,A_2\} \mbox{ and } F(p_2,p_3) = \{A_3\}.$$
\begin{figure}[h]
\epsfysize=3.5cm \centerline{\epsfbox{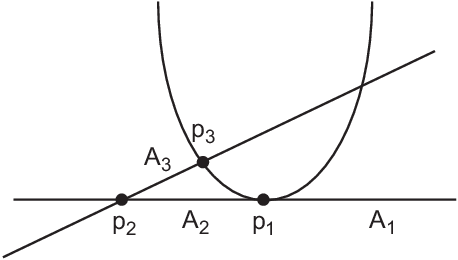}}
\caption{An example of a CL arrangement for computing the set $F(\cdot,\cdot)$. Note that the faces $A_1,A_2$ and $A_3$ have dimension $1$.} \label{exmplDefF}
\end{figure}
(2) Looking at Figure \ref{prodNotGeo} above, we have that $F(p,x)=\emptyset$.}
\end{example}

The following proposition summarizes  some properties of the set $F(P_1,P_2)$:
\begin{prs}\label{rem1Def} Let $P_1,P_2 \in \LL = \LL(\A)$. Then: \\
(1) If $P_1 \neq P_2$, then $P_1$ and $P_2$ cannot be together in the set $F(P_1,P_2)$. \\
(2)
\begin{itemize}
\vspace{-15pt}\item[(a)] If $\dim(P_1)=1$ or $\dim(P_1)=2$, then $|F(P_1,P_2)| = 1$ for all $P_2 \in \LL$.
\item[(b)] If $\dim(P_1)=0$, then $|F(P_1,P_2)| \leq 2$ for all $P_2 \in \LL$.
\end{itemize}
\end{prs}

\begin{proof}
(1) Assume on the contrary that  $P_1,P_2 \in F(P_1,P_2)$, then $i(P_1) = i(P_2)$. Moreover, since $P_2 \in F(P_1,P_2)$, we have $P_1 \preceq P_2$ and this implies that either $P_1$ is contained in the boundary of $P_2$ by Remark \ref{rem-comb}(1), but then $i(P_1) \neq i(P_2)$, or that  $P_1$ and $P_2$ have the same dimension, which means that $P_1 = P_2$, which is a contradiction.

\medskip

\noindent
(2)(a) If $\dim(P_1) \in \{1,2\}$, then we claim that $|F(P_1,P_2)|=1$: if $\dim(P_1)=2$, then $P_1$ is a chamber. Since for any chamber $P_1$,  $i(P_1)i(P_2)=i(P_1)$ and there is a unique chamber (i.e. $P_1$) in the neighbourhood of a generic point of the chamber $P_1$ with the same vector of signs as the vector of $P_1$, hence
$|F(P_1,P_2)|=|\{ P_1 \}|=1$.

If $\dim(P_1)=1$, i.e. $P_1$ is a section of a component of $\A$, then the possible faces in $F(P_1,P_2)$ can be $P_1$ or one of the two chambers having $P_1$ in their boundary, but each of these three faces has a different vector of signs, and thus $|F(P_1,P_2)|=1$.

\medskip

\noindent
(2)(b) If $\dim(P_1)=0$, then $P_1$ is a point. We start by proving the claim for the case where there are exactly two components of $\A$ passing through the point $P_1$. In this case, we have to deal with only two cases: \\
(i)  Through the point $P_1$ pass two transversal components (lines or conics) or  two conics which intersect each other with intersection multiplicity $3$.\\
(ii) The intersection multiplicity of the two components at the point $P_1$ is either $2$ or $4$ (i.e. the two components are tangent to each other with multiplicity $2$ or $4$).

\medskip

For case (i), we may consider the local neighbourhood of $P_1$ as represented by the CL arrangement $\{xy=0\}$ or by the CL arrangement $\{y(y-x^3)=0\}$ (note that this {\em does not mean} that the curve itself has degree $3$). Then, we have: $\LL_0(\{xy=0\}) = \LL_0(\{y(y-x^3)=0\}) = (L^1_2)^2$ (note that by the definition of $F(P_1,P_2)$, a face $P \in F(P_1,P_2)$ always intersects any neighbourhood of $P_1$).
By looking at this neighbourhood,  it is easy to see that if $P_a$ and $P_b$ are two different faces, which are both not equal to $P_1$ and satisfy: $P_1 \preceq P_a$ and $P_1 \preceq P_b$, then
$i(P_a) \neq i(P_b)$, which means that  $|F(P_1,P_2)|=1$ for any face $P_2 \in \LL$.

For case (ii), we may consider the local neighbourhood of $P_1$ as represented by the CL arrangement $\{y(y-x^2)=0\}$ or by the CL arrangement $\{(y+x^2)(y-x^2)=0\}$ (note that this {\em does not mean} that the curve itself has degree $2$ or $4$). Then, we have:
$$\LL_0(\{y(y-x^2)=0\}) = \LL_0(\{(y+x^2)(y-x^2)=0\})=(L_2^1)^2 - \{(-,+),(-,0),(0,+)\},$$
where the first coordinate corresponds to the line $\{y=0\}$ or to the curve $\{y+x^2=0\}$ and the second coordinate corresponds to the curve $\{y-x^2=0\}$. In this case, there may be two different faces $P_a$ and $P_b$  both not equal to $P_1$, such that $P_1 \preceq P_a$, $P_1 \preceq P_b$ and $i(P_a) = i(P_b)$ (such as the two sections on the line on which the point $P_1$ is lying on, see e.g. sections $A_1$ and $A_2$ in Figure \ref{exmplDefF}). However, there cannot be a triple of faces having this property, as can be checked directly. Hence, $|F(P_1,P_2)| \leq 2$ for any face $P_2 \in \LL$.

\medskip

Now we pass to the general case: we have to show that adding more lines or conics which pass through $P_1$ will not increase $|F(P_1,P_2)|$. As we have shown above, the fact $|F(P_1,P_2)| = 2$ implies that $P_1$ is a point, and the two components passing through $P_1$ are tangent to each other (with multiplicity $2$ or $4$). Now, there are two different ways for adding a new component through this point: either tangent to the existing two tangent components or transversal to both of them. In the case of adding a transversal component, we have, similar to case (i) above, that $|F(P_1,P_2)|=1$. In the case of adding a tangent component (see an example in Figure \ref{NewChambers}), the original two adjacent faces with the same vector of signs (e.g. faces $Q_1,Q_2$ in Figure \ref{NewChambers}(a)) will be now split into four adjacent faces (e.g. faces $Q_{1,1},Q_{1,2},Q_{2,1},Q_{2,2}$ in Figure \ref{NewChambers}(b)), but only two of them will have the same vector of signs, because the other two faces will be on the other side of the new tangent component $H$, and hence their vector of signs will have a different sign with respect to the new component $H$ (e.g. faces $Q_{1,1},Q_{2,1}$ in Figure \ref{NewChambers}(b) have the same vector of signs and faces $Q_{1,2},Q_{2,2}$ have a {\it different} common vector of signs). So we have at most two faces with the same vector of signs.

\begin{figure}[h]
\epsfysize=3cm \centerline{\epsfbox{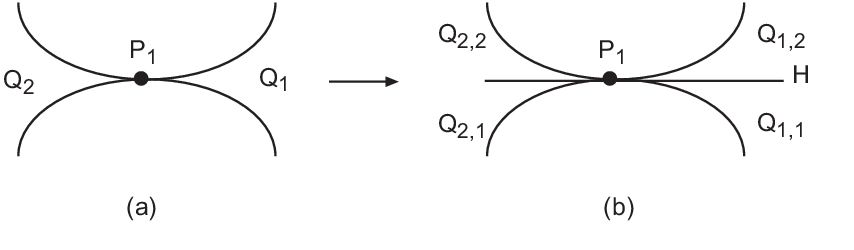}}
\caption{An example for illustrating the proof of Proposition \ref{rem1Def}(2)(b) for the general case of adding a tangent component passing through the tangent point $P_1$.}\label{NewChambers}
\end{figure}

\end{proof}

\begin{corollary} \label{cor_lineArr}
(1) If $\A$  is a line arrangement, then for any two faces $P_1,P_2 \in \LL$, $|F(P_1,P_2)| = 1$.\\
(2) If all the intersection points in a CL arrangement $\A$ are transversal, then $|F(P_1,P_2)| \leq  1$.
\end{corollary}

\begin{proof}
(1) First, note that for line arrangements, if $i(P) = i(P_1)i(P_2)$, then immediately $P_1 \preceq P$. Second, if there were two faces $P,Q$ such that
$i(Q) = i(P) = i(P_1)i(P_2)$, then $i$ would not be monomorphic, which is a contradiction to the situation in line arrangements. Moreover, by the geometric interpretation of the product on the set of faces on line arrangements, there always exists a unique face $P \in \LL$ satisfying $i(P) = i(P_1)i(P_2)$.

\medskip

\noindent
(2) By the proof of Proposition \ref{rem1Def}(2)(b), we have that the fact $|F(P_1,P_2)|=2$ implies that $P_1$ is a tangency point. Since all the intersection points in $\A$ are transversal, we immediately have that $|F(P_1,P_2)| \leq  1$.
\end{proof}

We are now ready to define the geometric product. {\em Informally,} the product of the faces $P_1$ and $P_2$ will be the element of $F(P_1,P_2)$ (which intersects any neighbourhood of the face $P_1$) that is ``closest'' to the face $P_2$. If there are two such faces and none of them is $P_2$, then take the first one in the clockwise direction; and if no such face exists, then take the face $P_1$.

Note that we cannot guarantee that this product will be associative (see Example \ref{exampNonAssoc} below for some examples of  real CL arrangements inducing a non-associative product).

Note also that in the following definition, one has to check the cases {\it sequentially case by case}.

\begin{definition} (Geometric product on $\LL(\A)$) \label{defPart2}\\
Let $\A$ be a real CL arrangement and let $P_1,P_2 \in \LL(\A)$.

If $|F(P_1,P_2)|=0$, then define $P_1\cdot P_2 \doteq P_1$.

\smallskip

If $|F(P_1,P_2)|=1$, i.e. $F(P_1,P_2) = \{P\}$, then define  $P_1\cdot P_2\doteq P$.

Otherwise, we know that  $|F(P_1,P_2)|=2$.

\medskip

If $P_2 \in F(P_1,P_2)$, then define $P_1~\cdot~P_2~\doteq~P_2$.

\smallskip

Otherwise, we have that $|F(P_1,P_2)|=2$ and $P_2 \not \in F(P_1,P_2)$. By Proposition \ref{rem1Def}(2)(b), this can only happen when $P_1$ is a point, and all the components of $\A$ passing through $P_1$ are tangent to each other (at $P_1$), with multiplicity $2$ or $4$.

If $P_1$ and $P_2$ are on the same \emph{unbounded} $1$--dimensional component $H$, then $P_1\cdot P_2$ is defined to be the face in $F(P_1,P_2)$ we get after moving from $P_1$ on $H$ in the direction of $P_2$ (see Figure \ref{cases}(a)~\footnote{\ Note that the last sentence has no meaning if $P_1$ and $P_2$ are located on {\em two} $1$--dimensional unbounded components (e.g. two parabolas or a line and a parabola), but this case is impossible, since then through $P_1$ pass at least two components which intersect transversally (by the geometric properties of CL arrangements) and hence we have $|F(P_1,P_2)| \leq 1$, by the proof of Proposition \ref{rem1Def}(2)(b).}).

Otherwise, either all the elements in $F(P_1,P_2)$ are chambers or  that $P_1$ and $P_2$ are on the same \emph{bounded} $1$--dimensional component (i.e. an ellipse).
For each $P \in F(P_1,P_2)$, let $\ell_P$ be the minimal length of an arc passing through the point $P_1$, a generic point in $P$ and a generic point in $P_2$. If the minimum of the set $\{\ell_P\}_{P \in F(P_1,P_2)}$ is attained only once, say, at a face $P_0$, then define $P_1\cdot P_2 \doteq P_0$ (see Figure \ref{cases}(b)).
However, if there exist two faces $P',P'' \in F(P_1,P_2)$  satisfying:
$$\inf_{P \in F(P_1,P_2)} \{\ell_P\} = \ell_{P'} = \ell_{P''},$$
then draw a circle $C$ through $P_1$, a generic point in $P'$ (or in $P''$) and a generic point in $P_2$ and define $P_1\cdot P_2 \doteq P$, where $P \in \{P',P''\}$ is the face we are in after moving slightly \emph{clockwise} on the circle $C$ from $P_1$ (see Figure \ref{cases}(c)).

\begin{figure}[h]
\epsfysize=4.5cm \centerline{\epsfbox{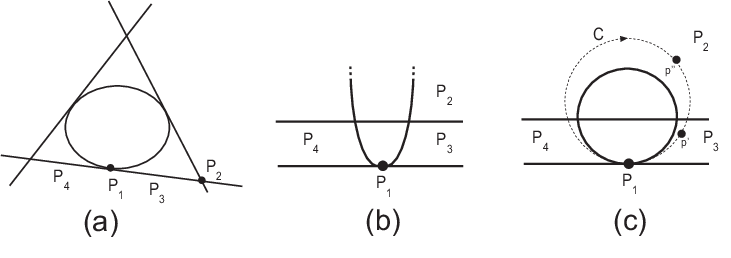}}
\caption{\small
Different situations for the geometric product on $\LL(\A)$:
In part (a),\break $P_2 \not\in F(P_1,P_2) = \{P_3,P_4\}$ and $P_1$ and $P_2$ are on the same \emph{unbounded} $1$--dimensional component, so $P_1\cdot P_2 = P_3$.
In part (b), $P_2 \not\in F(P_1,P_2) = \{P_3,P_4\}$. Moreover, $\ell_{P_3} < \ell_{P_4}$, so we have: $P_1\cdot P_2 = P_3$.
In part (c), again  $P_2 \not\in F(P_1,P_2) = \{P_3,P_4\}$, but in this case  $\ell_{P_3} = \ell_{P_4}$, so we draw a dotted circle $C$ through $P_1$, a generic point $P'$ in $P_3$ and a generic point $P''$ in $P_2$, and move \emph{clockwise} on it from $P_1$ to get: $P_1\cdot P_2 = P_4$.} \label{cases}
\end{figure}
\end{definition}

Obviously, by the definition, requirement (3) (i.e. if $x \cdot y = z$, then $x \preceq z$) holds. Note that in  the example presented in Figure \ref{prodNotGeo}, when we use this product, then $p \cdot x = p$, since $F(p,x) = \emptyset$.

Moreover, if  $x,y,z \in \LL(\A)$, then $x\cdot(y\cdot z)$ is a face $\alpha \in \LL(\A)$ satisfying $x \preceq \alpha$, $x \cdot y$ is a face $\beta' \in \LL(\A)$ satisfying $x \preceq \beta'$ and
$(x \cdot y)\cdot z$ is a face $\beta \in \LL(\A)$ satisfying $\beta' \preceq \beta$; thus $x \preceq \beta$. This implies that even if the product is not associative, then:
$$x \subseteq  \overline{x\cdot(y\cdot z)} \cap \overline{(x\cdot y)\cdot z}.$$

\medskip

The following proposition presents some properties of the geometric product:

\begin{prs} \label{remDef} Let $P_1,P_2 \in \LL = \LL(\A)$. Then the following properties hold: \\
(1) If $P_1 \in F(P_1,P_2)$, then $P_1 \cdot P_2 = P_1$.\\
(2) If $P_1 \preceq P_2$, then $P_1 \cdot P_2 = P_2$ (requirement (4)).\\
(3) If $|F(P_1,P_2)|=1$ for every two faces $P_1,P_2 \in \LL$, then the product is associative.
\end{prs}

\begin{proof}
(1) The proof depends on the dimension of the face $P_1$: if dim$(P_1)= 2$, it is obvious, as already $i(P_1)i(P_2) = i(P_1)$ and $P_1$ is the only face $X$ satisfying $P_1 \preceq X$.
If dim$(P_1)< 2$, then in the neighbourhood of $P_1$, the only face with the same vector of signs as $P_1$ is $P_1$ (note that if $P_1 \in F(P_1,P_2)$, then by definition $i(P_1)i(P_2) = i(P_1)$).

\medskip

\noindent
(2) If $P_1=P_2$, then $P_1 \cdot P_2 = P_1 \cdot P_1 = P_1^2 = P_1 = P_2$ (note that $P_1 \in F(P_1,P_1)$ and by part (1) above, $P_1^2 = P_1$).
Otherwise, since $P_1 \preceq P_2$ and $P_1 \neq P_2$, then  $i(P_1)i(P_2) = i(P_2)$ (by Lemma \ref{Lemma2.9}) and therefore $P_2 \in F(P_1,P_2)$. By definition, $P_1 \cdot P_2 = P_2$.

\medskip

\noindent
(3) Let $x,y,z \in \LL$. We know that any neighbourhood of $x$ intersects both $w \doteq (x \cdot y) \cdot z$ and $v \doteq x \cdot (y \cdot z)$ (by the property $x \subseteq  \overline{x\cdot(y\cdot z)} \cap \overline{(x\cdot y)\cdot z}$ mentioned above), and
$w$ and $v$ have the same vector of signs. Indeed, note that since $|F(p,q)|=1$ for every $p,q \in \LL$, then $i(p\cdot q) = i(p)i(q)$, i.e. the function $i$ is a homomorphism and thus:
$$i(w) = i(x\cdot y)i(z) = (i(x)i(y))i(z) = i(x)(i(y)i(z)) =i(x)i(y \cdot z) =i(v).$$

The continuation of the proof depends on the dimension of $x$. If dim$(x) > 0$, then a neighbourhood of $x$ can intersect only one face with a given vector of signs (see Proposition \ref{rem1Def}(2)(a)),
which implies that $v=w$.

If dim$(x)=0$,  a neighbourhood of $x$ may intersect two different faces with the same vector of signs (see Proposition \ref{rem1Def}(2)(b)).
That is, $x$ is in the boundary of $w$ and $v$, and thus $F(x,w) = \{w,v\}$ (as $x \preceq w$, $x \preceq v$ and $i(x)i(w) = i(w) = i(v)$; the first equality is derived from requirement (4), which holds by the definition of $F(\cdot,\cdot)$ and that $P_1~\cdot~P_2~=~P_2$ in the case that $P_2 \in F(P_1,P_2)$). Hence, $|F(x,w)| = 2$, which is a contradiction. This means that $v=w$.
\end{proof}

\begin{remark}
{\rm Proposition \ref{remDef}(3) {\em does not mean} that if all the intersection points are transversal, then the induced product is associative, see Example \ref{exampNonAssoc}(1) below. Indeed, in line arrangements, all the intersection points are transversal and $|F(P,Q)|=1$ for every two faces $P,Q \in \LL$, but the problems arise when there exist two faces $P,Q \in \LL$ satisfying $|F(P,Q)| \neq 1$.}
\end{remark}

Note that if $\A$ is a line arrangement, then $|F(P,Q)|=1$ for every two faces $P,Q \in \LL$. This means that the geometric product introduced in Definition \ref{defPart2} indeed generalizes the original product defined in the case of line arrangements.

Since requirement (4) (i.e. if $x \preceq y$, then $x\cdot y =y$) holds, we can now prove requirement (1); i.e.  that $\LL$ is an \emph{alternative LRB}:

\begin{prs}
Assume that $\cdot$ is the product on $\LL$ introduced in Definition \ref{defPart2}.
%and that if $x \subseteq \bar y$ then $x \cdot y = y$.
Let $x,y \in \LL$. Then $(\LL,\cdot)$ is an \emph{alternative LRB}, i.e.:
\begin{enumerate}
\item $x^2=x$,
\item $x \cdot (x \cdot y) = (x \cdot x) \cdot y$ \emph{and} $x \cdot (y \cdot y) = (x \cdot y) \cdot y$,
\item $(x \cdot y) \cdot x = x \cdot (y \cdot x) = x \cdot y$.
\end{enumerate}
\end{prs}

\begin{proof}
(1) As $x \in F(x,x)$, we get that $x^2=x$ (by Proposition \ref{remDef}(1)).

\medskip

\noindent
(2) We prove that $x \cdot (x \cdot y) = (x \cdot x) \cdot y = x \cdot y$ and $x \cdot y = x \cdot (y \cdot y) = (x \cdot y) \cdot y$.
If we denote $z = x \cdot y$, then by definition $x \preceq z$. Thus $x \cdot (x \cdot y) = x\cdot z = z = x \cdot y$ (the second equality is by Proposition \ref{remDef}(2)).

Now, we will show that $z \cdot y = z$. Its proof depends on the cardinality of the set $F(x,y)$. If $|F(x,y)|>0$, then $i(z) = i(x)i(y)$ and thus
$i(z)= i(x)i(y) = i(x)i(y)i(y) = i(z)i(y)$ (since this holds in $(L_2^1)^n$) and therefore $z \in F(z,y)$; thus $z \cdot y = z$ (by Proposition \ref{remDef}(1)). Otherwise, $|F(x,y)|=0$ and thus $x \cdot y = x$, i.e. $z=x$. Thus $z\cdot y = x \cdot y = x = z$; i.e. $z \cdot y = z$.

Therefore, in both cases, $(x \cdot y) \cdot y = x \cdot y = x \cdot (y \cdot y)$ and hence $(\LL,\cdot)$ is an alternative magma.

\medskip

\noindent
(3) By considering the algebra $\mathbb{R} \LL$, one obtains that $\mathbb{R} \LL$ is an alternative algebra, and by \cite[pp. 27--28]{Sch}, in these algebras the {\em flexible identity} $(x \cdot y) \cdot x = x \cdot (y \cdot x)$ holds. Hence, it holds also for any two faces $x,y$ in $\LL$, as this identity does not depend on the embedding algebra. Therefore, the expression $x \cdot y \cdot x$ is well defined.

As before, the proof depends on the cardinality of the set $F(x,y)$. If $|F(x,y)|>0$, then $i(x \cdot y) = i(x)i(y)$
and hence $i(x \cdot y) = i(x)i(y) = i(x)i(y)i(x) = i(x \cdot y)i(x)$ and therefore we have that $x \cdot y \in F(x \cdot y,x)$ and thus by Proposition \ref{remDef}(1), $x \cdot y \cdot x = x \cdot y$. Otherwise, $|F(x,y)|=0$ and thus $x \cdot y = x$ and so
$x \cdot y \cdot x =  x \cdot x = x = x \cdot y$.
\end{proof}

The next example shows that the geometric product is not always associative.

\begin{example} [Non-associative geometric products] \label{exampNonAssoc} \ \\
{\rm In the following examples, we show that the geometric product may  not be associative if there exist two faces $P_1,P_2 \in \LL$ satisfying $|F(P_1,P_2)| \neq 1$. The first example shows the non-associativity where there exist two faces $P_1,P_2 \in \LL$ satisfying $|F(P_1,P_2)| = 0$, and the second example shows the non-associativity where there exist two faces $P_1,P_2 \in \LL$ satisfying $|F(P_1,P_2)| = 2$.

\medskip

\noindent
(1) Given the real CL arrangement $\A_1$ presented in  Figure \ref{exampNonAssocProd1}, note that $|F(P_1,P_2)| \leq 1$ for any two faces $P_1,P_2 \in \LL$ since all the intersection points are transversal. Note also that $$|F(p_1,p_2)|=|F(p_1,p_3)|=0$$ (where $p_1,p_2,p_3$ are the intersection points and $s$ is the shorter arc between $p_2$ and $p_3$).

\begin{figure}[!ht]
\epsfysize 3.5cm
\centerline{\epsfbox{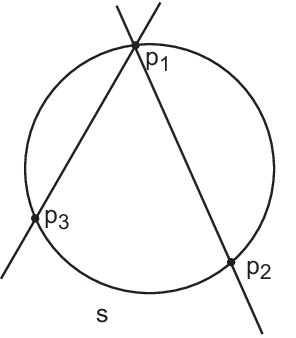}} % to be changed
\caption{An example of a non-associative geometric product in the case of existence of two faces $P_1,P_2 \in \LL$ satisfying $|F(P_1,P_2)| = 0$: $$s = (p_2 \cdot p_1) \cdot p_3 \neq p_2 \cdot (p_1 \cdot p_3) = p_2.$$}\label{exampNonAssocProd1}
\end{figure}

We have: $p_2 \cdot p_1 = p_2$, and
thus  $(p_2 \cdot p_1) \cdot p_3 = p_2 \cdot p_3 = s$. On the other hand, since $|F(p_1,p_3)| = 0$ we have: $p_1 \cdot p_3 = p_1$, and then:
$$p_2 \cdot (p_1 \cdot p_3) = p_2 \cdot p_1 = p_2 \neq s = (p_2 \cdot p_1) \cdot p_3.$$
Thus the geometric product is not associative for the CL arrangement $\A_1$. However, note that $\LL_0(\A_1)$ is an associative LRB, by a direct check.

\medskip

\noindent
(2) Look at the real CL arrangement $\A_2$ presented in  Figure \ref{exampNonAssocProd}, where the circle in $\A_2$ is denoted by $C$.
All the labeled faces are on the circle, where $x$ and $y$ are tangent points  and $b,w,m$ and $z$ are 1-dimensional faces. Let us compute $(x \cdot y) \cdot z$ and $x \cdot (y \cdot z)$.

\begin{figure}[!ht]
\epsfysize 3.5cm
\centerline{\epsfbox{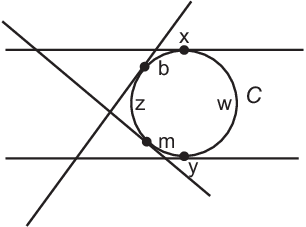}}
\caption{An example of a non-associative geometric product in the case of existence of two faces $P_1,P_2 \in \LL$ satisfying $|F(P_1,P_2)| = 2$: $$b = x \cdot (y \cdot z) \neq (x \cdot y) \cdot z = w.$$}\label{exampNonAssocProd}
\end{figure}

Note that $F(x,y) = \{b,w\}$ (so we have $|F(x,y)| =2$) and  $\ell_{b} = \ell_{w}$. Thus, we should go clockwise on the circle $C$ from $x$ to $y$ and therefore $x \cdot y = w$ and so: $(x \cdot y) \cdot z = w \cdot z = w$. However,  $F(y,z) = \{m,w\}$ and  $\ell_{m} < \ell_{w}$. Thus $y \cdot z = m$ and by the same argument, $x \cdot m = b$. Therefore:
$$x \cdot (y \cdot z) = x \cdot m = b \neq w = (x\cdot y)\cdot z .$$
Thus the geometric product on $\LL(\A_2)$ is not associative.
However, note that $\LL_0(\A_2)$ is an associative LRB, by Proposition \ref{prsL0semigrp}.}
\end{example}

\begin{remark}\label{hyperLRB-rem}
{\rm Note that one way to overcome the need to choose an element from the set $F(P_1,P_2)$ when $|F(P_1,P_2)|=2$ is to consider the set of faces $\LL(\A)$ as an hyper-LRB. An {\em hyper-LRB} is a nonempty set $S$ together with a map $*: P^*(S) \times P^*(S) \to P^*(S)$, where $P^*(S)$ consists of all nonempty subsets of $S$, such that this algebraic structure satisfies the additional properties of an LRB: $\{ x \} * \{ x \}=\{ x \}$ and $\{ x \} * \{ y \} * \{ x \} = \{ x \} * \{ y \}$ for all $x,y \in S$, where the equalities are equalities between sets (some references for this proposed algebraic structure are \cite{Co,CoLe,Voug}). The investigation of this natural structure will be treated in a future research.}
\end{remark}

\subsubsection{The associative product} \label{subsecAssocProd}

As we saw in Example \ref{exampNonAssoc}, the geometric product introduced in Definition \ref{defPart2} may not be  associative. Moreover, it does
not satisfy requirement (2), i.e., if $x \cdot y = z$ then $i(x)i(y) = i(z)$, where $i : \LL \to \LL_0$ is the sign function, sending each face to its associated vector of signs. In this section, we introduce a different product on $\LL$  that will be associative and satisfy requirement (2). However, in order to obtain this, we have to assume that $\LL_0(\A) = {\rm Image}(i)$ is  closed under the product induced by $(L_2^1)^n$ (see Example \ref{exmNonClosedI} above for CL arrangements whose Image($i$) is not closed under this product and see Proposition \ref{prsL0semigrp} for a sufficient condition for the closeness of the product on $\LL_0(\A)$).

\begin{definition} (Associative product on $\LL(\A)$) \label{defAssocProd}\\
Let $\A$ be a real CL arrangement such that $\LL_0(\A)$ is closed under the product induced by $(L_2^1)^n$. Define a function $j:\LL_0 \rightarrow \LL$ as follows. For every $a \in \LL_0$,
if $|i^{-1}(a)| = 1$, then $j(a) \doteq i^{-1}(a)$. Otherwise, choose an element $a_0 \in i^{-1}(a)$ and define $j(a) \doteq a_0$.

For any two faces $x,y \in \LL$, define the product:
$$x \cdot y \doteq j(i(x)i(y)).$$
\end{definition}

Note that in the case of a line (and hyperplane) arrangement $\A$, for every $a \in \LL_0(\A)$, $|i^{-1}(a)| = 1$. This means that the associative product introduced in Definition \ref{defAssocProd} coincides with the one induced by $(L^1_2)^n$. Hence, this product can also be thought of as a generalization of the corresponding product defined in the case of line arrangements.

\medskip

In the following proposition, we present some properties of this product:

\begin{prs}\label{prsAssocProdProperty}
Let $(\LL,\preceq)$ be the poset of faces of a real CL arrangement, and let $\cdot$ be the product introduced in Definition \ref{defAssocProd} (i.e. the function $j$ is already given). Then:
\begin{enumerate}
\item $x \cdot (y \cdot z) = (x \cdot y) \cdot z$ (associativity),
\item $x \cdot y \cdot x = x \cdot y$,
\item $x^2$ is not necessarily equal to $x$.
\end{enumerate}
\end{prs}

\begin{proof}
We will prove property (1). Property (2) is proven similarly using the identity $i(x)i(y)i(x) = i(x)i(y)$ in $\LL_0(\A)$. For proving property (1), we have to show that $x \cdot (y \cdot z) = (x \cdot y) \cdot z$:
\begin{eqnarray*}
x \cdot (y \cdot z) & \stackrel{(*)}{=} & x \cdot j(i(y)i(z)) = \\
& \stackrel{(*)}{=} & j(i(x)i(j(i(y)i(z)))) =\\
& \stackrel{(**)}{=} & j(i(x)(i(y)i(z)))=\\
& \stackrel{(***) }{=} & j((i(x)i(y))i(z))=\\
& \stackrel{(**)}{=} & j(i(j(i(x)i(y)))i(z))=\\
& \stackrel{(*)}{=} & (j(i(x)i(y))) \cdot z =\\
& \stackrel{(*)}{=} & (x \cdot y) \cdot z
\end{eqnarray*}
where $(*)$ uses the definition of the associative product, $(**)$ is by the identity $i \circ j ={\rm Id}$ and $(***)$ uses the associativity in $\LL_0(\A)$.

For showing property (3), look at the CL arrangement consisting of a line intersecting transversally a circle. Let $p_1$ and $p_2$
be the two intersection points, and denote $\al = i(p_1)$. Note that $i(p_1) = i(p_2) = \al$. We may choose $j(\al) = p_1$ or $j(\al) = p_2$. If we choose $j(\al) = p_1$, we have that:
$$p_2\cdot p_2 = j(i(p_2)i(p_2)) = j(i(p_2)) = j(\al) = p_1,$$
and thus: $p_2^2 = p_1 \not= p_2$. For the other choice, i.e. $j(\al) = p_2$, we get that $p_1^2 = p_2 \not= p_1$.
\end{proof}

\begin{remark}
{\rm (1) Notwithstanding Proposition \ref{prsAssocProdProperty}(3), the product introduced in  Definition \ref{defAssocProd} still satisfies $x^2 = x^3$ (this is a specific case of Proposition \ref{prsAssocProdProperty}(2), when taking $x=y$). Thus, $(\LL, \cdot)$  is an \emph{aperiodic semigroup}, i.e. for every $x \in \LL$, $x^2$ is an idempotent, since
$$(x^2)^2 = x^4 = x^3\cdot x = x^2 \cdot x = x^3 = x^2,$$
and the set $\{x^2 : x \in \LL\}$ is an LRB.

(2) Note  that once there are different faces in $\LL$ having the same image under $i$, then Definition \ref{defAssocProd} does not introduce a unique product on $\LL$, as it depends on the choice made by the function $j$ in this definition.

(3) Note that one way to overcome the need to choose of an element made by the function $j$ (mentioned in part (2) of this remark) is to consider the set of faces $\LL(\A)$ as an hyper-LRB, as already mentioned with respect to a similar problem in the geometric product (see Remark \ref{hyperLRB-rem} above). The investigation of this natural structure will be treated in a future research as well.}
\end{remark}

\subsection{Applications} \label{subsec_App}

In this section, we deal with three applications of the LRB structure associated to the face poset of CL arrangements, introduced in the above sections: random walks (Section \ref{subsecRandomWalk}), stereographic projections (Section \ref{subsec_strProj}) and an example of an LRB, induced by a CL arrangement, which cannot be induced by any hyperplane arrangement (Section \ref{subsecLRB_cl}).

\subsubsection{Random walks}\label{subsecRandomWalk}

In this subsection, we briefly investigate the random walks on the set of faces of a given CL arrangement, with comparison to the known results in hyperplane arrangements.

Note that for a semigroup $S$, a \emph{step} in the random walk goes from $s \in S$ to $x \cdot s$, where $x\in S$ is chosen with probability $w_x$. Concentrating on LRBs, it is reasonable to choose $s$ to be a chamber, hence getting a random walk on the set of chambers (which is an ideal, since $x \cdot s$ is also a chamber, regardless of what $x$ is).

In general, there are several results about random walks on semigroups, see e.g. \cite{BCL}, but more surprising results arise when we concentrate on LRBs.  For example (see Theorem \ref{brownThm} below), the transition matrix of the random walk on the chambers of an LRB can be diagonalized over $\RR$ and its eigenvalues and their multiplicities can be easily computed \cite{B,BD}.

Let $\A$ be a CL arrangement having at least one conic, such that there is an {\em associative} product on $\LL = \LL(\A)$, inducing on it an LRB structure. In order to formulate the theorem exactly, we need to introduce a new semi-lattice $L'$ associated to $(\LL, \cdot)$ (see \cite{B,B2} and \cite[p. 153]{St}):

\begin{definition}
For $x,y\in \LL$, define $x \sim y$ if and only if $x\cdot y=x$ and $y \cdot x=y$. Define $L' = \LL/\sim $ and let  the map ${\rm supp'}: \LL \twoheadrightarrow L'$ be the quotient map. For $x \in \LL$, let $[x]$ be the corresponding element in $L'$.
\end{definition}

One can see that $L'$ is a semi-lattice and ${\rm supp'}$ is order-preserving (see \cite[p. 153]{St}).

Note that if the product is not associative, then the relation $\sim$ is not an equivalence relation (see Example \ref{exmDiffLattice} below).

Although $L'$ is a semi-lattice, it may not be a lattice, since a greatest lower bound may not exist (but can be added; see the definition of $\hat{L'}$ below). $L'$ is a graded semi-lattice of either rank $2$ or rank $3$: if the lattice has a minimal element, then it is of rank $3$; note that this minimal element would be the point (if it exists) which all the $1$-dimensional components pass through it (note that if the arrangement $\A$ is a central line arrangement, i.e. all the lines intersect at one point, then $L'$ is of rank $2$). For example, looking at the CL arrangement $\A$ in Figure \ref{example_L_tag} below and at the associated associative LRB $\LL(\A)$, one notes that there are three $0$-dimensional faces (the three intersection points $x,y,z$), nine $1$-dimensional faces and seven $2$-dimensional faces. When computing the associated semi-lattice $L'$, we see that the equivalence class of the point $z$ is the minimal element, whereas the rank of the equivalence classes of $x$ and $y$ is 1 and the whole lattice $L'$ has rank $3$.
If the lattice does not have a minimal element, then the equivalence class of a point has rank $0$, a line or a conic have rank $1$ and $\RR^2$ has rank $2$. $L'$ is also a semimodular semi-lattice. Recalling the definition of a semimodular semi-lattice (i.e. for every $u,v\in L$ such that $u$ and $v$ cover $x\in L$, there exists $y\in L$ that covers $u$ and $v$), we see that in our case (a lattice of rank $2$ or $3$), $[x]$ can only be an equivalence class of a point, and $[x]$ is contained in $[u]$ and $[v]$ where $u$ and $v$ are lines or conics, and hence $y = \RR^2$.

\begin{figure}[!ht]
\epsfysize 3.5cm
\centerline{\epsfbox{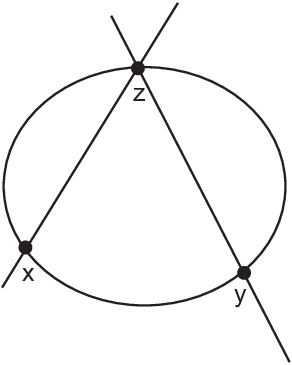}}
\caption{An example of a CL arrangement whose semi-lattice $L'$ has rank 3. Note that $L'$ is isomorphic to the intersection semi-lattice $L$ in this case, as can be seen in Figure \ref{prodNotGeo}(b).}\label{example_L_tag}
\end{figure}

\begin{example}\label{exmDiffLattice}
\rm{
Let us look again at Example \ref{exampNonAssoc}(1) and Figure \ref{exampNonAssocProd1}. Recalling that the faces $s,p_1,p_2,p_3$ satisfy $s = (p_2 \cdot p_1) \cdot p_3 \neq p_2 \cdot (p_1 \cdot p_3) = p_2$, we note that $p_1\cdot p_2 = p_1$ and $p_2\cdot p_1 = p_2$, and also that $p_1\cdot p_3 = p_1$ and $p_3\cdot p_1 = p_3$. Hence $p_1 \sim p_2$ and $p_1 \sim p_3$. However, $p_2 \cdot p_3 = s$ and hence $p_2 \not\sim p_3$. Hence the non-associativity implies that the relation $\sim$ is not transitive, i.e. the relation $\sim$ is not an equivalence relation if the product is not associative. Indeed, assume that $x \sim y$ and $y \sim z$. So we have: $x \cdot y=x$ and $y \cdot z=y$. Hence, the inequality $(x \cdot y)\cdot z \neq x \cdot (y \cdot z)$ implies that $x \cdot z \neq x \cdot y = x$. Thus $x \cdot z \neq x$ and hence
$x \not\sim z$.
}
\end{example}

Since for a CL arrangement, the associated face LRB $\LL$ does not usually have an identity, we artificially associate an identity element $e$ to it, forming an LRB $\hat{\LL}$ (we will see later what is the effect of this association  on a random walk). Since $\hat{\LL}$ is an (associative) LRB with identity, we can apply Brown's theorem \cite{B} on it. Let $\hat{L'} = {\rm supp'} (\hat{\LL})$, and let $\mathcal{C}$ be the set of chambers of $\hat{\LL}$. For $X \in \hat{L'}$, let $c_x$ be the number of chambers $c\in \mathcal{C}$ such that $x \preceq c$, where $x$ is any element satisfying ${\rm supp'}(x)=X$. Then:

\begin{thm}[Brown {\cite{B}, Theorem 1}]\label{brownThm}
Let $\{w_x\}_{x \in \LL}$ be a probability distribution on $\hat{\LL}$, and let $P$ be the transition matrix of the random walk on the chambers, defined as follows:
$$ P(c,d) = \sum\limits_{x \cdot c=d}w_x \ \  \mbox{   for  } \ c,d\in \mathcal{C}$$

Then $P$ is diagonalizable and has an eigenvalue
$\lambda_X = \sum\limits_{{\rm supp'}(y) \leq X}w_y,$ for each $X \in \hat{L'}$
with multiplicity $m_X$, where
$\sum\limits_{Y \geq X}m_Y = c_X$.
Equivalently, $m_X = \sum\limits_{Y \geq X} \mu(X,Y)c_Y$,
where $\mu(\cdot,\cdot)$ is the M\"{o}bius function of the lattice $\hat{L'}$ (see the definition of the M\"{o}bius function in the appendix in Section \ref{sec-app1}).
\end{thm}

We now give an example of a CL arrangement and the computation of the eigenvalues of its associated transition matrix in order to see the significant difference between the case of CL arrangements and the case of hyperplane arrangements.

\begin{example}
\rm{ Consider the CL arrangement $\A$ consisting of a circle $C$ centered at the origin and three concurrent lines $L_1,L_2,L_3$, intersecting the circle transversally and pass through the origin (see Figure \ref{randomWalk}). This CL arrangement has twelve $2$-dimensional faces ($d_1,\dots,d_{12}$), eighteen $1$-dimensional faces ($c_1$,\ldots,$c_6$ and $s_1$,\ldots,$s_{12}$) and seven $0$-dimensional faces ($p_1$,\ldots,$p_6$ and $O$). The product on $\LL$ is defined using the geometric product, i.e. by Definition \ref{defPart2}. We have that $p_1 \sim p_4$, $p_2 \sim p_5$ and $p_3 \sim p_6$, so we get the semi-lattice $L'$. Note that since the arrangement is not central, we artificially add to it the identity element $e$, getting the lattice $\hat{L'}$, presented in Figure \ref{latticeLtag}.  The fact that the product on $\LL$ is associative can be checked directly.

\begin{figure}[h]
\epsfysize=8cm \centerline{\epsfbox{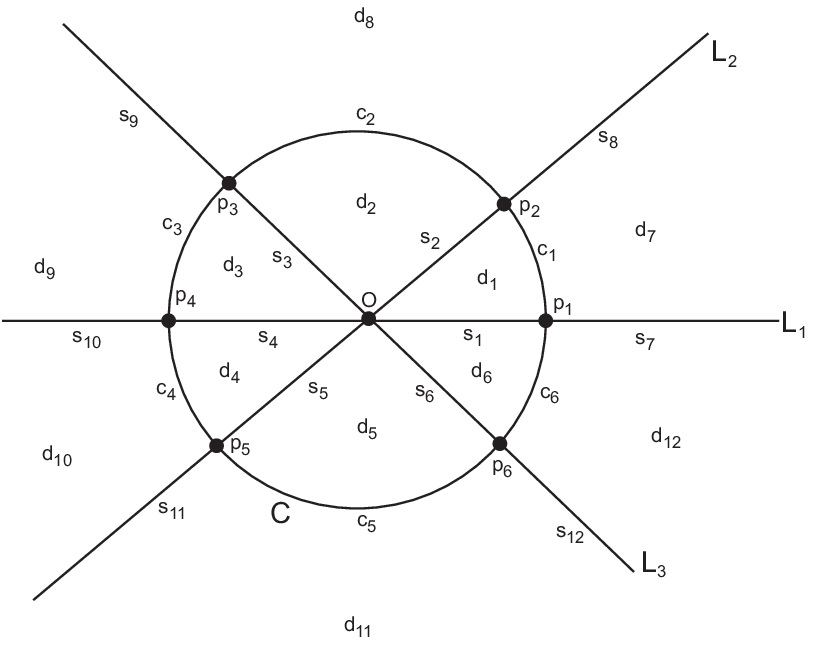}}
\caption{The faces of a given CL arrangement $\A$, for computing the eigenvalues of its associated transition matrix}\label{randomWalk}
\end{figure}

\begin{figure}[h]
\epsfysize=5cm \centerline{\epsfbox{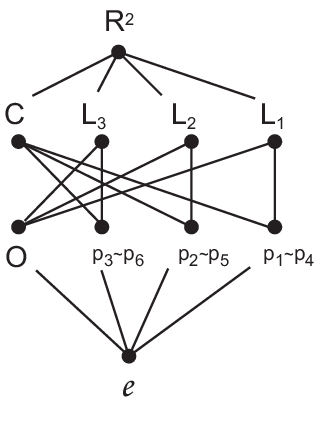}}
\caption{The lattice $\hat{L'}$ of the CL arrangement $\A$}\label{latticeLtag}
\end{figure}

Following the example of dihedral arrangements given by Brown and Diaconis \cite[Section 3A]{BD}, we assume that the measure $w$ is supported only on the 1-dimensional faces. For the example, we attach two different measures to these faces: the first measure $w_1$ will be uniform, and hence: $w_1(c_i)=w_1(s_j)=\frac{1}{18}$
where $1 \leq i \leq 6$ and $1 \leq j \leq 12$. The second measure $w_2$ will have different weights on the arcs of the circle and on the segments of the lines: fix $k>1$ and define $w_2(c_i)=\frac{1}{6k}$ where $k>1$ and $1 \leq i \leq 6$ and $w_2(s_j)=\frac{1}{12}\left( 1 -\frac{1}{k} \right)$ where $1 \leq j \leq 12$.

By Theorem 1 of Brown \cite{B} (appeared above as Theorem \ref{brownThm}), one can compute the eigenvalues and their multiplicities of the transition matrix associated to the CL arrangement and its given measure. In Table \ref{table_eigen} we summarize the eigenvalues and their multiplicities for both measures (its explanations will follow). Moreover, in the right columns we write the corresponding values of the eigenvalues (and their multiplicities) in the case of $m$ concurrent lines passing through the origin (instead of $3$ lines in the CL arrangement presented in Figure \ref{randomWalk}), and a circle  centered at the origin.

\begin{table}[!ht]
{\small \begin{tabular}{|c|c|c|c|c|c|c|}
\hline
\multirow{2}{*}{Face $X$} & \multicolumn{2}{|c|}{Eigenvalue  (w.r.t. $w_1$)} & \multicolumn{2}{|c|}{Eigenvalue (w.r.t. $w_2$)} & \multicolumn{2}{|c|}{Multiplicity} \\
\cline{2-7}
 & 3 lines & $m$ lines & 3 lines & $m$ lines & 3 lines & $m$ lines \\
\hline
\hline
&&&&&& \\ [-1em]
$X=\RR ^2$       & $\lambda_X=1$ & $\lambda_X=1$ & $\lambda_X=1$ & $\lambda_X=1$ & $m_X = 1$  & $m_X =1$ \\
&&&&&& \\ [-1em]
\hline
&&&&&& \\ [-1em]
&&&&&& \\ [-1em]
$X \in \{ L_1,L_2,L_3 \}$  & $\lambda_X=\frac{4}{18}$ & $\lambda_X=\frac{4}{6m}$ & $\lambda_X=\frac{1}{3}\cdot \left( 1 -\frac{1}{k} \right)$ & $\lambda_X=\frac{1}{m}\cdot \left( 1 -\frac{1}{k} \right)$ & $m_X = 1$ & $m_X =1$ \\
&&&&&& \\ [-1em]
&&&&&& \\ [-1em]
\hline
&&&&&& \\ [-1em]
&&&&&& \\ [-1em]
$X=C$            & $\lambda_X=\frac{1}{3}$ & $\lambda_X=\frac{1}{3}$ & $\lambda_X=\frac{1}{k}$ & $\lambda_X= \frac{1}{k}$ & $m_X =1$ & $m_X =1$ \\
&&&&&& \\ [-1em]
&&&&&& \\ [-1em]
\hline
&&&&&& \\ [-1em]
$X \in \{ p_1,p_2,p_3 \}$  & $\lambda_X=0$ & $\lambda_X=0$ & $\lambda_X=0$ & $\lambda_X=0$ & $m_X =1$ & $m_X =1$ \\
&&&&&& \\ [-1em]
\hline
&&&&&& \\ [-1em]
$X=O$            & $\lambda_X=0$ & $\lambda_X=0$ & $\lambda_X=0$ & $\lambda_X=0$ & $m_X =2$ & $m_X =m-1$ \\
&&&&&& \\ [-1em]
\hline
&&&&&& \\ [-1em]
$X=e$  & $\lambda_X=0$ & $\lambda_X=0$ & $\lambda_X=0$ & $\lambda_X=0$ & $m_X =2$ & $m_X =m-1$ \\
\hline
\end{tabular}}
\medskip
\caption{ The eigenvalues and their multiplicities for the transition matrix of the CL arrangement presented in Figure \ref{randomWalk} with respect to the given measures $w_1$ and $w_2$. In the right columns, we write the corresponding values of the eigenvalues and their multiplicities in the case of $m$ concurrent lines (instead of $3$ lines, which the corresponding values are written in the left columns) passing through the origin.}\label{table_eigen}
\end{table}

We now explain the content of the table row by row for the first measure $w_1$ (the eigenvalues for the second measure $w_2$ are computed similarly). Note that the computation of the multiplicities is independent of the choice of the specific measure, as can be seen directly from Theorem \ref{brownThm}.

If $X=\RR ^2$, all the 1-dimensional faces $F$ satisfy $F \subseteq X$ and hence $\lambda_X=1$. Also, we have $c_X=1$ (see its definition before Theorem \ref{brownThm}) and therefore $m_X=1$.

If $X=L_i$ for $1 \leq i \leq 3$, the support of the four segments of the line $L_i$ is contained in the line $L_i$ and hence $\lambda_X=\frac{4}{18}$ (since the measure of each segment is $\frac{1}{18}$). Moreover, we have $c_X=2$ and therefore $m_X=1$. By the same argument, if $X=C$, the six arcs of the circle $C$ satisfy that their support is contained in the circle $C$ and hence $\lambda_X=\frac{1}{3}$. Similarly, we have $c_X=2$ and therefore $m_X=1$.

If $X=p_i$ for $1 \leq i \leq 3$, there is no 1-dimensional face which is contained in a point, and hence $\lambda_X=0$. Moreover, we have $c_X=4$ and therefore $m_X=1$ (e.g. if $X=p_1$, we have $4=c_{p_1}=m_{\RR^2} +m_C+m_{L_1}+m_{p_1}=3+m_{p_1}$, which implies that $m_{p_1}=1$). By the same argument, if $X=O$, then $\lambda_X=0$. As before, we have $c_X=6$ and therefore $m_X=2$ (since we have $6=c_O=m_{\RR^2}+m_{L_1}+m_{L_2}+m_{L_3}+m_O=4+m_O$, which implies that $m_O=2$).

If $X=e$, there is no 1-dimensional face which is contained in $e$, and hence $\lambda_X=0$. Moreover, we have $c_X=12$ (since there are twelve $2$-dimensional faces ($d_1$,\ldots,$d_{12}$) in the CL arrangement). Therefore, we have:
$$12=c_e=m_{\RR^2} +m_C+m_{L_1}+m_{L_2}+m_{L_3}+m_{p_1}+m_{p_2}+m_{p_3}+m_O+m_e=10+m_e,$$
which implies that $m_e=2$.

\medskip

The last computation shows a significant difference between line (and hyperplane) arrangements and CL arrangements: in the case of non-central line arrangements (i.e. not all the lines intersect in the same single point), the identity element $e$ does not contribute any eigenvalue (i.e. the multiplicity of the eigenvalue associated to the element $e$ is $0$), see \cite[p. 881, Example 1]{B}. On the other hand, as is shown in the last computation,  in a CL arrangement whose associated LRB of faces has no identity element, the identity element $e$ does contribute the eigenvalue $0$ in a non-zero multiplicity.
}
\end{example}

\subsubsection{Stereographic projection} \label{subsec_strProj}

In this subsection, we investigate the relation between hyperplane arrangements in $\RR^3$ and CL arrangements in $\RR^2$.

We start by describing the transition between a hyperplane arrangement $\B$ in $\RR^3$ and a CL arrangement $\A_\B$ in $\RR^2$. Let $S$ be a
unit sphere in $\RR^3$,  centered at $(0,0,-1)$. Let $O = (0,0,0)$ and let $\pi = \{z = -2\}$ be a plane tangent to $S$. Consider the stereographic projection
$p_O : S \to \pi$, centered at  $O$. Using $p_O$, $\B$ induces a CL arrangement on $\pi$, by setting $\A_\B \doteq p_O(\B \cap S)$. In this way, every plane that belongs to $\B$, which passes through $O$ and is not tangent to $S$, will induce under $p_O$ a line in $\A_\B$, and every plane that belongs to $\B$, which does not pass through $O$, intersects $S$ but is not tangent to it, will induce under $p_O$ an ellipse in $\A_\B$.

Note that we exclude hyperplane arrangements in $\RR^3$ that have planes that are tangent to $S$ or that do not intersect $S$, since in that case, the resulting projection would be either a point or an empty set, which are excluded by the definition of a real CL arrangement (see Definition \ref{defRealCLarr}).

\begin{prs}
(1) Every line arrangement in the plane $\pi$ can be induced by a central plane arrangement in $\RR^3$ via $p_O$.\\
(2) Not every CL arrangement in $\pi$ can be induced by a  plane arrangement in $\RR^3$ via $p_O$.
\end{prs}

\begin{proof}
(1) The corresponding plane for every line $\ell$ in the line arrangement in $\pi$ is the plane passing through $\ell$ and $O$.\\
(2) Indeed, a parabola cannot be induced in $\A_\B$ by $p_O$.
\end{proof}

Now we discuss the connections between the LRBs: Let $\B$ be a central hyperplane arrangement in $\RR^3$ with no planes tangent to $S$,  and $\A_\B = p_O(\B \cap S)$ be the corresponding CL arrangement in $\RR^2$. A natural question is:
what is the connection between the LRBs $\LL(\B)$ and $\LL(\A_\B)$? Even
when $\LL(\A_\B)$ is an associative LRB or when $\B$ is a central hyperplane arrangement, the map $p_O$ does not induce a homomorphism. For example, it is known that there is only one combinatorially-equivalent central plane arrangement with three planes. Choose two such arrangements $\B_1, \B_2$, such that $\B_1 \cap S$ consists of only great circles, while $\B_2 \cap S$ consists of two great circles and a conic section which is not a great circle. Then $\A_{\B_1}$ is a central line arrangement with three lines, while $\A_{\B_2}$ is a generic line arrangement with three lines. Thus, $\LL(\B_1) \simeq \LL(\B_2)$; however, $\LL(\A_{\B_1}) \not\simeq \LL(\A_{\B_2})$.

Note that although $p_O$ does not induce a homomorphism in the level of the LRBs, it does help to distinguish between hyperplane arrangements whose face LRBs are isomorphic (an example for that is the above arrangements $\B_1, \B_2$). Hence, one may pose the following question:

\begin{question}
Are there any other invariants that distinguish between $\B_1$ and $\B_2$? Can the above method help us to distinguish between hyperplane arrangements with isomorphic face LRBs?
\end{question}

\subsubsection{Non-geometric LRBs coming from CL arrangements} \label{subsecLRB_cl}

In this subsection, we present an example of an LRB, being the face LRB of a CL arrangement, which cannot be embedded in $(L_2^1)^n$ for any $n \in \N$. This immediately implies
that this LRB is not isomorphic to the face LRB of any hyperplane arrangement, and hence the family of face LRBs associated to CL arrangements is {\em broader} than the corresponding family of LRBs associated to hyperplane arrangements.

\begin{example}\label{exmLRBnonEmb}
\rm{
Consider the real CL arrangement $\mathcal A$ which consists of a line and a circle tangent to it (see Figure \ref{NotH21LRB}). This arrangement has 7 faces, and we denote the two 1-dimensional sections (i.e. faces) of the line by $b$ and $a$, the circle by $c$, and the 2-dimensional face below the line and outside the circle by $d$. Let $e$ be the tangency point.
As usual, denote $\LL =\LL(\A)$.

\begin{figure}[h]
\epsfysize=2cm \centerline{\epsfbox{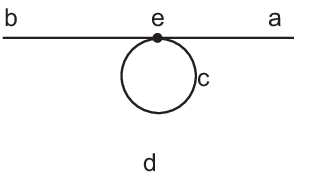}}
\caption{An example of a CL arrangement whose LRB is not geometric (i.e. cannot be embedded in $(L^1_2)^n$ for any $n$)}\label{NotH21LRB}
\end{figure}

The product on $\LL$ is defined using the geometric product (Definition \ref{defPart2}). Note that $e$ is the unit element and one can check that $(\LL,\cdot)$ is indeed an LRB. We have the following  multiplication table for $\{a,b,c\}$:

$$\begin{array}{c||c|c|c}
\cdot & a & b & c \\
\hline\hline
a     & a & a & d \\
\hline
b     & b & b & d \\
\hline
c     & d & d & c \\
\end{array}$$

\medskip

Assume that we have a\emph{ monomorphism}  $h:\LL \to (L_2^1)^n$ for some $n$. Note that the equalities
$$a \cdot b~=~a,\qquad b \cdot a~=~b$$
imply that $h(a)$ and $h(b)$ have zeros in the same coordinates. Indeed, if $(h(a))_j=0$, then  $(h(a))_j(h(b))_j =  (h(b))_j$ (by the multiplication laws in $L_2^1$), but since $(h(a))_j (h(b))_j = (h(a))_j = 0$ (the left equality is by the homomorphism), so we have $(h(b))_j = 0$. By the same argument, using the second equality, we get that $(h(b))_j=0$ implies that $(h(a))_j=0$.

Since $a \neq b$, then $h(a) \neq h(b)$, which means that there exists a coordinate $j$, $1 \leq j \leq n$, such that $(h(a))_j \neq (h(b))_j$
and both  coordinates are not zero (so without loss of generality, one is $+$ and the other is $-$).
But:
$$(h(d))_j \stackrel{d = a \cdot c}{=} (h(a \cdot c))_j = (h(a))_j(h(c))_j \stackrel{(h(a))_j \neq 0}{=} (h(a))_j   \neq $$
$$\neq (h(b))_j \stackrel{(h(b))_j \neq 0}{=} (h(b))_j(h(c))_j = (h(b \cdot c))_j \stackrel{b \cdot c =d}{=}  (h(d))_j,$$
by the multiplication laws  in $L_2^1$, which is a contradiction.}
\end{example}

\medskip

Therefore, $\LL=\LL(\A)$ with the geometric product (Definition \ref{defPart2}) is an example of an LRB which is not \emph{geometric} (i.e. it cannot be embedded in $(L_2^1)^n$ for any $n$, see \cite[Section 3.7]{MSS}).

\medskip

Moreover, note that for the LRB $\LL_0(\mathcal A)$ (which is contained in $(L_2^1)^2$), one cannot find a hyperplane arrangement $\mathcal A' \subset \RR^N$ such that
$\LL_0(\mathcal A) \simeq \LL(\mathcal A') \simeq \LL_0(\mathcal A')$. Indeed, $\LL_0(\mathcal A)$ has 6 elements, has a unit $i(e)=(0,0) \in (L_2^1)^2$ and the three  elements $(+,+),(-,-),(-,+)$ form the unique two-sided ideal of $\LL_0(\mathcal A)$. Thus, if such a hyperplane arrangement $\mathcal A'$ exists, it should be a central hyperplane arrangement with three chambers, which is impossible.

\section{Plane curve arrangements: structure of sub-LRBs} \label{secStrSubLRB}

Let $\A$ be an arrangement of curves $\{H_1, \ldots, H_m\}$ in $\RR^2$, where $H_i$ is defined by $\{ f_i=0 \}$ and $f_i \in \RR[x,y]$.
In this section, we study the explicit structure of sub-LRBs of $\LL_0(\A)$ induced by the embedding of a component $H$ into the arrangement $\A$. The motivation for this study comes from the global interest in combinatorics and algebra in general, and in arrangements in particular, in order to understand the connections between the algebraic structure (the LRB structure, in our context) of the whole arrangement and its induced arrangements (the deleted arrangement and the restricted arrangement). Moreover, in line arrangements, $\LL_0(\A) \simeq \LL (\A)$, and hence it is interesting to study the structure of $\LL_0(\A)$ and its sub-LRBs for a CL  arrangement $\A$, despite the fact that $\LL_0(\A) \not\simeq \LL (\A)$ for almost any CL arrangement $\A$.

For a line arrangement $\A$, let $H \doteq H_i \in \A$ be a given line. As before, recall that $\A^H = \A - \{H\}$ is the \emph{deleted arrangement} in $\RR^2$ and $\A_H = \left\{ K \cap H \ |\ K \in \A^H\right\}$ is the \emph{restricted arrangement} contained in $H$. Thus, one can define two associated LRBs: the {\it deleted LRB} $\LL_0 \left( \A^H \right)$, corresponding to the deleted arrangement $\A-\{H\}$ in $\RR^2$, and the {\it restricted LRB} $\LL_0 \left( \A_H \right)$, corresponding to the restriction of the arrangement $\A$ to $H$. Obviously, $\LL_0 \left( \A^H \right)$ is obtained from $\LL_0(\A)$ by deleting the $i^\text{th}$ coordinate. However, the question is how $\LL_0 \left( \A_H \right)$ is embedded in $\LL_0(\A)$. We answer this question for line arrangements in Section \ref{subsec-3.1}. After defining the notion of an LRB of a pointed curve in Section \ref{subsec-3.2}, we can answer this question also for CL arrangements in Section \ref{subsecEmbedCL}.

\medskip

Given an arrangement $\A$ of $m$ smooth real curves (we require that the curves will be smooth for preventing self-intersections), where every curve has only one component in $\RR^2$, one can associate a vector of signs in $(\{+,-,0\})^m$ to any face of the
arrangement, which describes the mutual position of this face with respect to the curves $H_i$, where $H_i$ is  defined by $\{ f_i=0 \}$, $1 \leq i \leq m$. Explicitly, as before, one can associate to $\A$ a subset $\LL_0(\A)$ of $(L_2^1)^m$ induced by these vectors of signs.

For a given $i \in \{1,\dots,m\}$, which is the index of $H=H_i$ in the arrangement $\A$, define:
$$\LL_0(\A)|_H \doteq \{ x \in \LL_0(\A) : (x)_i = 0\} \subset (L_2^1)^m.$$
$\LL_0(\A)|_H$ is a sub-LRB of $\LL_0(\A)$, to which corresponds the restricted
arrangement $\A_H$ as a sub-LRB, since it is a subset of $\LL_0(\A)$ and thus the associativity and the LRB properties: $x^2=x$ and $x \cdot y \cdot x=x \cdot y$ are immediately satisfied. The closure under the product is obvious. Note that $ \sharp \LL_0(\A_H) = \sharp (\LL_0(\A)|_H)$.

The question is: what are the connections between the LRBs $\LL_0(\A_H)$ and $\LL_0(\A)|_H$? This question is manageable once one defines a structure of an LRB on $\A_H$ as will be done in Definitions \ref{defLRBunbounded} and \ref{defLRBbounded}.
Note that when either $H$ is a bounded component or an unbounded one, $\A_H$ is a collection of $k$ points $\{p_1,\ldots,p_k\}$ in $H$.

\medskip

\emph{Note}: From now on, we  assume that $\LL_0(\A)$ is an LRB, i.e. it is closed under the product induced by $(L_2^1)^m$ (see Section \ref{subsec_semiCL} above). Moreover, to simplify notations, we assume that each component  $H_i$ is connected in $\RR^2$, where $H_i$ is defined by $\{f_i = 0\}$ (i.e. $H_i$ can not be an hyperbola).

\begin{remark}\label{remNotationI}
{\rm Note that $\LL_0(\A_H) \subseteq (L^1_2)^k$ and $\LL_0(\A)|_H \subseteq (L^1_2)^m$. In order to distinguish between the different vectors of
signs when we refer to a face, which can be thought of both as a face in $\LL(\A_H)$ and in $\LL(\A)|_H \subseteq \LL(\A)$ (where $\LL(\A)|_H$ is the set of faces of $\LL(\A)$ contained in $H$), we denote:
$$i_\A: \LL(\A)|_H \to (L^1_2)^m,\,\, {\rm{Image}} (i_\A) = \LL_0(\A)|_H$$
and
$$i_H: \LL(\A_H) \to (L^1_2)^k,\,\, {\rm{Image}} (i_H) = \LL_0(\A_H),$$
where both maps describe the vectors of
signs in $\LL_0(\A)|_H$ (resp. $\LL_0(\A_H)$) of a face in $\LL(\A)|_H$ (resp. $\LL(\A_H)$). See the exact definition of the map $i_H$ in Definitions \ref{defLRBunbounded} and \ref{defLRBbounded}.}
\end{remark}

\subsection{Preliminaries: The embedding principle for the face LRB of line arrangements}\label{subsec-3.1}

In this section, we present the embedding principle for the face LRB of line arrangements in $\RR^2$ (which can be easily generalized to hyperplane arrangements), i.e. we present the connections between $\LL_0(\A_H)$ and $\LL_0(\A)|_H \subset \LL_0(\A)$ for a line arrangement $\A$ and $H \in \A$. This is done as a preparation for Proposition \ref{prsEmbed}, which deals with the embedding principle for arrangements of smooth real curves in $\RR^2$.

\begin{lemma} \label{lemLineArrLRB}
Let $\A = \{H_1,\ldots,H_m\}$ be an arrangement of lines in $\RR^2$, where $H_i$ is defined by $\{f_i=0\}$ and $f_i \in \RR[x,y]$. Denote $H = H_1$ and let $ \{ H \cap H_i \ | \ 1 \leq i \leq m \} = \{p_1,\dots,p_k\} \subset H$ be $k$ points. Then, there is an isomorphism of LRBs:
$$\varphi : \LL_0(\A_H)  \stackrel{\sim}{\to} \LL_0(\A)|_H \subseteq (L_2^1)^m,$$
satisfying the following properties:
\begin{enumerate}
\item $(\varphi(\LL_0(\A_H)))_1 = 0.$
\item For every $j>1$:

\noindent
(a)  If $H \cap H_j = \emptyset$, then $(\varphi(\LL_0(\A_H)))_j $ is constant (either $+$ or $-$, depending on the mutual position of the parallel lines $H$ and $H_j$). Explicitly, all the vectors in $\varphi(\LL_0(\A_H))$ have the same sign in their $j^{\rm th}$ coordinate.

\noindent
(b)  If $H \cap H_j = \{p_s\}$ for some $1 \leq s \leq k$, then $(\varphi(\LL_0(\A_H)))_j = a\cdot (\LL_0(\A_H))_s$, where $a\in \{\pm 1\}$.

\end{enumerate}

\end{lemma}

Before the proof, we illustrate the result of the above lemma by an example.

\begin{example}\label{Exam3.3}
\begin{figure}[h]
\epsfysize=4.5cm \centerline{\epsfbox{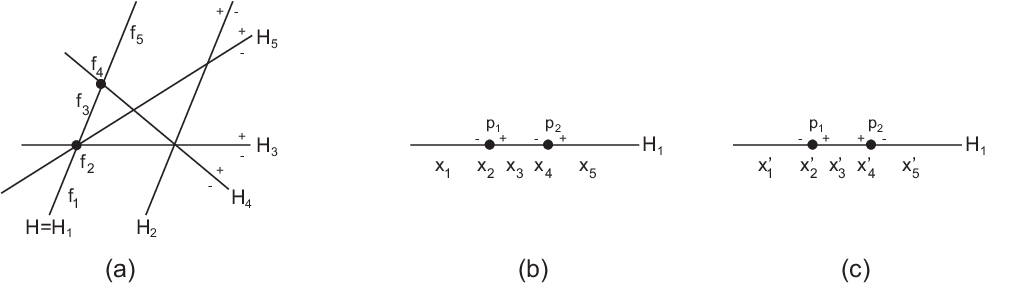}}
\caption{An example for illustrating the result of Lemma \ref{lemLineArrLRB}: Part (a) presents the line arrangement $\A$. $f_i$ are the faces contained in $H_1$ in the face set $\LL(\A)$. Parts (b) and (c) present two arrangements of points on a line, which can be thought of as the restricted arrangement $\A_{H_1}=\A_H$. The faces in the face set $\LL_0(\A_{H_1})$ are $x_i$ in part (b) (and $x'_i$ in part (c)).}\label{ExamLineEmb}
\end{figure}
{\rm
Figure \ref{ExamLineEmb}(a) presents a line arrangement $\A$  and Figures \ref{ExamLineEmb}(b) and \ref{ExamLineEmb}(c) present two  arrangements of points on a line, both can be thought of as the restricted arrangement $\A_{H_1}=\A_H$. Note that the difference between the arrangements  appearing in Figures \ref{ExamLineEmb}(b) and \ref{ExamLineEmb}(c) is that the signs assigned with respect to the point $p_2$ are opposite.

\begin{enumerate}
\item Considering the point arrangement in Figure \ref{ExamLineEmb}(b), the faces of $\A_{H_1}$  are denoted by $x_1,\dots,x_5$; their corresponding images by $\varphi$, i.e. these faces in the arrangement $\A$, are denoted by $f_1,\dots,f_5$ (resp.). Let $H = H_1$. Then, the corresponding LRBs are
$$\LL_0(\A_H) = \left\{
\begin{array}{c}
i_H(x_1)=(-,-),\ i_H(x_2)= (0,-),\  i_H(x_3)= (+,-),\\
i_H(x_4)= (+,0),\ i_H(x_5)= (+,+)
\end{array} \right\},$$
and in a table form in Table \ref{table_A1},
\begin{table}[!ht]
$$\begin{array}{|c||c|c|}
\hline
 & (i_H(x_i))_1 & (i_H(x_i))_2 \\
\hline
\hline
i_H(x_1) & - & - \\
i_H(x_2) & 0 & - \\
i_H(x_3) & + & - \\
i_H(x_4) & + & 0 \\
i_H(x_5) & + & + \\
\hline
\end{array}$$
\caption{$\LL_0(\A_H)$ in a table form}\label{table_A1}
\end{table}
and
$$\LL_0(\A)|_H = \varphi(\LL_0(\A_H)) =
\left\{
\begin{array}{c}
i_{\A}(f_1)=(0,+ ,-,-,-),\ i_{\A}(f_2)= (0,+,0,-,0),\\
i_{\A}(f_3)= (0,+ ,+,-,+ ),\ i_{\A}(f_4)= (0,+ ,+,0,+),\\
i_{\A}(f_5)= (0,+ ,+, +,+ )
\end{array} \right\},$$
and in a table form in Table \ref{table_A2}.
\begin{table}[!ht]
$$\begin{array}{|c||c|c|c|c|c|}
\hline
 & (i_{\A}(f_j))_1 & (i_{\A}(f_j))_2 & (i_{\A}(f_j))_3 & (i_{\A}(f_j))_4 & (i_{\A}(f_j))_5 \\
\hline
\hline
i_{\A}(f_1) & 0 & + & - & - & - \\
i_{\A}(f_2) & 0 & + & 0 & - & 0 \\
i_{\A}(f_3) & 0 & + & + & - & + \\
i_{\A}(f_4) & 0 & + & + & 0 & + \\
i_{\A}(f_5) & 0 & + & + & + & + \\
\hline
\hline
 & 0 & + & (i_H(x_i))_1 & (i_H(x_i))_2 & (i_H(x_i))_1 \\
\hline
\end{array}$$
\caption{$\LL_0(\A)|_H$ in a table form, where in the last row, the relations to
$\LL_0 (\A_H)$ are presented.}\label{table_A2}
\end{table}

\medskip

We deal with $\varphi(\LL_0(\A_H))$ coordinate by coordinate:
\begin{enumerate}
\item First, note that $(\varphi(\LL_0(\A_H)))_1  = 0$ (property (1) of the lemma). In the table form, all the values in the first column of Table \ref{table_A2} are $0$.
\item Since $H_2 \cap H = \emptyset$,  $(\varphi(\LL_0(\A_H)))_2  = +$, i.e., by property (2)(a) of the lemma, the second coordinate in all the vectors of $\varphi(\LL_0(\A_H))$ is $+$. In the table form, all the values in the second column of Table \ref{table_A2} are $+$.
\item Since $H_3 \cap H  = H_5 \cap H = \{p_1\}$, we have:
    $$(\varphi(\LL_0(\A_H)))_3  = (\varphi(\LL_0(\A_H)))_5  = (\LL_0(\A_H))_1$$
    (by property (2)(b) of the lemma, where in this case $a = +$).
    In the table form, the third and fifth columns of Table  \ref{table_A2} are equal to the first column of Table \ref{table_A1}.
\item Since $H_4 \cap H  = \{p_2\}$, we have: $(\varphi(\LL_0(\A_H)))_4  = (\LL_0(\A_H))_2$ (again by property (2)(b) of the lemma, where in this case $a = +$). In the table form, the fourth column of Table  \ref{table_A2} is equal to the second column of Table \ref{table_A1}.

\end{enumerate}

\medskip

\item
Considering the point arrangement in Figure \ref{ExamLineEmb}(c), the faces of $\A_H$  are denoted by $x'_1,\dots,x'_5$. In this case, we have:
$$\LL_0(\A_H) = \left\{
\begin{array}{c}
i_H(x'_1)=(-,+),\ i_H(x'_2)= (0,+),\ i_H(x'_3)= (+,+),\\
i_H(x'_4)= (+,0),\ i_H(x'_5)= (+,-)
\end{array} \right\},$$
and in a table form in Table \ref{table_A3} (note that $(i_H(x'_i))_2 = -(i_H(x_i))_2$ for all $1 \leq i \leq 5$ due to the change in the assignment of signs with respect to the point $p_2$).
\begin{table}[!ht]
$$\begin{array}{|c||c|c|}
\hline
 & (i_H(x'_i))_1 & (i_H(x'_i))_2 \\
\hline
\hline
i_H(x'_1) & - & + \\
i_H(x'_2) & 0 & + \\
i_H(x'_3) & + & + \\
i_H(x'_4) & + & 0 \\
i_H(x'_5) & + & - \\
\hline
\end{array}$$
\caption{$\LL_0(\A_H)$ in a table form}\label{table_A3}
\end{table}

As before, since $H_3 \cap H  = H_5 \cap H = \{p_1\}$, we have:
$$(\varphi(\LL_0(\A_H)))_3  = (\varphi(\LL_0(\A_H)))_5  = (\LL_0(\A_H))_1,$$
by property (2)(b) of the lemma, where in this case $a = +$. In the table form, the third and fifth columns of Table  \ref{table_A2} are equal to the first column of Table \ref{table_A3}.

On the other hand, as $H_4 \cap H  = \{p_2\}$, we have: $(\varphi(\LL_0(\A_H)))_4  = -(\LL_0(\A_H))_2$, by property (2)(a) of the lemma, but in this case $a = -$. Explicitly, in contrast to Example \ref{Exam3.3}(1)(d) above, in order to obtain $(\varphi(\LL_0(\A_H)))_4$, one has to multiply all the values in $(\LL_0(\A_H))_2$ by the scalar ($-1$). In the table form, one has to multiply the second column of Table \ref{table_A3} by the scalar ($-1$) in order to get the fourth column of Table \ref{table_A2}.
\end{enumerate}
}
\end{example}

\begin{proof}[Proof of Lemma \ref{lemLineArrLRB}]
We start by proving that $\varphi$ is an isomorphism. First we show that $\varphi$ is a homomorphism. We will prove coordinatewise, i.e. for each $i$ and for each $x,y \in \LL_0(\A_H)$ we have that
$(\varphi(xy))_i = (\varphi(x))_i (\varphi(y))_i$. For $i=1$, both sides indeed equal $0$.
For $i>1$ and $H \cap H_i = \emptyset$, then for every $z \in \LL_0(\A_H)$, $(\varphi(z))_i$ is a constant $c$ (either $+$ or $-$), and therefore we have that $(\varphi(x))_i (\varphi(y))_i = c \cdot c = c = (\varphi(xy))_i$ by the idempotency property of $a \in L_2^1$.
Otherwise (when $H \cap H_i = \{p_s\}$), we use again the idempotency property of $a \in L_2^1$, i.e.
$+ \cdot + = +$ and $- \cdot - = -$, to get:
$$(\varphi(xy))_i = a \cdot (xy)_s = a^2 \cdot (xy)_s = a \cdot a \cdot (x)_s \cdot (y)_s \stackrel{(*)}{=} a \cdot (x)_s \cdot a \cdot (y)_s = (\varphi(x))_i (\varphi(y))_i,$$
where equality $(*)$ holds due to the fact that $a \neq 0$ and the properties of the multiplication in $L_2^1$.

Next, $\varphi$ is injective, since if there were two faces in $\LL_0(\A_H)$ sent to the same face in $\LL_0(\A)|_H$ that would have meant that these two different faces in $H$ have the same mutual position with respect to all the other lines $H_2,\dots,H_m$, which is impossible. By the fact that $\sharp (\LL_0(\A)|_H) = \sharp \LL_0(\A_H)$, we get that $\varphi$ is isomorphism.

\medskip

Now we pass to the proof of the two properties of $\varphi$. We use the notations introduced in Remark \ref{remNotationI}. Property (1) is obvious (since we are in $H_1$).

For property (2)(a), note that if $H \cap H_j = \emptyset$ where $H_j$ is defined by $\{f_j=0\}$, then the line $H_j$ is parallel to $H$ and all the faces of $\A$ which lays in $H$ are either in the halfplane $\{f_j>0\}$ (in this case $(\varphi(\LL_0(\A_H)))_j  = +$) or in  $\{f_j<0\}$ (in this case $(\varphi(\LL_0(\A_H)))_j  = -$).

As for property (2)(b), assume that $H \cap H_j =\{p_s\}$ for some
$1 \leq s \leq k$ and let $c$ be a face of $\A$ which lays in $H$.
As $c$ goes over all the faces which lay in $H$, it goes over all the set $\rm Image(\varphi)$.
Then, either $c \subset \{f_j>0\}$, $c \subset \{f_j<0\}$ or $c \subset \{f_j=0\}$.
In the third case, $c = p_s \in H$ and thus $(i_H(c))_s=0$ and as $c \in H_j$, $(i_\A(c))_j = 0$. As for the first two cases, the fact that $c$ is in  one of the two halfplanes is determined by the position of $H$ with respect to $H_j$ (as $c \subset H$), which is reduced to checking if $c$ is located to the right of $\{p_s\} = H \cap H_j$ or to its left. Therefore, up to a constant scalar multiplication by $a \in \{\pm 1\}$ (for all faces $c$ which lay in $H$), $c \subset \{f_j>0\}$
is equivalent to the fact that $c$ is to the right of the point $p_s$. The (constant) scalar multiplication is needed, since a priori there is no connection between the sign in $\LL_0(\A_H)$ that is assigned to the faces to the right of $p_s$ and the sign in $\LL_0(\A)$ assigned to these faces in the halfplane above $H_j$ (see Figures  \ref{ExamLineEmb}(b) and \ref{ExamLineEmb}(c) for an example of two different assignments of signs for $\A_H$).
\end{proof}

\begin{remark}
{\rm Note that there is a natural assignment of signs for the elements in $\A_H$, that is induced by $\LL_0(\A)$, in the following way: if the sign in $\LL_0(\A)$ that is given to the halfplane $\{f_j>0\}$ is $+$ (in the $j^{\rm th}$ coordinate, where $H_j$ is defined by $\{f_j=0\}$ and $H \cap H_j =\{p_s\}$), and the section
$\{ x > p_s\}$ on the line $H$ is contained in $\{f_j>0\}$, then the sign in $\LL_0(\A_H)$ associated to $\{ x > p_s\}$  (i.e. in the $s^{\rm th}$ coordinate) will also be $+$. If $\{ x < p_s\} \subset \{f_j>0\} $, then the sign associated to the faces contained in this section will be $-$. In the case of this natural assignment of signs to $\A_H$, the scalar multiplication in property (2)(b) of Lemma \ref{lemLineArrLRB} is not needed (i.e. the scalar $a$ will always be $+$).

However, note that Lemma \ref{lemLineArrLRB} is more general, as we do not assume any a priori connection between the signs associated to the halfplanes in $\A$ and the signs associated to the half-lines in $\A_H$.}
\end{remark}

\begin{remark}
{\rm Note that the isomorphism of $\LL_0(\A_H)$ to $\LL_0(\A)|_H$ can also be described in the language of oriented matroids. Explicitly, recall that the map $z$ taking a vector of signs (i.e. an element in $\LL_0(\A)$) to its set of zero indices is a semigroup homomorphism to the power set equipped with intersection (see e.g. \cite[Proposition 4.1.13]{matroids}).  The image $L$ is the geometric lattice associated to the oriented matroid, which can be used to construct the matroid.  If $H$ is one of the hyperplanes, then the subposet of $L$, which contains those sets of indices consisting (also) of the index of $H$, is the geometric lattice associated to $\LL_0(\A)|_H$ , i.e. to the restriction of the arrangement to $H$. The map from the associated lattice to $\LL_0(\A)|_H$ to the associated lattice to $\LL_0(\A_H)$ is via the forgetting of the index corresponding to $H$. General results about matroids and their contractions can be found in \cite{matroids}.  }
\end{remark}

\subsection{The structure of the face LRB of a real pointed curve}\label{subsec-3.2}
Before dealing with the embedding principle in the general case of arrangements of smooth curves (Section \ref{subsecEmbedCL}), one has to consider two cases with respect to the structure of the induced face LRB of a real pointed curve $\A_H$ (i.e. an arrangement of points on a real curve $H$): where $H$ is an unbounded component and where $H$ is a bounded one. In this section, we study the structure of the face LRB of a real pointed curve in these two cases.

Given an arrangement of smooth curves $\A$ in $\RR^2$ and a connected component $H \in \A$,
the restricted arrangement $\A_H$ will be the real curve $H$ with points on it corresponding to the intersection points of the deleted arrangement $\A-\{H\}$ with the component $H$, i.e. we get an arrangement of real points on a connected component $H$.

\medskip

We start with the case of an unbounded component:

\begin{definition}[The face LRB structure of an unbounded component] \label{defLRBunbounded} \ \\
Let $H \subset \RR^2$ be an unbounded smooth connected real plane curve (with no self-intersections). Let $\{p_1,\ldots,p_k\}$ be a collection of points on $H$ and let $\LL(H)$ be the set of faces of $H$ with respect to these points. Explicitly, the faces are the points themselves and the open sections of the curve that are bounded by the points (by either one or two points).

Each point $p_j \in H$, for $1 \leq j \leq k$, divides the curve $H$ into three disjoint parts: the point itself and two open sections: $H_{j,1}$ and $H_{j,2}$, such that $H = \{p_j\} \cup H_{j,1} \cup H_{j,2}$.
Associate to the set $H_{j,1}$ the sign $+$, to the set $H_{j,2}$ the sign $-$ and to the set $\{p_j\}$ the sign $0$. Obviously, one can rename the set $H_{j,1}$ as $H_{j,2}$ and $H_{j,2}$ as $H_{j,1}$ and thus induce a different assignment of signs, but once we assign these signs for each set, they are fixed.

For each face $P \in \LL(H)$, we associate  an element $i_H(P)$ in $(L^1_2)^k$, that is, a vector of signs, in the following way: for each $j$, $1 \leq j \leq k$, if $P \subseteq H_{j,1}$, then $(i_H(P))_j = +$; otherwise, if $P \subseteq H_{j,2},$ then $(i_H(P))_j = -$; otherwise, we have that $P = \{ p_j\},$ and then $(i_H(P))_j = 0$.
\end{definition}

\begin{example}
Let $H = \{y=x^2\},\ p_1=(0,0)$ and $p_2 = (2,4)$. Let $p_1,p_2,X_1,X_2,X_3$ be the faces of $\LL(H)$ (see Figure \ref{examLRBunbound}(a)).
We set the signs of $H_{i,j}\,,  i,j \in \{1,2\}$, according to Figure \ref{examLRBunbound}(b). Thus:
$$i_H(p_1) = (0,-),\, i_H(p_2) = (+,0),\, i_H(X_1) = (+,+),\, i_H(X_2) = (+,-),\, i_H(X_3) = (-,-).$$

\begin{figure}[h]
\epsfysize=3.7cm \centerline{\epsfbox{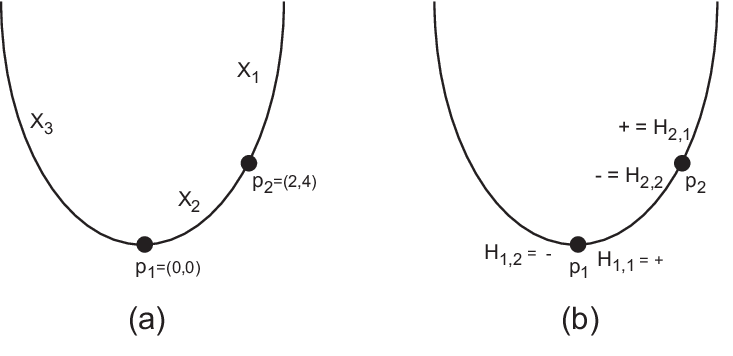}}
\caption{An example for illustrating Definition \ref{defLRBunbounded}: the face LRB structure of an unbounded pointed component.}\label{examLRBunbound}
\end{figure}

\end{example}

In this way, we get a monomorphic map $i_H: \LL(H) \to (L^1_2)^k$ and we can identify $\LL(H)$ with its image $i_H(\LL(H)) \subseteq (L^1_2)^k$.

\medskip

We have the following lemma:
\begin{lemma} \label{lemLRBwellDef}
(1) The set $i_H(\LL(H))$ is closed under the product induced by the LRB $(L^1_2)^k$, so it is an LRB as well.\\
(2) Different assignments of signs to  $H_{j,1}$ and $H_{j,2}$ (as described above) induce isomorphic LRBs.
\end{lemma}

\begin{proof}
(1) The set $i_H(\LL(H))$ is closed under the product induced by the LRB $(L^1_2)^k$, since $H$ is topologically equivalent to a line, and the assignment of the vectors of signs to  $\LL(H)$ is thus equivalent to associating an LRB structure to the set of faces of a pointed line, as a special case of a hyperplane arrangement (as described in Section \ref{subsecLRB_Hyp}).

\medskip

\noindent
(2) Since $H$ is topologically equivalent to a line, different assignments of signs to  $H_{j,1}$ and $H_{j,2}$, will induce isomorphic LRBs, by Remark \ref{remDifSignIsoLRB}.
\end{proof}

\medskip

We proceed to the case of a bounded component. If $H$ is a smooth bounded component in $\RR^2$, i.e. an oval, we can consider  an arrangement of points $\{p_1,\ldots,p_k\}$ on $H$ and look at the corresponding set of faces $\LL(H)$. The component $H$ will be later called {\it a smooth pointed (bounded) oval}. We cannot treat $\LL(H)$ as in the former case, since there is no meaning to the phrase
``every point  divides the curve $H$ into three disjoint parts'', when we are on an oval. We introduce here an alternative way to associate an LRB structure to $\LL(H)$.

\begin{definition}[The face LRB structure of a bounded component] \label{defLRBbounded} \ \\
Let  $C$ be a smooth pointed (bounded) oval, where $\{p_1,\ldots,p_k\}$ is the set of $k$ points on it numerated consecutively clockwise. As can easily be seen, the set of faces $\LL(C)$ contains $2k$ faces: $k$ points $\{p_1,\ldots,p_k\}$ and  $k$ sections of the oval that are bounded by the points. Let $p_1'$ be a point to the left of $p_1$ which is infinitesimally-close to $p_1$ (see Figure \ref{LRBconic}(a)),\footnote{\ Note that we need two different points $p_1,p_1'$, since they have different roles: $p_1$ will be a face in $\LL(C)$ and $p_1'$ is a cutting point of $C$ and will not be considered later as a face in $\LL(C)$.} and let $C_1 = C -\{p_1'\}$. $C_1$ is topologically equivalent to an open segment $S = (a_1',a_1'')$, that is, there exists a distance-preserving homeomorphism $f : C_1  \rightarrow S$. Denote $f(p_i) = a_i$ for $1 \leq i \leq k$. Explicitly, $C_1$ can be considered as a straight segment $S$ that starts at the point $a_1'$ and ends at the point $a_1''$, where its two ends $a_1',a_1''$ are identified by $f^{-1}$ with $p_1'$ in $C$ (see Figure \ref{LRBconic}(b)).

\begin{figure}[h]
\epsfysize=3.5cm \centerline{\epsfbox{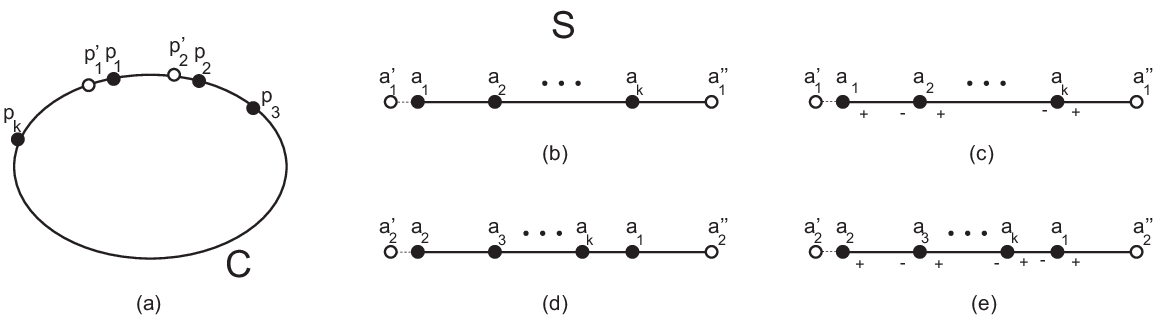}}
\caption{The LRB structure associated to an arrangement of points on an oval: Part (a) presents the pointed oval. In parts (b) and (c) we cut the oval at the point $p_1'$, and we ignore the section $(a_1',a_1)$. In parts (d) and (e) we cut the oval at the point $p_2'$, and we ignore the section $(a_2',a_2)$.}\label{LRBconic}
\end{figure}

On the pointed segment $S \cup \{a_1,\ldots,a_k\}$, the set of faces consists of $2k+1$ faces. However, on $C$, the sections $f^{-1}((a_1',a_1))$ and $f^{-1}((a_k,a_1''))$ (i.e. the preimage of the sections $(a_1',a_1)$ and $(a_k,a_1'')$) are contained in the same face. As $a_1 - a_1' =\varepsilon \ll 1$, we ignore this infinitesimally-small face and thus $\LL(S)$, the set of faces of $S$, has only $2k$ faces: $k$ points $\{a_1,\ldots,a_k\}$ and $k$ open sections of the segment. We will identify this set of faces with the set of faces $\LL(C)$.

We can now associate an LRB structure to $\LL(S)$, as it is done for a set of faces of a pointed line; that is, to every face  $P \in \LL(S)$, we associate a vector of signs $i_H(P) \in (L^1_2)^k$ in the following way: Given $1 < j \leq k$, the  point $a_j$ divides  $S$ into three disjoint parts: the point itself $\{a_j\}$ and two other open sections: $H_{j,1} = \{x > a_j\}$ and $H_{j,2} = \{x < a_j\}$.
Associate to the set $H_{j,1}$ the sign $+$, to the set $H_{j,2}$ the sign $-$ and to the set $\{ a_j\}$ the sign $0$. Obviously, as in the case of an unbounded component, one can rename the set $H_{j,1}$ as $H_{j,2}$ and $H_{j,2}$ as $H_{j,1}$ and thus induce a different assignment of signs, but once we assign these signs for each set, they are fixed.

For $j=1$, since we have ignored the section $\{a_1' < x < a_1\}$, the  point $a_1$ divides  $S$ into only two disjoint parts: the point $\{ a_1\}$ itself and  $H_{1,1} = \{x > a_1\}$. Associate to the set $H_{1,1}$ the sign $+$ (or $-$) and to the set $\{ a_1\}$ the sign $0$. Again, once we assign these signs for each set, they are fixed  (see Figure \ref{LRBconic}(c)).

Thus, the map $i_H:\LL(S) \to (L^1_2)^k$ is defined as in Definition \ref{defLRBunbounded}: for each face $P \in \LL(S)$, the $j^{\rm th}$ coordinate of $i_H(P)$ depends on whether $P = \{ a_j \}$, $P \subseteq H_{j,1}$ or $P \subseteq H_{j,2}$.
\end{definition}

Note that $\LL(S)$ has an LRB structure (by the same arguments of Lemma \ref{lemLRBwellDef}(1)) and hence by the identification of $\LL(S)$ and $\LL(C)$, also $\LL(C)$ has an LRB structure. Similar to the case of an unbounded component, as  $C_1$ is topologically equivalent to an open segment, different assignments of signs to  $H_{j,1}$ and $H_{j,2}$, as described above, will induce isomorphic LRBs, by Remark \ref{remDifSignIsoLRB}.

\medskip

We still have to prove that Definition \ref{defLRBbounded} does not depend on the choice of the initial point $p_1$, when numerating the points on $C$. This is equivalent to prove that if we choose a point $p_2'$ as a point to the left of $p_2$ (being infinitesimally-close  to $p_2$) and consider the induced LRB structure on $C_2 = C -\{p_2'\}$ (see Figures \ref{LRBconic}(a) and \ref{LRBconic}(d)), then the LRBs $i_H(\LL(C_1))$ and $i_H(\LL(C_2))$ are isomorphic:

\begin{prs}\label{prsLRBconic}
The LRBs $\LL_0(C_1) = i_H(\LL(C_1))$ and $\LL_0(C_2) = i_H(\LL(C_2))$ are isomorphic. Therefore, the LRB structure on $C \cup \{p_1,\ldots,p_k\}$ is independent of the choice of the cutting point.
\end{prs}

\begin{proof}
As noted after Definition \ref{defLRBbounded}, different sign  assignments on $\LL(C_1)$ (or on $\LL(C_2)$) induce isomorphic LRB structures.
Thus, we first fix an assignment of signs for $\LL(C_1)$ and $\LL(C_2)$ and then prove that the LRBs are isomorphic.

The sign assignment on $\LL(C_1)$ is the following: for each $1 < j \leq k$, we assign the sign $+$ to $H_{j,1}$, the sign $-$ to $H_{j,2}$
and the sign $0$ to $\{a_j\}$; for $j=1$, we assign the sign $+$ to $H_{1,1}$
and the sign $0$ to $\{a_1\}$ (see Figure \ref{LRBconic}(c)).

The sign assignment on $\LL(C_2)$ is the following: for each $1 \leq j \leq k $ where $ j \neq 2$, we assign the sign $+$ to $H_{j,1}$, the sign $-$ to $H_{j,2}$ and the sign $0$ to $\{a_j\}$; for $j=2$, we assign the sign $+$ to $H_{2,1}$
and the sign $0$ to $\{a_2\}$ (see Figure \ref{LRBconic}(e)).

Thus, going over all the $2k$ faces of $\LL(C_1)$ from left to right, we get that:
$$\LL_0(C_1)=i_H(\LL(C_1)) = \left\{ \begin{array}{c}
(0,-,-,\ldots,-),\ (+,-,-,\ldots,-),\ (+,0,-,\ldots,-),\\
(+,+,-,\ldots,-),\ \ldots,\ (+,\ldots,+)
\end{array} \right\}.$$
In the same way, going over all the $2k$ faces of $\LL(C_2)$ from left to right, we get that:
$$\LL_0(C_2) = i_H(\LL(C_2))=\left\{
\begin{array}{c}
(-,0,-,-,\ldots,-),\ (-,+,-,-,\ldots,-),\ (-,+,0,-,\ldots,-),\\
(-,+,+,-,\ldots,-),\ \ldots,\ (-,+,\ldots,+),\ (0,+,\ldots,+),\ (+,+,\ldots,+)
\end{array} \right\}.$$
Both LRBs describe the movement over the $2k$ faces along a open straight (bounded) segment with $k$ marked points, i.e. given faces $x$ and $y$, then $x\cdot y$ is the face we enter in after the movement from $x$ to $y$ on this line and thus they are isomorphic.  Thus the explicit isomorphism from $\LL_0(C_1)$ to
$\LL_0(C_2)$ maps the points $p_i \mapsto p_{i({\rm mod}\,\,k)+1}$ and the sections of $C_1$ are mapped to the corresponding sections of $C_2$, according to the mapping of the points.
\end{proof}

Thus, given an arrangement $\A$ and a bounded component $H \in \A$ such that
$H \cap \A^H =\{ p_1, \dots,p_k \}$, we can choose a point $p\in H$ infinitesimally-close to the point $p_1$ and delete it. In this way, we can consider the face LRB associated to $\A_H - \{p\}$, as in the case of an unbounded component (when ignoring the infinitesimally-small section between $p$ and $p_1$). As was shown, this face LRB does not depend on the location of $p$ (when the only condition is that $p \neq p_j$ for all $j$) up to an isomorphism. Denote this associated LRB by $\LL_0(\A_H)$, which is a sub-LRB of $(L_2^1)^k$. For more examples, see Example \ref{exampEmbLRB}(2) and Figure \ref{ExamResEmb2} below.

\subsection{The embedding principle for the face LRB of CL arrangements}\label{subsecEmbedCL}

We are now ready to formulate the main result of this section:  the structure of the sub-LRBs of $\LL_0(\A)$ induced by the  components of a CL arrangement $\A \subset \RR^2$.

\begin{prs} \label{prsEmbed}
Let $\A = \{H_1,\ldots,H_m\}$ be an arrangement of smooth connected curves in $\RR^2$, such that $H_i$ is defined by $\{f_i=0\}$ where $f_i \in \RR[x,y]$. Let $H \doteq H_1$
and
$$H \cap \{H_2,\dots,H_m\} = \{p_1,\dots,p_k\} \subset H.$$
Then there is a {\em bijective function}, {\em which is not necessarily an isomorphism}, of LRBs:
$$\varphi : \LL_0(\A_H) \to \LL_0(\A)|_H \subseteq (L_2^1)^m,$$
satisfying:
\begin{enumerate}
\item $(\varphi(\LL_0(\A_H)))_1 = 0.$
\item For every $j>1$:

\noindent
(a) If $H \cap H_j = \emptyset$, then $(\varphi(\LL_0(\A_H)))_j$ is constant (either $+$ or $-$, depending on the mutual position of the non-intersecting components $H$ and $H_j$). Explicitly, all the vectors in $\varphi(\LL_0(\A_H))$ have the same sign in their $j^{\rm th}$ coordinate.

\noindent
(b) If $H \cap H_j \neq \emptyset$, let $H \cap H_j = \{ p_i \ | \ i \in K_j \}$, where $K_j$ is the set of indices of the points in $H \cap H_j$. Then:
$$(\varphi(\LL_0(\A_H)))_j = s \cdot \prod_{i \in K_j}   ((\LL_0(\A_H))_i)^{m_i}$$
where  $s \in \{\pm 1\}$, $m_i = {\rm multi}(p_i)$ is the intersection multiplicity  at the point $p_i$ (see Definition \ref{def_inter_mul}),
and the multiplication of signs (in the right hand side) is the usual product (explicitly, $+\cdot + = - \cdot - = +,\, +\cdot - = - \cdot + = -,\, 0 \cdot \{\pm\} = 0 $). Note that the numeration of the indices in the right hand side is according to the numeration of the points in the arrangement of points in $H = H_1$.
\end{enumerate}
\end{prs}

As before, we illustrate the result of this proposition by some examples before proving it.
\begin{example}\label{exampEmbLRB} {\rm
\begin{enumerate}
\item Figure \ref{ExamResEmb}(a) presents a CL arrangement $\A$ consisting of three  lines and a conic tangent to one of the lines, and Figure \ref{ExamResEmb}(b) presents the restricted arrangement $\A_{H_1}$. By Proposition \ref{prsL0semigrp}, $\LL_0(\A)$ is indeed a semigroup. The faces of $\A_{H_1}$  are denoted by $x_1,\dots,x_5$
    and their corresponding faces in $\A$ are denoted by $f_1,\ldots,f_5$.

\begin{figure}[h]
\epsfysize=5cm \centerline{\epsfbox{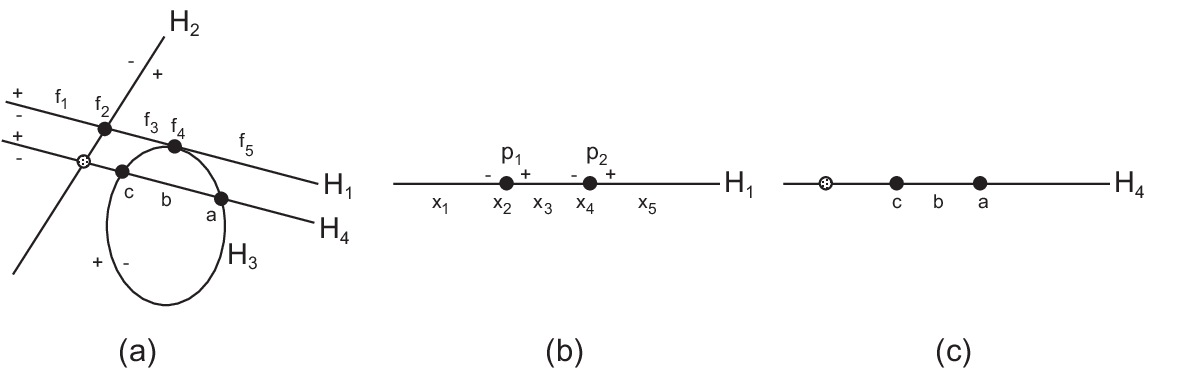}}
\caption{An example for illustrating the result of  Proposition \ref{prsEmbed}: In part (a), the CL arrangement is presented and $f_i$ are the faces contained in $H_1$ in the face set $\LL(\A)$. In part (b), we present an arrangement of points on a line which can be thought of as the restricted arrangement $\A_{H_1}$, and $x_i$ are faces in the face set $\LL_0(\A_{H_1})$. Part (c) presents the faces $a,b,c$ contained in $H_4$ (needed for the proof of Proposition \ref{prsEmbed}).}\label{ExamResEmb}
\end{figure}

    Let $H = H_1$. Then, the corresponding LRBs are:
$$  \LL_0(\A_H) = \left\{
\begin{array}{c}
i_H(x_1)=(-,-),\ i_H(x_2)=(0,-),\ i_H(x_3)= (+,-),\\
i_H(x_4)= (+,0),\ i_H(x_5)= (+,+)
\end{array} \right\},$$
and in a table form in Table \ref{table_B1},
\begin{table}[!ht]
$$\begin{array}{|c||c|c|}
\hline
 & (i_H(x_i))_1 & (i_H(x_i))_2 \\
\hline
\hline
i_H(x_1) & - & - \\
i_H(x_2) & 0 & - \\
i_H(x_3) & + & - \\
i_H(x_4) & + & 0 \\
i_H(x_5) & + & + \\
\hline
\end{array}$$
\caption{$\LL_0(\A_H)$ in a table form}\label{table_B1}
\end{table}
and
$$ \LL_0(\A)|_H = \varphi(\LL_0(\A_H)) = \left\{
\begin{array}{c}
i_{\A}(f_1)=(0,+ ,+,+ ),\ i_{\A}(f_2)=(0,0 ,+,+ ),\\
i_{\A}(f_3)= (0,- ,+,+ ),\ i_{\A}(f_4)= (0,- ,0,+ ),\\
i_{\A}(f_5)= (0,- ,+,+ )
\end{array} \right\},$$
and in a table form in Table \ref{table_B2}.
\begin{table}[!ht]
$$\begin{array}{|c||c|c|c|c|}
\hline
 & (i_{\A}(f_j))_1 & (i_{\A}(f_j))_2 & (i_{\A}(f_j))_3 & (i_{\A}(f_j))_4 \\
\hline
\hline
i_{\A}(f_1) & 0 & + & + & + \\
i_{\A}(f_2) & 0 & 0 & + & + \\
i_{\A}(f_3) & 0 & - & + & + \\
i_{\A}(f_4) & 0 & - & 0 & + \\
i_{\A}(f_5) & 0 & - & + & + \\
\hline
\hline
 & 0 & -(i_H(x_i))_1 & ((i_H(x_i))_2)^2 & + \\
\hline
\end{array}.$$
\caption{$\LL_0(\A)|_H$ in a table form, where in the last row, the relations to
$\LL_0 (\A_H)$ are presented.}\label{table_B2}
\end{table}

\medskip

We deal with $\varphi(\LL_0(\A_H))$ coordinate by coordinate:
\begin{enumerate}
\item First, note that $(\varphi(\LL_0(\A_H)))_1  = 0$ (property (1) of the proposition). In the table form, all the values in the first column of Table \ref{table_B2} are $0$.
\item Since $H_2 \cap H = \{p_1\}$ where $m_1 = {\rm multi}(p_1)=1$, then by property (2)(b) of the proposition,
   $$(\varphi(\LL_0(\A_H)))_2  = -(\LL_0(\A_H))_1$$
   (note the scalar multiplication by ($-1$)). In the table form, the second column of Table  \ref{table_B2} is equal to the first column of Table \ref{table_B1} multiplied by the scalar $(-1)$.
\item Since $H_3 \cap H  = \{p_2\}$, where $m_2 = {\rm multi}(p_2)=2$, then again by property (2)(b) of the proposition,
   $$(\varphi(\LL_0(\A_H)))_3  = ((\LL_0(\A_H))_2)^2.$$
   In the table form, the values in the third column of Table  \ref{table_B2} are equal to the square of the corresponding values of the second column of Table \ref{table_B1}.
\item Since $H_4 \cap H  = \emptyset$, then $(\varphi(\LL_0(\A_H)))_4  = +$ (by property (2)(a) of the proposition). In the table form, all the values in the fourth column of Table \ref{table_B2} are $+$.
\end{enumerate}
\medskip

\item Relabel the CL arrangement in Figure \ref{ExamResEmb}(a), such that the conic will be now labeled as $H_1$, see Figure \ref{ExamResEmb2}(a).
\begin{figure}[h]
\epsfysize=5cm \centerline{\epsfbox{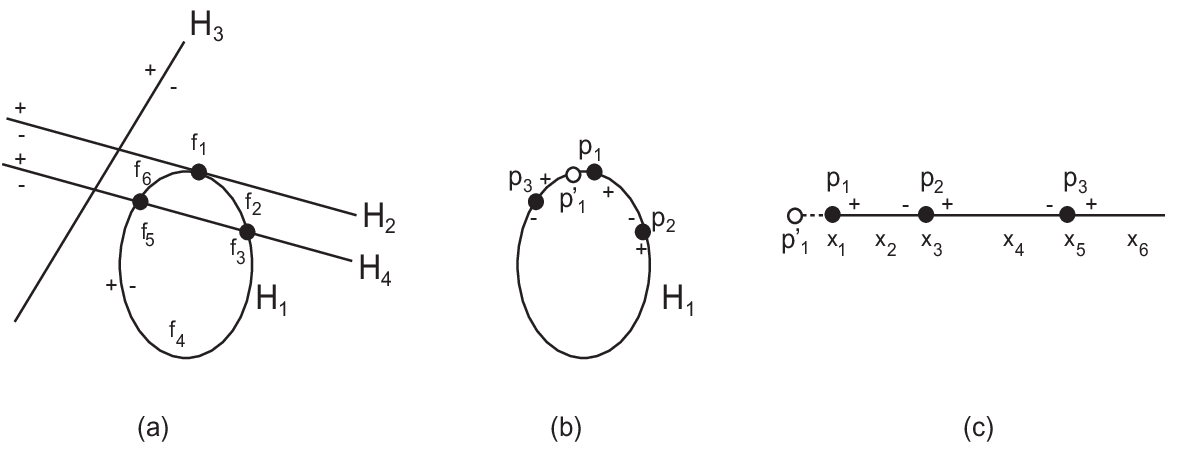}}
\caption{Another example for illustrating the result of Proposition \ref{prsEmbed}:
$f_i$ are the faces contained in the conic $H_1$ in the face set $\LL(\A)$ (see part (a)). $x_i$ are the faces of the face set $\LL_0(\A_{H_1})$ (see part (c)). The three parts illustrate the process of associating an LRB structure to the conic $H_1$. First, we remove a point $p_1'$ from $H_1$ to the left of $p_1$ (see part (b)). Then, we consider $H_1$ as a section with this  point deleted, i.e. a segment which starts from $p_1$ (see part (c)).}\label{ExamResEmb2}
\end{figure}

The faces of $\A_{H_1}$  are denoted by $x_1,\ldots,x_6$ (see Figure \ref{ExamResEmb2}(c); note that the section between $p'_1$ and $p_1$ is ignored) and their corresponding faces in the CL arrangement $\A$ are denoted by $f_1,\ldots,f_6$ (see Figure \ref{ExamResEmb2}(a)).
Let $H = H_1$. As was explained in Definition \ref{defLRBbounded}, one can induce an LRB structure on $\A_H$. Then, the corresponding LRBs are:
$$\LL_0(\A_H) = \left\{
\begin{array}{c}
i_H(x_1)=(0,-,-),\ i_H(x_2)=(+,-,-),\ i_H(x_3)=(+,0,-),\\
i_H(x_4)=(+,+,-),\ i_H(x_5)=(+,+,0),\ i_H(x_6)=(+,+,+)
\end{array} \right\}$$
and in a table form in Table \ref{table_C1},
\begin{table}[!ht]
$$\begin{array}{|c||c|c|c|}
\hline
 & (i_H(x_i))_1 & (i_H(x_i))_2 & (i_H(x_i))_3 \\
\hline
\hline
i_H(x_1) & 0 & - & - \\
i_H(x_2) & + & - & - \\
i_H(x_3) & + & 0 & - \\
i_H(x_4) & + & + & - \\
i_H(x_5) & + & + & 0 \\
i_H(x_6) & + & + & + \\
\hline
\end{array}$$
\caption{$\LL_0(\A_H)$ in a table form}\label{table_C1}
\end{table}
and
$$\LL_0(\A)|_H = \varphi(\LL_0(\A_H)) = \left\{
\begin{array}{c}
i_{\A}(f_1)=(0,0,-,+ ), \ i_{\A}(f_2)= (0,- ,-,+ ),\\
i_{\A}(f_3)= (0,- ,-,0 ),\ i_{\A}(f_4)= (0,- ,-,- ),\\
i_{\A}(f_5)= (0,- ,-,0),\ i_{\A}(f_6)=(0,-,-,+)
\end{array} \right\}$$
and in a table form in Table \ref{table_C2}.
\begin{table}[!ht]
$$\begin{array}{|c||c|c|c|c|}
\hline
 & (i_{\A}(f_j))_1 & (i_{\A}(f_j))_2 & (i_{\A}(f_j))_3 & (i_{\A}(f_j))_4 \\
\hline
\hline
i_{\A}(f_1) & 0 & 0 & - & + \\
i_{\A}(f_2) & 0 & - & - & + \\
i_{\A}(f_3) & 0 & - & - & 0 \\
i_{\A}(f_4) & 0 & - & - & - \\
i_{\A}(f_5) & 0 & - & - & 0 \\
i_{\A}(f_6) & 0 & - & - & + \\
\hline
\hline
 & 0 & -((i_H(x_i))_1)^2 & - & (i_H(x_i))_2 \cdot (i_H(x_i))_3 \\
\hline
\end{array}$$
\caption{$\LL_0(\A)|_H$ in a table form, where in the last row, the relations to
$\LL_0 (\A_H)$ are presented.}\label{table_C2}
\end{table}

\medskip

We deal with $\varphi(\LL_0(\A_H))$ coordinate by coordinate:
\begin{enumerate}
\item First, note that $(\varphi(\LL_0(\A_H)))_1  = 0$ (property (1) of the proposition). In the table form, all the values in the first column of Table \ref{table_C2} are $0$.
\item Since $H_2 \cap H = \{p_1\}$, where $m_1={\rm multi}(p_1)=2$, then by property (2)(b) of the proposition, we have: $(\varphi(\LL_0(\A_H)))_2  = -((\LL_0(\A_H))_1)^2$ (note the scalar multiplication by $(-1)$). In the table form, the values in the second column of Table  \ref{table_C2} are equal to the square of the corresponding values in the first column of Table \ref{table_C1} multiplied by the scalar $(-1)$.
\item Since $H_3 \cap H  = \emptyset$, then $(\varphi(\LL_0(\A_H)))_3  = -$ (by property (2)(a) of the proposition). In the table form, all the values in the third column of Table \ref{table_C2} are $-$.
\item Since $H_4 \cap H  = \{p_2,p_3\}$, where
$$m_2={\rm multi}(p_2)=m_3={\rm multi}(p_3)=1,$$
then by property (2)(b) of the proposition,
$$(\varphi(\LL_0(\A_H)))_4  = (\LL_0(\A_H))_2 \cdot (\LL_0(\A_H))_3.$$
In the table form, the values in the fourth column of Table \ref{table_C2} are equal to the {\em usual} multiplication of the corresponding values in the second and the third columns of Table \ref{table_C1}.
\end{enumerate}

\end{enumerate}
}
\end{example}

\begin{remark}
{\rm Note that if  every singular point is locally a transversal intersection of several components (as in the case, for example, of a line arrangement), then one can easily see that Proposition \ref{prsEmbed} is indeed a generalization of Lemma \ref{lemLineArrLRB}.}
\end{remark}

\begin{proof}[Proof of Proposition \ref{prsEmbed}]
We start by proving that $\varphi$ is a bijection, which is not an isomorphism. As in the proof of Lemma \ref{lemLineArrLRB}, $\varphi$ is injective, since  if there were two faces in $\LL_0(\A_H)$ sent to the same face in $\LL_0(\A)|_H$, that would have meant that these two different faces on $H$ have the same mutual position with respect to all the other components $H_2,\dots,H_m$, which is impossible. By the fact that $\sharp (\LL_0(\A)|_H) = \sharp \LL_0(\A_H)$, we get that $\varphi$ is bijection.

On the other hand, we show now that $\varphi$ is not necessarily a homomorphism. We use the same notations introduced in Remark \ref{remNotationI}. Let $\A$ be the CL arrangement presented in Figure \ref{ExamResEmb}(a) above, and let $H=H_4$. Let $a$ and $c$ be the intersection points of $H$ with the conic and let $b$ be the 1-dimensional segment between them (see also Figure \ref{ExamResEmb}(c)). When considering $a,b,c$ as faces of $\A_H$, then in $\LL_0(\A_H)$, $i_H(a)  i_H(c) = i_H(b)$. However, when considering $a,b,c$ as faces of $\A$ (see Figure \ref{ExamResEmb}(a)),  $i_\A(a)$ and $i_\A(c)$ have a zero value in the coordinate corresponding to the conic. However,  $i_\A(b)$ does not have a zero value in that coordinate (since $b$ is not contained in the conic). Thus, in $\LL_0(\A)$:
$\varphi(i_H(a))  \varphi(i_H(c))=i_\A(a)  i_\A(c) \neq i_\A(b) = \varphi(i_H(b))$.

\medskip

Now we pass to the proof of the two properties of $\varphi$.
The proofs of properties (1) and (2)(a) are  identical to the corresponding proofs in Lemma \ref{lemLineArrLRB}.

We now prove property (2)(b). Let $j>1$ and assume that $H \cap H_j = \{p_i \ |\ i \in K_j \}$. Let $c$ be a face of $\A$ which lays in the component $H$. Note that if $c= \{ p_k \}$ for $k\in K_j$, then $(i_H(c))_k=0$ in
$\LL_0(\A_H)$ and $(i_\A(c))_j=0$ in $\LL_0(\A)|_H = \varphi(\LL_0(\A_H)) \subset \LL_0(\A)$; thus property (2)(b) is satisfied for a $0$-dimensional face $c$.

Therefore, we can assume that the face $c$ has dimension $1$. Then, either $c \subset \{f_j>0\}$ or $c \subset \{f_j<0\}$. We claim that $i_{\A}(c)$ is determined by the mutual position of $c$ with respect to the singular points $\{p_i \ |\ i \in K_j \}$: the (usual) product of the signs (of the $i^{\rm th}$ coordinate of $\LL_0(\A_H)$, where $i \in K_j$) determines whether $c$ is in $\{f_j>0\}$ or in $\{f_j<0\}$. We check explicitly all the possible cases:
\begin{enumerate}
 \item If $H \cap H_j = \{p_i\}$ is a single transversal intersection point ($m_i={\rm multi}(p_i)=1$), then, as $H$ and $H_j$ have only  one intersection point in $\RR^2$, we can proceed as in the proof of property (2)(b) in Lemma \ref{lemLineArrLRB}.
 \item If $H \cap H_j = \{p_i\}$ is a single tangent point ($m_i={\rm multi}(p_i)=2$), then we claim that the $j^{\rm th}$ coordinate of $\varphi(\LL_0(\A_H))$ is constant: either $+$ or $-$ (except for the face $x= \{ p_i \}$, whose sign in the $j^{\rm th}$ coordinate is $0$, as was described above for the case of a $0$-dimensional face). This is since  $H$ is either entirely  outside or entirely inside the domain $\{f_j > 0\}$, and the value of the $j^{\rm th}$ coordinate is determined by the signs attached to the two domains of the plane divided by the curve $H_j$. In the first case $(i_\A(c))_j=+$ and in the second case $(i_\A(c))_j=-$. Also, in any case, $((i_H(c))_i)^2=+$ and thus we proved that $(i_\A(c))_j= \pm ((i_H(c))_i)^2 = \{\pm 1\}$, thus the $j^{\rm th}$ coordinate of $\varphi(\LL_0(\A_H))$ is indeed constant.

  \item Generalizing cases (1) and (2), if $H \cap H_j = \{p_i\}$ is a single singular point of multiplicity $m_i={\rm multi}(p_i)>2$, then we are only interested in the {\it parity} of $m_i$.  If $m_i$ is even, then locally at $p_i$, the curve $H$ does not cross $H_j$ to its ``other side'' (i.e. it is contained only in the domain $\{ f_j \geq 0 \}$ or in the domain $\{ f_j \leq 0 \}$), and thus the treatment of this case is as in case (2), where $m_i=2$. If $m_i$ is odd, then locally at $p_i$, the curve $H$ does  cross $H_j$ to its ``other side'', and thus the treatment of this case is as in case (1), where $m_i=1$.

  \item Assume now that $H \cap H_j = \{p_{s_1},p_{s_2}\}$ is two transversal intersection points (i.e. $m_{s_1}={\rm multi}(p_{s_1})=m_{s_2}={\rm multi}(p_{s_2})=1$; for example, when $H$ is a conic and $H_j$ is a circle intersecting $H$ transversally in two points, see Figure \ref{exampleEmb}(1)(a)). Recall that the structure of the induced LRB  of a pointed real curve $C \cup \{p_1,\ldots,p_k\}$ (see Section \ref{subsec-3.2}) allows us to think
      on faces which are to the right (or to the left) of a point $p_i$, $1 \leq i \leq k$. Without loss of generality, we can assume that $p_{s_2}$ is to the right of $p_{s_1}$.

    \begin{figure}[h]
    \epsfysize=9cm \centerline{\epsfbox{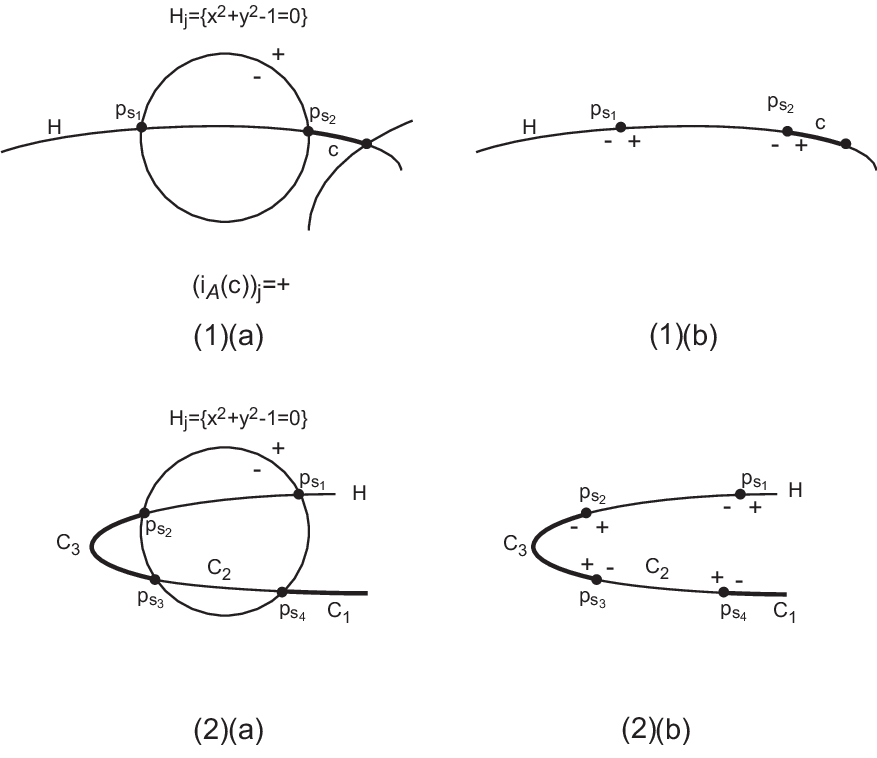}}
    \caption{\small Examples for the proof of cases (4) and (5) of Proposition \ref{prsEmbed}: Part (1)(a) illustrates the situation of case (4), and part (1)(b) presents the restricted arrangement $\A_H$. In this example, we have: $$(i_{\A}(c))_j=+=+ \cdot + = (i_H(c))_{s_1} \cdot (i_H(c))_{s_2}.$$ Part (2)(a) illustrates the situation of case (5), and part (2)(b) presents the restricted arrangement $\A_H$. In this example, we have:
    $$(i_{\A}(c_1))_j= + = + \cdot + \cdot + \cdot += (i_H(c_1))_{s_1} \cdot (i_H(c_1))_{s_2}\cdot (i_H(c_1))_{s_3} \cdot (i_H(c_1))_{s_4},$$
    $$(i_{\A}(c_2))_j= - = - \cdot + \cdot + \cdot += (i_H(c_2))_{s_1} \cdot (i_H(c_2))_{s_2}\cdot (i_H(c_2))_{s_3} \cdot (i_H(c_2))_{s_4},$$
    $$(i_{\A}(c_3))_j= + = - \cdot - \cdot + \cdot += (i_H(c_3))_{s_1} \cdot (i_H(c_3))_{s_2}\cdot (i_H(c_3))_{s_3} \cdot (i_H(c_3))_{s_4}.$$ }\label{exampleEmb}
    \end{figure}

      If $H$ is an unbounded curve, then the fact that $c \subset \{f_j > 0\}$ is equivalent to the fact that $c$ is to the right of $p_{s_2}$ (as presented in Figures \ref{exampleEmb}(1)(a) and \ref{exampleEmb}(1)(b)) or to the  left of $p_{s_1}$. In the first case:
      $$(i_H(c))_{s_1} \cdot (i_H(c))_{s_2} = + \cdot + = + = (i_\A(c))_j.$$
      In the second case:
      $$(i_H(c))_{s_1} \cdot (i_H(c))_{s_2} = - \cdot - = + = (i_\A(c))_j.$$
      We use a similar argument when $c \subset \{f_j < 0\}$.

      If $H$ is a (bounded) oval, then, as described in Definition \ref{defLRBbounded}, one chooses a point $p$ infinitesimally-close to a point $p_i \in \{p_1,\ldots,p_k\}$. Thus, an LRB structure of the set of faces of $\A_H$ is induced independently of the choice of the point $p$, when looking on $H$ as a (bounded) segment (by Proposition \ref{prsLRBconic}). Therefore, we can use the same argument used in the case of an unbounded curve.

  \item  Generalizing  case (4), assume that $H \cap H_j = \{p_{s_1},\ldots,p_{s_n}\}$, i.e. the intersection of $H$ and $H_j$ is a transversal intersection of $n$ points ($m_{s_i} = {\rm multi}(p_{s_i})= 1$ for $1 \leq i \leq n$). An example for this case is when $H$ be a parabola and $H_j$ is a circle intersecting $H$ transversally at $4$ points, see Figure \ref{exampleEmb}(2)(a).

      Assume that $H$ is an unbounded connected curve and thus without loss of generality, we can numerate the points $\{p_{s_i}\}$ consecutively, such that the point $p_{s_n}$ will be the  rightmost point. Assume also that in $\LL_0(\A)$, the domain $\{f_j > 0\}$ induces the sign $+$ in the $j^{\rm th}$ coordinate. Let $c$ be a 1-dimensional face in $\A_H$. Assume now that $c$ is to the right of $p_{s_n}$ (e.g. the section $c_1$ in Figures \ref{exampleEmb}(2)(a) and \ref{exampleEmb}(2)(b)). Thus $$(i_H(c))_{s_1}\cdot \ldots \cdot (i_H(c))_{s_n}= +\cdot \ldots \cdot + = +$$ in $\LL_0(\A_H)$. In addition, if $c \subset \{f_j > 0\}$, then in $\LL_0(\A)$ (or, more specifically, in $\LL_0(\A)|_H$), $(i_{\A}(c))_j = +$ (otherwise, if $c \subset \{f_j < 0\}$, then $(i_{\A}(c))_j = -$).

      Now, if we  move to the consecutive 1-dimensional face $c'$, adjacent to $c$  (i.e. between $p_{s_n}$ and $p_{s_{n-1}}$; see e.g. the section $c_2$ in Figures \ref{exampleEmb}(2)(a) and \ref{exampleEmb}(2)(b)), then in $\LL_0(\A_H)$:
      $$\hspace{1.7cm}(i_H(c'))_{s_1}\cdot \ldots \cdot (i_H(c'))_{s_{n-1}} \cdot (i_H(c'))_{s_n}= \underbrace{ +\cdot  \ldots \cdot +}_{n-1 \text{ times }} \cdot - = -,$$
      while in $\LL_0(\A)$, as $c' \subset \{f_j < 0\}$ (if indeed $c \subset \{f_j > 0\}$), $(i_{\A}(c'))_j~=~-$. Note that if $c \subset \{f_j < 0\}$, then $c' \subset \{f_j > 0\}$, so $(i_{\A}(c'))_j = +$, i.e. it is a constant scalar multiplication by $\{ \pm 1\}$ of $\prod\limits_{v=1}^n (i_H(c'))_{s_v}$.

      In this way, we can proceed to the next adjacent 1--dimensional face and so on (e.g. the section $c_3$ in Figures \ref{exampleEmb}(2)(a) and \ref{exampleEmb}(2)(b)), till we have reached to the leftmost face, i.e. to the face to the left of $p_{s_1}$, proving property (2)(b) for this type of intersection.

      The treatment of the case when $H$ is a bounded oval is similar to the former case (see also case (4)).

  \item In other cases, i.e. when $H \cap H_j = \{p_{s_1},\ldots,p_{s_n}\}$ and $m_{s_i} = {\rm multi}(p_{s_i}) \geq 1$, then this case is treated as  case (5) (i.e. treating the faces sequentially, starting from the rightmost face and continuing to its adjacent face, and so on), combined with the insights of cases (1), (2) and (3).

\end{enumerate}
\end{proof}

\section{CL arrangements: Chamber counting}\label{secComb}

In this section, we generalize the chamber counting formula known for line arrangements, to the case of CL arrangements (Section \ref{subsecDelResCL}).
We start by recalling the chamber counting formula for hyperplane
arrangements (Section \ref{subsecDelRes}).

\subsection{Preliminaries: Chamber counting for line arrangements} \label{subsecDelRes}

Let $\A  = \{H_1,\dots, H_n \} \subset \RR^2$ be a line arrangement, and let $f_i \in \RR[x,y]$ be the corresponding forms of the lines. Let $L = L(\A)$ be the semi-lattice of nonempty intersections of elements of $\A$. The main reference for this subsection is \cite{OT}.

As before, given  $H\in \A$,
let $\A^H = \A - \{H\}$ be the  \emph{deleted arrangement} in $\RR^2$, and  let $\A_H = \left\{ K \cap H \ | \ K \in \A^H \right\}$ be the  \emph{restricted arrangement} in $H$. Let
${\mathcal C}(\A)$ be the set of chambers of $\A$, i.e. the connected components of $\RR^2 - \A$. Then,
Zaslavsky's chamber counting formula \cite{Z} states that:

\begin{equation} \label{eqnResDelHyperplane}
|{\mathcal C}(\A)|= \left| {\mathcal C} \left( \A^H \right) \right|+ \left| {\mathcal C} \left( \A_H \right) \right|.
\end{equation}

\begin{remark} \label{remAlgProofZasHyperplane}
\rm{
One can give a simple set-theoretic proof for this formula: Deleting a hyperplane $H$ from the arrangement $\A$ induces a surjection of LRBs  $f:\LL_0(\A) \to \LL_0 \left( \A^H \right)$, which deletes the coordinate corresponding to the hyperplane $H$. Thus, $|{\mathcal C}(\A)|$ is equal to the sum of the number of chambers of the deleted arrangement $\A^H$ plus the number of chambers which are identified by the map $f$. Given $C_1,C_2 \in {\mathcal C}(\A)$, note that $f(C_1)=f(C_2)$ if and only if $C_1$ and $C_2$ share a common codimension-1 face contained in $H$, i.e. a chamber of the restricted arrangement $\A_H$. Hence, the number of the identified chambers is equal to the number of the chambers of $\A_H$, and Equation (\ref{eqnResDelHyperplane}) follows.\footnote{\ This proof was introduced to us by an anonymous referee.}}
\end{remark}

Note that if we denote by $I(\A)$ the unique two-sided ideal of the LRB $\LL(\A)$ (which consists of the set of the chambers of $\A$), Equation (\ref{eqnResDelHyperplane}) is equivalent to the following equation:
$$|I(\A)|=\left| I \left( \A^H \right) \right|+ \left| I \left( \A_H \right) \right|.$$

\begin{remark}
{\rm Other restrictions on the combinatorics of  real and complex line arrangements can be found, for example, in Hirzebruch's seminal paper \cite{H}, but we do not deal with their generalizations here. Note also that Equation (\ref{eqnResDelHyperplane}) holds for any hyperplane arrangement and not only for line arrangements.}
\end{remark}

\subsection{Chamber counting for CL arrangements}\label{subsecDelResCL}
In this section, we develop an restriction-deletion formula for real CL arrangements (see Proposition \ref{prop-res-del-CL} below), since Equation (\ref{eqnResDelHyperplane}) for chamber counting does not hold anymore in this case. Additionally, we look at the connections between the amended formula we introduce and Zaslavsky's generalization \cite{Z2}.

\medskip

We start with a simple example: for the CL arrangement $\A$ appearing in Figure \ref{NoDelRes}, we have:
$$|\C(\A)|=4,\,\, \left| \C \left( \A^H \right) \right|=2,\,\,  \left| \C \left( \A_H \right) \right|=3 \,\,\Rightarrow\,\, |\C(\A)|\not= \left|\C \left( \A^H \right) \right|+ \left| \C \left( \A_H \right) \right|.$$
On the other hand:
$$\left| \C \left( \A^C \right) \right|=2,\,\, \left| \C \left( \A_C \right) \right|=2 \,\, \Rightarrow  \,\, |\C(\A)|=\left| \C \left( \A^C \right) \right|+\left| \C \left( \A_C \right) \right|.$$

\begin{figure}[h]
\epsfysize=2cm \centerline{\epsfbox{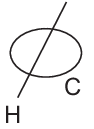}}
\caption{An example for illustrating the problem with the usual restriction-deletion formula (Equation (\ref{eqnResDelHyperplane})) for the case of CL arrangements:
$$4=|\C(\A)|\not= \left|\C \left( \A^H \right) \right|+ \left| \C \left( \A_H \right) \right|=2+3 =5.$$}\label{NoDelRes}
\end{figure}

Thus, the restriction-deletion formula has to be changed. In order to formulate this change accurately, we start by introducing some definitions.

\begin{definition} \label{defBound}
Let  $\A \subset \RR^2$ be a real CL arrangement.\\
(1) Let $H \in \A$. Define the function:
$${\rm bound} : \C(\A_H) \to \{Y \in P(\C(\A)) : |Y| = 2 \}$$
$${\rm bound}(E) = \{X_1,X_2\} \text{ where } E \subset \overline{X_1} \cap \overline{X_2},$$
where $P(\C(\A))$ is the power set of $\C(\A)$ and $\overline{X}$ is the (topological) closure of $X$.\\
(2) For $E_1,E_2 \in \C(\A_H)$, define the following equivalence relation $\sim$:
$$ E_1 \sim E_2 \,\,\Leftrightarrow \,\, {\rm bound}(E_1)  = {\rm bound}(E_2), $$
and define:
$$ b(H) = \C \left( \A_H \right) /\sim. $$
\end{definition}

\noindent
For example, for the CL arrangement presented in Figure \ref{DelResConic}, $|b(H)|=2$.
\begin{figure}[h]
\epsfysize=2cm \centerline{\epsfbox{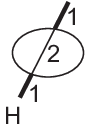}}
\caption{An illustration of the different elements in $b(H)$. The bold sections labeled by $1$ are identified in $b(H)$.}\label{DelResConic}
\end{figure}

\begin{prs}\label{prsBound_I_connection}
Let $H_j \in \A=\{H_1,\dots,H_n\}$ be a component  of a real CL arrangement $\A$, and let $E_1,E_2 \in \C(\A_{H_j})$. Then:
$${\rm bound}(E_1)  = {\rm bound}(E_2) \quad \Rightarrow \quad i(E_1) = i(E_2).$$
\end{prs}

\begin{proof}
Let $p_i \in E_i$, where $i \in \{ 1,2 \}$, be two generic points. Since ${\rm bound}(E_1)  = {\rm bound}(E_2)$, there exists a path $P:[0,1] \to \RR^2$, such that $P(0)=p_1, P(1) = p_2$ and $P(t) \cap \A = \emptyset$ for $0 < t < 1$, otherwise any path $P$ starting at $p_1$ would have to pass to another 2-dimensional face in order to reach $p_2$, and thus ${\rm bound}(E_1)  \neq {\rm bound}(E_2)$ (see Figure \ref{exampleBound}).
Note that there may exist a path that connects the points $p_1$ and $p_2$, and intersects another 2-dimensional face; however this does not necessarily mean that ${\rm bound}(E_1)  \neq {\rm bound}(E_2)$.

\begin{figure}[h]
\epsfysize=4cm \centerline{\epsfbox{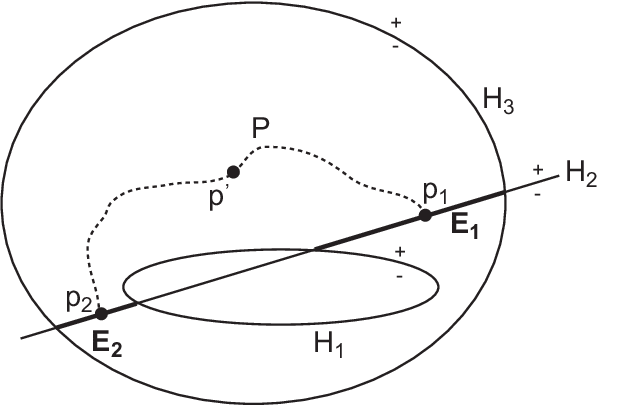}}
\caption{An example illustrating the proof of Proposition \ref{prsBound_I_connection}: we have $E_1 \sim E_2$. Hence, for any choice of points $p_1 \in E_1$ and $p_2 \in E_2$, we have: $i(p_1) = i(p_2) = (+,0,-)$, and for any point $p'$ on the path $P$ which connects $E_1$ and $E_2$, we have: $i(p') = (+,+,-)$.}\label{exampleBound}
\end{figure}

The point $P(0)=p_1 \in H_j$ is on a $1$-dimensional face $E_1 \subseteq H_j$ and hence only one entry in its vector of signs $i(p_1)$ is zero. Therefore,
if $i(p_1) = (a_1,\dots,a_{j-1},0,a_{j+1},\dots,a_n)$ where $0 \neq a_i \in L^1_2$ for $i \neq j$, then for every $0 < \varepsilon \ll 1$, $i(P(\varepsilon)) = (a_1,\dots,a_{j-1},a_j,a_{j+1},\dots,a_n)$ where $0 \neq a_j \in L^1_2$, since all the points which are in the neighbourhood of $p_1$ have the same mutual position with respect to the other components $H_i\in \A, i\neq j$, as $p_1$. By  Remark \ref{rem2.7}, we have: $i(P(\varepsilon)) = i(P(1-\varepsilon))$. Since $p_2 \in H_j$, then $(i(p_2))_j = 0$ and all the other entries of $i(p_2)$ are the same as $i(P(1-\varepsilon))$; hence $i(p_2) = i(p_1)$.
\end{proof}

The opposite direction  of Proposition \ref{prsBound_I_connection} is not correct, as can be seen in the CL arrangement presented in Figure \ref{exmplBoundSign}. Although the two faces $A,B$ have the same vector of signs: $i(A) = i(B) = (0,+,+)$, we have:
$$\{C_1,C_3\}={\rm bound}(A)  \neq {\rm bound}(B) = \{C_1,C_2\},$$
where $C_1,C_2,C_3$ are $2$-dimensional faces.

\begin{figure}[h]
\epsfysize=4cm \centerline{\epsfbox{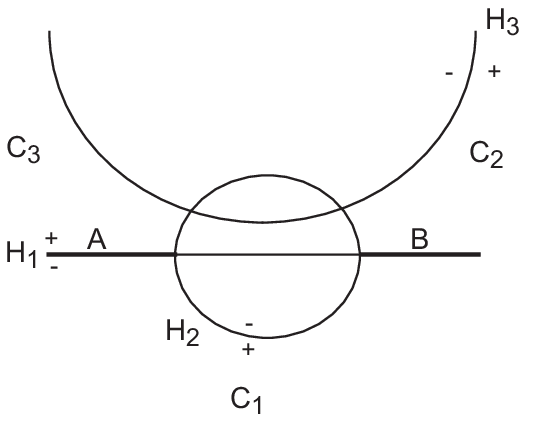}}
\caption{An example showing that $i(A) = i(B)$ does not imply ${\rm bound}(A)  = {\rm bound}(B)$.}\label{exmplBoundSign}
\end{figure}

\begin{remark}\label{remConnec}
{\rm For a CL arrangement $\A$, let  $E_1,E_2 \in \C(\A_H)$ be two different (1-dimensional open) faces in $H  \in \A$. The component $H$ divides the plane into two regions: $R_1$ and $R_2$. Assuming that $ E_1 \sim E_2$, then by definition: ${\rm bound}(E_1)  = {\rm bound}(E_2)$, which means that there exists {\it two} paths $P_i:[0,1] \to \RR^2$ for $i \in \{1,2\}$,
such that $P_1(0) = P_2(0) \in E_1 $ and $P_1(1) = P_2(1) \in E_2$ and for $0 < t <1$, $P_1(t) \cap \A = P_2(t) \cap \A = \emptyset$, but
$P_1(t) \subset R_1$ and $P_2(t) \subset R_2$ (see Figure \ref{exampleBound2}(a)).}
\end{remark}

\begin{figure}[h]
\epsfysize=5cm \centerline{\epsfbox{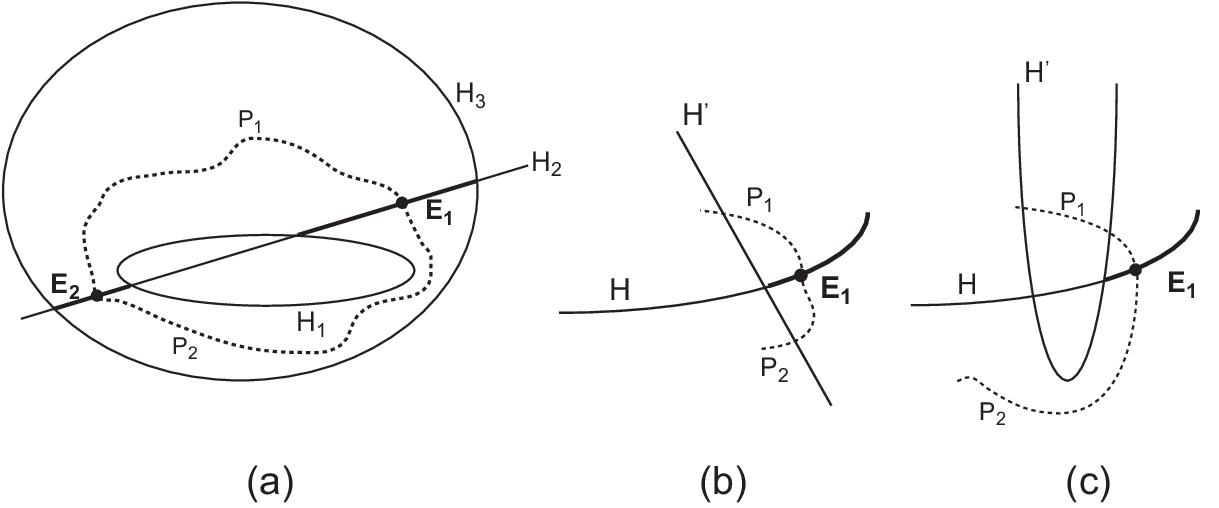}}
\caption{Part (a) is an example for Remark \ref{remConnec} regarding the equivalence of the sections $E_1$ and $E_2$. Parts (b) and (c) are examples for the proof of Proposition \ref{rem-1.3}(2).}\label{exampleBound2}
\end{figure}

\begin{prs}\label{rem-1.3}
Let $\A$ be a CL arrangement and $H \in \A$. Then:
\begin{enumerate}
\item $|b(H)| \leq |\C \left( \A_H \right)|$.
\item If $\A$ consists only of lines and parabolas, then $|b(H)| = |\C \left( \A_H \right)|$.
\end{enumerate}
\end{prs}

\begin{proof}
(1) Obvious.\\
(2) We have to show that for every two sections $E_1,E_2 \in \C(\A_H)$, we have: $E_1 \not\sim E_2$. Taking into account the notations of Remark \ref{remConnec}, we assume by contradiction that $E_1 \sim E_2$ and hence the two paths $P_1,P_2$ mentioned in Remark \ref{remConnec} exist. Then, at least one of the boundary points of the face $E_1$ (i.e. an intersection point of two components of $\A$) is also contained in another component $H'$, intersecting $H$ (see Figures \ref{exampleBound2}(b) and \ref{exampleBound2}(c)). Since $H'$ is either a line or a parabola, this means the following: if $H'$ is a line, every two optional paths $P_1$ and $P_2$ would have to intersect $H'$ if we wish to connect $E_1$ and $E_2$ (see Figure \ref{exampleBound2}(b)). If $H'$ is a parabola, there would exist only one path (e.g. $P_1$) that would connect $E_1$ and $E_2$ without intersecting $H'$, but the second path would have to intersect it (see Figure \ref{exampleBound2}(c)). Hence, in any case, for every $E_1,E_2 \in \C(\A_H)$, we have: $E_1 \not\sim E_2$, since crossing an unbounded component will necessarily change the chamber. So we have: $|b(H)| = |\C \left( \A_H \right)|$ for any component $H \in \A$.
\end{proof}

Here one can pose the following question:
\begin{question}
Is there a necessary and sufficient condition on a CL arrangement $\A$ and a component $H \in \A$ such that $|b(H)| = |\C(\A_H)|$?
\end{question}

Now, we are ready to present the generalized restriction-deletion chamber counting formula for CL arrangements:

\begin{prs}\label{prop-res-del-CL}
Let $H \in \A$ be a component of a real CL arrangement $\A$. Then:
$$|\C(\A)|=|\C \left( \A^H \right)|+|b(H)|.$$
\end{prs}

Note that by Proposition \ref{rem-1.3}(2), Proposition \ref{prop-res-del-CL} is indeed a natural generalization of the situation for line arrangements to the case of real CL arrangements. It also means that the original Zaslavsky's chamber counting formula holds for CL arrangements consisting of only lines and parabolas.

\begin{proof} [Proof of Proposition \ref{prop-res-del-CL}]
For every chamber $X \in \C \left( \A^H \right)$ satisfying $H \cap X \neq \emptyset$, $H$
divides $X$ into  a certain number of  chambers;  we denote this number by $k_X$. Thus:
$$|\C(\A)|=|\C \left( \A^H \right)|+\sum_{\substack{X \in \C ( \A^H ) \\ H \cap X \neq \emptyset}} (k_X-1),$$
\noindent
since every chamber $X \in \C \left( \A^H \right)$ in the sum splits into
$k_X$ chambers, but we do not count $X$ itself, as it is already
counted in $\left| \C \left( \A^H \right) \right|$. For each $X \in \C \left( \A^H \right)$  in the sum,
denote:
$$X = \bigcup_{i=1}^{k_X} X_i,\, H_X = H \cap X,$$
that is, (the interior of) $X$ is divided into $k_X$ chambers $X_i$, whose union (of their closure) is (the closure of) $X$.

Note that $H_X$ is possibly a  union of disjoint sections, each of which is an element of $\C \left( \A_H \right)$. Therefore, we need to prove that $1+\left| b \left( H_X \right) \right|= k_X$. We numerate the sections of $H_X$ consecutively, which induces a numeration $E_1,E_2,\ldots$ of the {\em different} sections of $b \left( H_X \right)$ from right to left (note that $E_i$, $i \geq 1$, is an {\em equivalence class} of elements in $H_X$). For each $E_i \in b \left( H_X \right)$, $1 \leq i \leq \left| b \left( H_X \right) \right|$, we look at the pair $\bound(E_i) = \{X_{i_1}, X_{i_2}\}$; see Figure \ref{sectionOfH} for an example.

\begin{figure}[h]
\epsfysize=3.5cm \centerline{\epsfbox{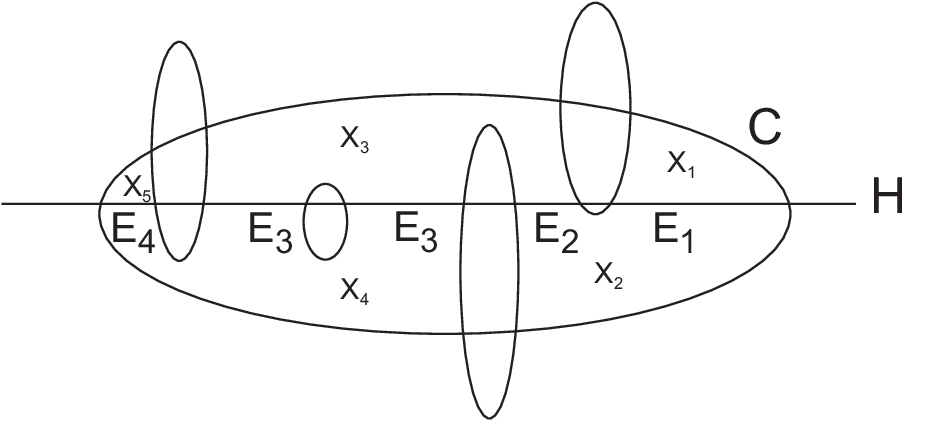}}
\caption{An example for the partition of $H_X$, where $X = X_1 \cup \cdots \cup X_5$ is a chamber bounded by the conic $C$ and $X_1,\ldots,X_5$ are the chambers whose union of their closure is the closure of $X$. The equivalence classes of the sections of $H_X = H \cap X$ are $E_1,\dots,E_4$.}\label{sectionOfH}
\end{figure}

We claim that for all $i$, we have $|\bound(E_i) \cap \bound(E_{i+1})|=1$. Indeed,
$$|\bound(E_i) \cap \bound(E_{i+1})|< 2,$$
otherwise $E_i = E_{i+1}$ by the definition of the equivalence relation.
Assume by contradiction that $|\bound(E_i) \cap \bound(E_{i+1})|=0$. This means that we have the  situation depicted in Figure \ref{FourChambers}.

\begin{figure}[h]
\epsfysize=2cm \centerline{\epsfbox{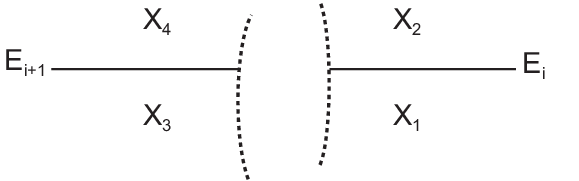}}
\caption{An illustration for the situation $|\bound(E_i) \cap \bound(E_{i+1})|=0$}\label{FourChambers}
\end{figure}

However, this situation is impossible, since these sections of $H_X$ are consecutive, and if $\{X_1,X_2\} \cap \{X_3,X_4\} = \emptyset$, then $E_i$ and $E_{i+1}$
will not dissect the same (single) chamber $X \in \C \left( \A^H \right)$ (since $X_1,X_2$ and $X_3,X_4$ will be contained in different chambers of $\C \left( \A^H \right)$) -- indeed, even before the equivalence relation $\sim$, one can connect a generic point  from $E_i$ with a generic point from $E_{i+1}$ by a continuous path which lays only in $X$, which mean that the above intersection is always non-empty (see also Remark \ref{rem2.7}).

Thus, we define recursively the  map
$\ell_X : b(H_X) \to \{X_1,\dots,X_{k_X}\}$ as follows:
Let $\bound(E_1) = \{X_1,X_2\}$. Assume, after possible renumeration, that $X_2 \in \bound(E_2)$ (as was proved above, it is not possible that either both $X_1$ and $X_2$ belong to $\bound(E_2)$ or both $X_1$ and $X_2$ do not belong to $\bound(E_2)$). Hence, we define  $\ell_X(E_1) = X_2$.
For $i>1$, define $\ell_X(E_i)$ to be the chamber $X'$ satisfying that $X' \in \bound(E_i)$ and for every $j<i$, $\ell_X(E_j) \neq X'$. This means that if $Y$ is a chamber such that there exists $k$ such that $Y \in \bound(E_k) \cap \bound(E_{k+1})$, then $Y \in {\rm Image}(\ell_X)$.
Note that for every $i>1$ there is only one way to choose $\ell_X(E_i)$ (recall that
$\left| \bound(E_i) \ \cap \ \bound(E_{i+1}) \right| =1$
for each $i$). By its definition, the map  $\ell_X$ is injective, hence $|b(H_X)| \leq k_X-1$ (since $X_1 \not\in {\rm Image}(\ell_X)$). Note that in CL arrangements it is not possible that $X_1$ will appear again as a boundary of a non-consecutive section of $H$. This might occur if one deals with curves of degree $4$ or higher.

For proving that $|b(H_X)| = k_X-1$, assume on the contrary that $|b(H_X)| < k_X-1$. Then there are at least two chambers in $\{X_1,\ldots, X_{k_X}\}$ which are not in the image of $\ell_X$. One of them is $X_1$ (by the construction above). Denote the other chamber by $X_j $, where $1 < j \leq k_X$.

Hence, there exists $t$ such that $1 < t \leq |b(H_X)|$ such that $X_j \in \bound(E_t)$. By the definition of bound and the discussion above, either $X_j \in \bound(E_{t-1}) \cap \bound(E_t), $ or $X_j \in \bound(E_t) \cap \bound(E_{t+1})$, or $X_j \in \bound \left( E_{|b(H_X)|} \right)$. In the first two cases, $X_j \in {\rm Image}(\ell_X)$ by the definition of $\ell_X$, which is a contradiction. In the last case, denote: $\bound\left(E_{|b(H_X)|}\right) = \{Y',Y''\}$. By the construction, one can assume that
$$Y' \in \bound \left(E_{|b(H_X)|-1}\right) \cap \bound\left(E_{|b(H_X)|}\right)  .$$
Hence, by the definition of $\ell_X$, $\ell_X\left(E_{|b(H_X)|-1}\right) = Y'$, and so we have: $Y' \in {\rm Image}(\ell_X)$ which implies that: $X_j = Y''$, and then necessarily by the definition of $\ell_X$, we have that: $$\ell_X\left(E_{|b(H_X)|}\right) = Y'' =X_j,$$
which is a contradiction. The contradiction implies that our assumption $|b(H_X)| < k_X-1$ was false, and hence $|b(H_X)|= k_X-1$.
\end{proof}

\begin{remark}
\rm{
A set-theoretic proof for Proposition \ref{prop-res-del-CL} can be given, which is a generalization of the one given in Remark \ref{remAlgProofZasHyperplane} for line arrangements. Still this proof is based on arguments given in the previous proof of Proposition \ref{prop-res-del-CL}.

As in the case of line arrangements, deleting a component $H$ from the arrangement $\A$ induces a surjection of LRBs  $f:\LL_0(\A) \to \LL_0(\A^H)$, which deletes the sign corresponding to the component $H$. Thus, the number of chambers of $\A$ is equal to the sum of the number of chambers of $\A^H$ plus the number of chambers which are identified by the map $f$. We have shown in the proof of Proposition \ref{prop-res-del-CL} that the map $\ell_X:b(H) \to \{X_2,\ldots,X_{k_X}\}$ is a bijection between the number of chambers which are identified by the map $f$ and the elements in $b(H)$, and therefore the result follows as in the proof presented in Remark \ref{remAlgProofZasHyperplane}.
}
\end{remark}

We conclude this section by several remarks.

\begin{remark}
{\rm By the same arguments we have used above, one can easily see that Proposition \ref{prop-res-del-CL} holds for CL arrangements in $\RR\pp^2$ too. However, the definition of the function ${\rm bound}$ in
Definition \ref{defBound}(1) should be changed as follows:
$${\rm bound} : \C(\A_H) \to \{Y \in P(\C(A)) : |Y| \leq 2 \}$$
$${\rm bound}(E) = \{X_1,X_2\} \text{ such that } E \subset \overline{X_1} \cap \overline{X_2} \text{ or } E \subset \overline{X_1}.$$}
\end{remark}

\begin{remark}
\rm{
Zaslavsky \cite[Theorem 7.1]{Z2} offers a different way for computing the number of chambers induced by a CL arrangement $\A$.
However, this formula computes the number of chambers directly and does not give a recursion as in Proposition \ref{prop-res-del-CL} (which is similar to his original chamber counting formula for line arrangements). Indeed, computing $|\C(\A)| - |\C(\A^H)|$ from Zaslavsky's formula for a CL arrangement will depend on whether $H$ is a line, a parabola or an ellipse. For example, if $H$ is a line or a parabola, then we get that: $$|\C(\A)| - |\C(\A^H)| = 1 + |\{v_j \in H\}| = |\C(\A_H)|,$$ where the set $\{v_j\}$ is the set of intersection points (in $L(\A)$) on $H$. Note that the rightmost equality is due to the fact that $H$ is a line or a parabola (by Proposition \ref{rem-1.3}(2)).
}
\end{remark}

\begin{remark}\label{res-del-rem}
\rm{
There are several other restriction-deletion theorems with respect to CL arrangements in other contexts, which might be connected to the generalization of Zaslavsky's restriction-deletion chamber counting formula. We survey some of them here.

Schenck and Tohaneanu \cite{ST} proved the existence of other restriction-deletion theorems with respect to the module of $\A$-derivations $D(\A)$ for a CL arrangement $\A$ (for its definition, see the Appendix in Section \ref{sec-app2} below). However, the connection between these theorems and the results we have obtained with respect to a deleted or restricted CL arrangement is not clear.

\medskip

First, the restriction-deletion theorems in \cite[Theorem 2.5  and Theorem 3.4]{ST} for the module of $\A$-derivations $D(\A)$  can be applied only for free quasihomogeneous (see \cite[Definition 1.6]{ST}) triples $\left( \A^H, \A, \A_H \right)$ (where $H \in \A$; note that line arrangements are always quasihomogeneous). However, the restriction-deletion proposition for chamber counting (see Proposition \ref{prop-res-del-CL}) works for all CL arrangements, and the restriction-deletion proposition for $\LL_0(\A)$ (see Proposition \ref{prsEmbed}) can be applied only when $\LL_0(\A)$ is an LRB.

Second, for deleting a component $H$, the restriction-deletion chamber counting  formula (Proposition \ref{prop-res-del-CL}) depends on the number of 1-dimensional faces in $\LL(\A)$ on this component which are equivalent under $\sim$, an equivalence relation
which does not appear in the restriction-deletion theorem
for $D(\A)$ for deleting a component (see \cite[Theorem 2.5]{ST}).

Moreover, while for line arrangements, the connection between these theorems is obvious, for CL arrangements the connection is more subtle. For a \emph{free} line arrangement $\LL$ (see its definition in the Appendix in Section \ref{sec-app2} below), the chamber counting formula can be induced by the restriction-deletion theorem with respect to $D(\LL)$: indeed, the addition-deletion formula for $D(\LL)$ implies the restriction-deletion formula for the characteristic
polynomial $\pi(\LL,t)$ \cite[Theorem 4.61]{OT} and since $\pi(\LL,1) = |{\mathcal C}(\LL)|$, the restriction-deletion formula for chamber counting follows. However, for a free quasihomogeneous CL arrangement $\A$, the connections between the different restriction-deletion theorems (for $D(\A)$, for $\pi(\A,t)$ and for $\C(\A)$) are not clear; for example, $\pi(\A,1)\neq |\C(\A)|$ even for the CL arrangement $\A$ which consists of a line intersecting transversally a conic. We leave this for further investigation.

Moreover, note that while the characteristic polynomial is combinatorially determined (for any arrangement of curves in $\CC^2$), the module of $\A$-derivations $D(\A)$ for a CL arrangement $\A$ is not: in \cite{ST}, a pair of combinatorially-equivalent CL arrangements having different modules of $\A$-derivations is presented.
}
\end{remark}

\section{Appendix}\label{secApp}
In this appendix, we introduce some of the notions related to hyperplane arrangements which were mentioned in the paper. All the material presented in this section can be found in \cite{OT}.

\subsection{The Poincar\'{e} polynomial and Zaslavsky's restriction-deletion theorem}\label{sec-app1}

Let $\A$ be a real hyperplane arrangement in $V=\RR^n$ and let $L=L(\A)$ be its intersection lattice. Define the {\em M\"obius function} $\mu_\A=\mu: L \times L \to \ZZ$ as follows:
$$\left\{ \begin{array}{rcl}
\vspace{5pt}\mu(X,X)=1 & & \mbox{if } X \in L, \\
\vspace{5pt}\sum\limits_{X \leq Z \leq Y} \mu(X,Z)=0 & & \mbox{if } X,Y,Z \in L \mbox{ and } X<Y, \\
\mu(X,Y)=0 & & \mbox{otherwise.}
\end{array} \right.$$

For $X \in L$, define $\mu(X)=\mu(V,X)$.

\medskip

Now we can define the {\em Poincar\'{e} polynomial}:
\begin{definition}
Let $\A$ be a real hyperplane arrangement in $\RR^n$ with intersection lattice $L$ and M\"obius function $\mu$. Let $t$ be an indeterminate. Define the {\em Poincar\'{e} polynomial} of $\A$ as follows:
$$\pi(\A,t)=\sum\limits_{X \in L} \mu(X)(-t)^{r(X)}$$
where $r(X)={\rm codim}(X)$ is the rank function.
\end{definition}

Zaslavsky's restriction-deletion theorem \cite{Z} claims the following connection between the Poincar\'{e} polynomials of the arrangement, the deleted arrangement and the restricted arrangement:
\begin{equation}\label{eqnPoinc}
\pi (\A , t) = \pi (\A^H , t)+t \cdot \pi (\A_H , t)
\end{equation}
for any hyperplane $H \in \A$.

\subsection{Module of derivations and free arrangements}\label{sec-app2}
Let $\A$ be a real hyperplane arrangement in $V=\RR^n$.
Define $S=S(V^*)$ to be the symmetric algebra of the dual space $V^*$ of $V$. Note that if $x_1,\dots, x_n$ is the basis for $V^*$, then $S \cong \RR [x_1, \dots, x_n]$.

A {\it derivation} of $S$ over $\RR$ is a linear map $\theta: S \to S$ over $\RR$ satisfying:
$\theta (fg)=f \theta (g)+ g \theta (f)$ for all $f,g \in S$. The set of derivations of $S$ over $\RR$ is denoted by ${\rm Der}_{\RR}(S)$.
One can easily see that ${\rm Der}_{\RR}(S)$ is a free $S$-module of rank $n$, where its basis is
$\{ D_1,\dots, D_n \}$ where $D_i(f)=\frac{\partial f}{\partial x_i}$ (for references, see \cite[Chapter 4]{OT}).

For any $f \in S$, define:
$$D(f)= \{ \theta \in {\rm Der}_{\RR}(S) \ | \ \theta(f) \in fS \}.$$
Note that $D(f)$ is an $S$-submodule of ${\rm Der}_{\RR}(S)$.

\medskip

In this setting, we can define the {\em module of $\A$-derivations} and a {\em free arrangement}:
\begin{definition}
Let $\A$ be a real hyperplane arrangement in $V=\RR^n$ with defining polynomial
$Q(\A)=\prod\limits_{H \in \A} \alpha_H$ where $H={\rm ker} (\alpha_H)$. The {\em module of $\A$-derivations} is: $D(\A)=D(Q(\A))$.
\end{definition}

\begin{definition}
An arrangement $\A$ is called {\em free} if $D(\A)$ is a free module over $S$.
\end{definition}

Let $S_p$ be the subspace of $S \cong \RR [x_1, \dots, x_n]$ consisting of $0$ and the homogeneous polynomials of degree $p$ for $p \geq 0$. For $p<0$, define $S_p=0$. Then: $S=\bigoplus_{p \in \ZZ} S_p$ is a graded $\RR$-algebra. A nonzero element $\theta  \in  {\rm Der}_{\RR}(S)$ is {\em homogeneous of polynomial degree $p$} if $\theta =\sum\limits_{k=1}^n f_k D_k$ and $f_k \in S_p$ for $1 \leq k \leq n$. In this case, we write
${\rm pdeg}(\theta)=p$.

\medskip

Here we define the notion of {\em exponents} of a free arrangement $\A$:
\begin{definition}
Let $\A$ be a free arrangement and let $\{ \theta_1,\dots, \theta_n \}$ be a homogeneous basis for $D(\A)$. We call ${\rm pdeg}(\theta_1),\dots, {\rm pdeg}(\theta_n)$ the {\em exponents} of $\A$ and write:
$${\rm exp} (\A) = \{ {\rm pdeg}(\theta_1),\dots, {\rm pdeg}(\theta_n) \}.$$
\end{definition}

Note that ${\rm exp} (\A)$ may have repetitions and its order should be neglected.

The {\it addition-deletion theorem} (\cite{Terao}, see also \cite[Theorem 4.51]{OT}) states the following connection between the freeness properties of $\A$, $\A^H$ and $\A_H$:
\begin{thm}\label{thmModDer}
Suppose that $\A$ is a non-empty hyperplane arrangement and let $H \in \A$. Any two of the following statements imply the third:
\begin{enumerate}
\item $\A$ is free with ${\rm exp} (\A) =\{ b_1, \dots, b_{n-1},b_n \}$.
\item $\A^H$ is free with ${\rm exp} (\A^H) =\{ b_1, \dots, b_{n-1},b_n-1 \}$.
\item $\A_H$ is free with ${\rm exp} (\A_H) =\{ b_1, \dots, b_{n-1} \}$.
\end{enumerate}

\end{thm}
Note that Equation (\ref{eqnPoinc}) in Section \ref{sec-app1} can be induced by Theorem \ref{thmModDer} (see \cite[Chapter 4]{OT}).


\begin{thebibliography}{99}

\bibitem{AB} P. Abramenko and K. S. Brown, {\em Buildings: theory and applications},   Graduate Texts in Math. \textbf{248}, Springer, New York, 2008.

\bibitem{AM} M. Aguiar and S. Mahajan, {\em Coxeter groups and Hopf algebras}, Fields Institute Monographs  \textbf{23}, AMS, Providence, RI, 2006.

\bibitem{AT3} M. Amram, D. Garber and M. Teicher, {\em Fundamental groups of tangented conic-line arrangements with singularities up to order 6}, Math. Zeit. {\bf 256} (2007), 837--870.

\bibitem{ATY} M. Amram, M. Teicher and F. Ye, {\em Moduli spaces of arrangements of 10 projective lines with quadruple points}, Adv. App. Math. {\bf 51}(3) (2013), 392--418.

\bibitem{AT1} M. Amram, M. Teicher and  A. M. Uludag, {\em Fundamental groups of some quadric-line arrangements}, Topology Appl. {\bf 130}(2) (2003), 159--173.

\bibitem{ABGBVS} E. Artal-Bartolo, B. Guerville-Ballé and J. Viu-Sos, {\it Fundamental groups of real arrangements and torsion in the lower central series quotients}, Experimental Math., to appear (2018).

\bibitem{BCL}  G.R. Barnes, P.B. Cerrito and I. Levi, {\it Random walks on finite semigroups}, J. Appl. Probab. {\bf 35}(4) (1998), 824--832.

\bibitem{Bj} A. Bj\"{o}rner, {\em Random walks, arrangements, cell complexes, greedoids, and selforganizing libraries}, in: \emph{Building bridges}, Bolyai Soc. Math. Stud. {\bf 19},
165--203, Springer, Berlin, 2008.

\bibitem{matroids}  A. Bj\"{o}rner, M. Las Vergnas, B. Sturmfels, N. White and G. Ziegler, {\em Oriented matroids}, Second edition, Encyclopedia of Mathematics and its Applications {\bf 46}, Cambridge University Press, Cambridge, 1999.

\bibitem{Bj2} A. Bj\"{o}rner and G. Ziegler, {\em Combinatorial stratification of complex arrangements}, J.
Amer. Math. Soc. {\bf 5}(1) (1992), 105--149.

\bibitem{B} K.S. Brown, {\em Semigroups, rings, and Markov chains}, J. Theoret. Probab. {\bf 13}(3) (2000), 871--938.

\bibitem{B2} K.S. Brown, {\em Semigroup and ring theoretical methods in probability}, in: {\em Representations of finite dimensional algebras and related topics in Lie theory and geometry}, 3-–26, Fields Inst. Commun. {\bf 40}, Amer. Math. Soc., Providence, RI, 2004.

\bibitem{BD}  K.S. Brown and P. Diaconis, {\em Random walks and hyperplane arrangements}, Ann. Probab.
{\bf 26}(4) (1998), 1813--1854.

\bibitem{Co} P. Corsini, {\it Prolegomena of hypergroup theory}, Second edition, Aviani editor,
1993.

\bibitem{CoLe} P. Corsini and V. Leoreanu, {\it Applications of hyperstructure theory}, in: {\it Advances in Mathematics}, Kluwer Academic Publishers, Dordrecht, 2003.

\bibitem{EGT1} M. Eliyahu, D. Garber and M. Teicher, {\it A conjugation-free geometric presentation of fundamental groups of arrangements}, Manuscripta Math. {\bf 133}(1--2) (2010), 247--271.

\bibitem{EGT2} M. Eliyahu, D. Garber and M. Teicher, {\it A conjugation-free geometric presentation of fundamental groups of arrangements II: Expansion and some
    properties}, Int. J. Alg. Comput. {\bf 21}(5) (2011), 775--792.

\bibitem{FG} M. Friedman and D. Garber, {\em On the structure of conjugation-free fundamental groups of  conic--line arrangements},
J. Homotopy Relat. Struct. {\bf 10}(4) (2015), 685--734.

\bibitem{FG2} M. Friedman and D. Garber, {\em On the structure of fundamental groups of conic--line arrangements having a cycle in their graph}, Topology Appl. {\bf 177} (2014), 34--58.

\bibitem{Ful} W. Fulton, {\em Intersection theory}, Springer-Verlag, 1997.

\bibitem{GTV} D. Garber, M. Teicher and U. Vishne, {\it Classes of wiring diagrams and their invariants}, J. Knot Theory Ramifications {\bf 11}(8) (2002), 1165--1191.

\bibitem{GBVS} B. Guerville-Ballé and J. Viu-Sos, {\it Configurations of points and topology of real line arrangements}, Math. Ann., to appear (2018).

\bibitem{H} F. Hirzebruch, {\em Arrangements of lines and algebraic surfaces}, in: {\it Arithmetic and geometry}, Vol. II, 113--140, Progr. Math. {\bf 36}, Birkhauser, Boston, Mass., 1983.

\bibitem{MSS} S. Margolis, F.V. Saliola and B. Steinberg, {\em Combinatorial topology and the global dimension of algebras arising in combinatorics}, J. Europ. Math. Soc. {\bf 17}(12) (2015), 3037--3080.

\bibitem{massey} W.S. Massey, {\it Algebraic Topology: An Introduction}, Graduate Texts in Math. {\bf 56}, Springer-Verlag, 1967.

\bibitem{OT} P. Orlik and H. Terao, {\em Arrangements of Hyperplanes}, Grundlehren der Mathematischen Wissenschaften \textbf{300}, Springer--Verlag, 1992.

\bibitem{Sal1} F.V. Saliola, {\it The quiver of the semigroup algebra of a left regular band}, Internat. J. Alg. Comput. {\bf 17}(8) (2007), 1593--1610.

\bibitem{Sal2} F.V. Saliola, {\em The face semigroup algebra of a hyperplane arrangement}, Canad. J. Math. \textbf{61}(4) (2009), 904--929.

\bibitem{Sch} R. D. Schafer, {\em An Introduction to Nonassociative Algebras}, Academic Press, New-York, London, 1966.

\bibitem{ST} H. Schenck and S.O. Tohaneanu, {\it Freeness of conic-line arrangements in $\mathbb{P}^2$},
Comment. Math. Helv. {\bf 84} (2009), 235--258.

\bibitem{St} R. P. Stanley, {\em Enumerative Combinatorics}, Vol. \textbf{1}, Cambridge Studies in
Adv. Math. \textbf{49}, 1997.

\bibitem{Terao} H. Terao, {\it Arrangements of hyperplanes and their freeness I, II}. J. Fac. Sci. Univ. Tokyo Sect. 1 A, Math. {\bf 27} (1980), 293--320.

\bibitem{T} J. Tits, {\em Buildings of spherical type and finite BN-pairs}, Lecture Notes in Math. \textbf{386}, Springer-Verlag, Berlin, 1974.

\bibitem{Tok}
H. Tokunaga, {\it Sections of elliptic surfaces and Zariski pairs for conic-line arrangements via dihedral covers}, J. Math. Soc. Japan {\bf 66}(2) (2014), 613--640.

\bibitem{Voug} T. Vougiouklis, {\it Hyperstructures and their representations}, Hadronic Press,
Florida, 1994.

\bibitem{WY} S. Wang and S.S.-T. Yau, {\it Rigidity of differentiable structure for new class of line arrangements}, Comm. Anal. Geom. {\bf 13}(5) (2005), 1057--1075.

\bibitem{Ye} F. Ye,  {\it Classification of moduli spaces of arrangements of 9 projective lines}, Pacific J. Math. {\bf 265}(1) (2013), 243--256.

\bibitem{Z} T. Zaslavsky, {\em Facing up to arrangements: face-count formulas for partitions of space by hyperplanes}, Mem. Amer. Math. Soc. {\bf 1} (Issue 1, 154), 1975, vii--102.

\bibitem{Z2} T. Zaslavsky, {\it A combinatorial analysis of topological dissections}, Adv. Math. {\bf 25}(3) (1977), 267--285.

\end{thebibliography}
\end{document}